\documentclass{article}

\usepackage{amsmath,amssymb,amsthm,mathrsfs}
\usepackage{hyperref,xcolor}
\usepackage{changepage}
\usepackage{enumitem} 
\usepackage{graphicx}
\usepackage{tikz}
\usepackage[utf8]{inputenc}
\usepackage[T1]{fontenc}
\usepackage{hyperref}
\usepackage{comment}
\allowdisplaybreaks

\numberwithin{equation}{section}

\newtheorem{myDefn}{Definition}[section]

\newtheorem{myProp}[myDefn]{Proposition}

\newtheorem{myRem}[myDefn]{Remark}

\newtheorem{myLem}[myDefn]{Lemma}

\newtheorem{myTheorem}[myDefn]{Theorem}

\DeclareMathOperator*{\argmin}{argmin}
\DeclareMathOperator*{\minn}{minimize}

\def\eps{\varepsilon}

\def\nn{\mathrm{n}}

\usepackage{indentfirst}

\def\R{\mathbb{R}}
\def\N{\mathbb{N}}

\def\HH{\mathrm{H}}
\def\LL{\mathrm{L}}

\def\MD{\mathcal{D}}

\newcommand{\fonction}[5]{\begin{array}[t]{lrcl}#1 :&#2 &\longrightarrow &#3\\&#4& \longmapsto &#5 \end{array}}

\newcommand{\dual}[2]{\left\langle #1 , #2 \right\rangle}
\newcommand{\hookdoubleheadrightarrow}{%
  \hookrightarrow\mathrel{\mspace{-15mu}}\rightarrow
}

\newlist{primenumerate}{enumerate}{1}
\setlist[primenumerate,1]{label={\roman*$'$}}

\title{Sensitivity analysis and optimal control for a friction problem in the linear elastic model}

\author{Lo\"ic Bourdin\footnote{Institut de recherche XLIM. UMR CNRS 7252. Universit\'e de Limoges, France. \texttt{loic.bourdin@unilim.fr}}, 
Fabien Caubet\footnote{Universit\'e de Pau et des Pays de l'Adour, E2S UPPA, CNRS, LMAP, UMR 5142, 64000 Pau, France. \texttt{fabien.caubet@univ-pau.fr}}, Aymeric Jacob de Cordemoy\footnote{Sorbonne Université, CNRS, Université Paris Cité, Laboratoire Jacques-Louis Lions (LJLL),
F-75005 Paris, France. \texttt{aymeric.jacob\_de\_cordemoy@sorbonne-universite.fr}}
}

\begin{document}

\maketitle

\begin{abstract}
    This paper investigates, without any regularization procedure, the sensitivity analysis of a mechanical friction problem involving the (nonsmooth) Tresca friction law in the linear elastic model. To this aim a recent methodology based on advanced tools from convex and variational analyses is used. Precisely we express the solution to the so-called Tresca friction problem thanks to the proximal operator associated with the corresponding Tresca friction functional. Then, using an extended version of twice epi-differentiability, we prove the differentiability of the solution to the parameterized Tresca friction problem, characterizing its derivative as the solution to a boundary value problem involving tangential Signorini's unilateral conditions. Finally our result is used to investigate and numerically solve an optimal control problem associated with the Tresca friction model.
\end{abstract}


\medskip



\section{Introduction}

\paragraph{General context and motivation} On the one hand, \textit{optimal control theory} is the mathematical field aimed at finding the control of a given system that allows to minimize a given cost while satisfying given constraints. In order to numerically solve an optimal control problem, the numerical descent methods usually require to compute the gradient of the cost functional which usually depends on the solution to a partial differential equation with given boundary conditions. Therefore a crucial point is to perform the {\it sensitivity analysis} of the solution to the boundary value problem with respect to perturbations, in order to characterize its derivative.

On the other hand, \textit{solid mechanics} is the scientific field that studies the deformation of solids. A classical mechanical setting consists in a deformable body which is in contact with a rigid foundation, possibly sliding against it which causes friction on the contact surface. This friction can be mathematically modeled by the so-called {\it Tresca friction law} (see, e.g.,~\cite{KUSS}) which appears as a boundary condition involving nonsmooth inequalities depending on a friction threshold. Mechanical problems with friction  are usually investigated through the theory of {\it variational inequalities}, and the Tresca friction law causes nonlinearities and nonsmoothness in the corresponding variational formulations. 

As a consequence, in order to investigate optimal control problems with mechanical models involving the Tresca friction law, we have to perform the sensitivity analysis of nonsmooth variational inequalities. The standard methods found in the literature usually consist in regularization (see, e.g.,~\cite{MAUALLJOU,CHAUDET,CHAUDET2}, or~\cite[Section 10.4 Chapter 10]{KIKU}) and dualization (see, e.g.,~\cite{SOKOW88,SOKOZOLE2}) procedures. In a nutschell, regularization consists in replacing the nondifferentiable term by its Moreau's envelope to approximate the optimization problem associated with the model, thus the corresponding optimality condition is replaced by a smooth variational equality instead of a nonsmooth variational inequality. However this method does not take into account the exact characterization of the solution and perturbs the nonsmooth nature of the original physical model. The dualization method consists in describing the primal/dual pair of the model as a saddle point of the associated Lagrangian. The dual model leads to a characterization of its solution that involves only projection operators and thus Mignot's theorem (see~\cite{MIGNOT}) about conical differentiability can be applied. With this method, the derivative of the solution to the primal model, with respect to perturbations, can be obtained but is characterized only implicitly, due to the presence of dual elements.

In this paper the sensitivity analysis is performed using a recent methodology based on advanced tools from convex and variational analyses such as the notion of \textit{proximal operator} introduced by J.J.\ Moreau in~1965 (see~\cite{MOR}) and the notion of \textit{twice epi-differentiability} introduced by R.T.\ Rockafellar in~1985 (see~\cite{Rockafellar}). This methodology allows us to preserve the original nonsmooth nature of the model, that is, without using any regularization procedure, and to work only with the primal model.

\paragraph{Objective and methodology}
The present work follows from our previous papers~\cite{4ABC,BCJDC} in which the sensitivity analysis of boundary value problems involving the \textit{scalar version} of the Tresca friction law are performed. In this new paper we focus on the classical Tresca friction law which is about the linear elastic model. Precisely we consider the (parameterized) \textit{Tresca friction problem} given by
\begin{equation}\tag{TP\ensuremath{_{t}}}\label{PbNeumannDirichletTrescaParassssss2}\arraycolsep=2pt
\left\{
\begin{array}{rcll}
-\mathrm{div}(\mathrm{A}\mathrm{e}(u_{t})) & = & f_{t}   & \text{ in } \Omega , \\
u_{t} & = & 0  & \text{ on } \Gamma_{\mathrm{D}} ,\\
\sigma_\nn(u_{t}) & = & h_{t}  & \text{ on } \Gamma_{\mathrm{N}},\\
\left\|\sigma_\tau(u_{t})\right\|\leq g_{t} \text{ and } u_{t_\tau}\cdot\sigma_{\tau}(u_{t})+g_t\left\|u_{t_\tau}\right\| & = & 0  & \text{ on } \Gamma_{\mathrm{N}},
\end{array}
\right.
\end{equation} 
for all $t\geq0$, where $\Omega \subset \R^{d}$ is a nonempty bounded connected open subset of $\R^{d}$, with~$d\in\{2,3\}$ and with a $\mathcal{C}^{1}$-boundary denoted by $\Gamma:=\partial\Omega$, where $\nn$ is the outward-pointing unit normal vector to $\Gamma$ and where the boundary is decomposed as~$\Gamma=:\Gamma_{\mathrm{D}}\cup\Gamma_{\mathrm{N}}$, where $\Gamma_{\mathrm{D}}$ and $\Gamma_{\mathrm{N}}$ are two measurable (with positive measure) pairwise disjoint subsets of~$\Gamma$ such that almost every point of~$\Gamma_{\mathrm{N}}$ belongs to $\mathrm{int}_{\Gamma}({\Gamma_{\mathrm{N}}})$. Recall that, in linear elasticity,~$\mathrm{A}$ is the stiffness tensor,~$\mathrm{e}$ is the infinitesimal strain tensor,~$\sigma_{\nn}$ is the normal stress and~$\sigma_{\tau}$ is the shear stress (see Section~\ref{Mainresult1} for details). Moreover~$\left\|\cdot\right\|$ stands for the usual Euclidean norm of~$\R^d$ and we assume that $f_{t}\in\LL^{2}(\Omega,\R^d)$, $h_{t}\in\LL^{2}(\Gamma_{\mathrm{N}})$ and~$g_{t}\in\LL^{2}(\Gamma_{\mathrm{N}})$, with $g_{t}>0$ almost everywhere on $\Gamma_{\mathrm{N}}$, for all $t\geq0$. Finally we recall that the tangential boundary condition on~$\Gamma_{\mathrm{N}}$ is known as the Tresca friction law. 

\begin{myRem}\label{RemCondBordNeumann1}
In this paper we consider a model with a prescribed normal stress and a Tresca condition on~$\Gamma_{\mathrm{N}}$ (as studied for example in~\cite[Section 5.2 Chapter III]{DUVAUTLIONS}). This covers the particular case of zero normal stress (taking $h_t=0$) which corresponds to the case of no tensile or compressive stress. 
Nevertheless it would be possible to consider a contact problem by constraining the normal displacement, that is, by replacing $\sigma_{\nn}(u_t)=h_t$ by ${u_t}_{\nn}=0$ on~$\Gamma_{\mathrm{N}}$. This case corresponds to a bilateral contact (see, e.g.,~\cite[Remark~2.1]{CHOERNPIG}). This would not add any difficulty and could be dealt in the same way as the one presented in this paper (see Remark~\ref{remarextensigno} for details).
\end{myRem}

The main objective of this work is to characterize the derivative of the map $t\in\mathbb{R}_{+}\mapsto u_{t}\in\HH^{1}_{\mathrm{D}}(\Omega,\R^d)$ at $t=0$, where $\HH^{1}_{\mathrm{D}}(\Omega,\R^d):=\{ w\in\HH^{1}(\Omega,\R^d) \mid w=0 \text{ \textit{a.e.} on } \Gamma_{\mathrm{D}}\}$ and where the abbreviation \textit{a.e.} stands for \textit{almost everywhere}.
However the norm $\left\|\cdot\right\|$ which appears in the Tresca friction law generates nonsmooth terms in the variational formulation of Problem~\eqref{PbNeumannDirichletTrescaParassssss2} given by: find~$u_{t}\in\HH^{1}_{\mathrm{D}}(\Omega,\R^d)$ such that
\begin{multline*}
\displaystyle\int_{\Omega}\mathrm{A}\mathrm{e}(u_t):\mathrm{e}(w-u_t)+\int_{\Gamma_{\mathrm{N}}}g_t\left\|w_\tau\right\|-\int_{\Gamma_{\mathrm{N}}}g_t\left\|{u_t}_\tau\right\| \geq\int_{\Omega}f_t\cdot\left(w-u_t\right)\\+\int_{\Gamma_{\mathrm{N}}}h_t\left(w_\nn-{u_t}_\nn\right), \qquad \forall w\in\HH^{1}_{\mathrm{D}}(\Omega,\R^d),
\end{multline*}
for all~$t \geq 0$. Nevertheless recall that $u_t$ can be expressed, using the proximal operator (see Definition~\ref{proxi}), as
$$
\displaystyle u_{t}=\mathrm{prox}_{\Phi(t,\cdot)}(F_{t}),
$$
where $F_{t}\in\HH^{1}_{\mathrm{D}}(\Omega,\R^d)$ is the unique solution to the (smooth) \textit{parameterized Dirichlet-Neumann problem} given by: find~$F_{t}\in\HH^{1}_{\mathrm{D}}(\Omega,\R^d)$ such that
\begin{equation*}
\displaystyle\int_{\Omega}\mathrm{A}\mathrm{e}(F_t):\mathrm{e}(w)=\int_{\Omega}f_t\cdot w+\int_{\Gamma_{\mathrm{N}}}h_t w_\nn, \qquad \forall w\in\HH^{1}_{\mathrm{D}}(\Omega,\R^d),
\end{equation*}
for all~$t \geq 0$, and where $\Phi$ is the parameterized Tresca friction functional defined by
\begin{equation*}
\displaystyle\fonction{\Phi}{\mathbb{R}_{+}\times \HH^{1}_{\mathrm{D}}(\Omega,\R^d)}{\R}{(t,w)}{\displaystyle \Phi(t,w):=\int_{\Gamma_{\mathrm{N}}}g_{t}\left\|w_\tau\right\|.}
\end{equation*}
Similarly to our previous paper~\cite{BCJDC}, to deal with the differentiability (in a generalized sense) of the parameterized proximal operator~$\mathrm{prox}_{\Phi(t,\cdot)} : \HH^{1}_{\mathrm{D}}(\Omega,\R^d) \to \HH^{1}_{\mathrm{D}}(\Omega,\R^d)$, we will invoke the notion of twice epi-differentiability for convex functions introduced by R.T.~Rockafellar in~1985 (see~\cite{Rockafellar}) which leads to the \textit{protodifferentiability} of the corresponding proximal operators. Actually, since the work by R.T.~Rockafellar deals only with nonparameterized convex functions, we will use instead the recent work~\cite{8AB} in which the notion of twice epi-differentiability has been extended to parameterized convex functions (see Definition~\ref{epidiffpara}). 

\paragraph{Main result}
With the previous methodology and under some appropriate assumptions described in Theorem~\ref{caractu0derivDNT}, we prove that the map $t\in\mathbb{R}_{+}\mapsto u_{t}\in\HH^{1}_{\mathrm{D}}(\Omega,\R^d)$ is differentiable at~$t=0$, and its derivative $u'_{0}\in\HH^{1}_{\mathrm{D}}(\Omega,\R^d)$ is given by
$$
\displaystyle u_{0}'=\mathrm{prox}_{\mathrm{D}_{e}^{2}\Phi(u_{0}|F_{0}-u_{0})}(F_{0}'),
$$
where $\mathrm{D}_{e}^{2}\Phi(u_{0}|F_{0}-u_{0})$ stands for the second-order epi-derivative (see Definition~\ref{epidiffpara}) of the parameterized Tresca friction functional~$\Phi$ at~$u_{0}$ for $F_{0}-u_{0}$, and where $F'_{0}\in\HH^{1}_{\mathrm{D}}(\Omega,\R^d)$ is the derivative at~$t=0$ of the map $t\in\mathbb{R}_{+}\mapsto F_{t}\in \HH^{1}_{\mathrm{D}}(\Omega,\R^d)$. Moreover we prove that $u'_{0}\in\HH^{1}_{\mathrm{D}}(\Omega,\R^d)$ exactly corresponds to the unique weak solution to the \textit{tangential Signorini problem}
\begin{equation*}
{\arraycolsep=2pt
\left\{
\begin{array}{rcll}
-\mathrm{div}(\mathrm{A}\mathrm{e}(u'_{0})) & = & f'_0   & \text{ in } \Omega , \\[5pt]
u'_{0} & = & 0  & \text{ on } \Gamma_{\mathrm{D}} ,\\[5pt]
\sigma_{\nn}(u'_{0}) & = & h'_0  & \text{ on } \Gamma_{\mathrm{N}} ,\\[5pt]
u'_{0_\tau} & = & 0  & \text{ on } \Gamma_{\mathrm{N}^{u_0,g_0}_{\mathrm{T}}},\\[5pt]
\sigma_{\tau}(u'_{0})+\frac{g_{0}}{\left\|u_{0_\tau}\right\|}\left(u'_{0_\tau}-\left(u'_{0_\tau}\cdot \frac{u_{0_\tau}}{\left\|u_{0_\tau}\right\|}\right)\frac{u_{0_\tau}}{\left\|u_{0_\tau}\right\|}\right) & = & -g'_0\frac{u_{0_\tau}}{\left\|u_{0_\tau}\right\|}  & \text{ on } \Gamma_{\mathrm{N}^{u_0,g_0}_{\mathrm{R}}} ,\\[15pt]
u'_{0_\tau}\in\R_{-}\frac{\sigma_{\tau}(u_0)}{g_0}, \left(\sigma_{\tau}(u'_0)-g'_0 \frac{\sigma_{\tau}(u_0)}{g_0}\right)\cdot \frac{\sigma_{\tau}(u_0)}{g_0}\leq0 \\ \text{ and } u'_{0_\tau}\cdot\left(\sigma_{\tau}(u'_0)-g'_0 \frac{\sigma_{\tau}(u_0)}{g_0}\right)  & = & 0  & \text{ on } \Gamma_{\mathrm{N}^{u_0,g_0}_{\mathrm{S}}},
\end{array}
\right.}
\end{equation*}
where $\Gamma_{\mathrm{N}}$ is decomposed (up to a null set) as $\Gamma_{\mathrm{N}^{u_0,g_0}_{\mathrm{T}}}\cup\Gamma_{\mathrm{N}^{u_0,g_0}_{\mathrm{R}}}\cup\Gamma_{\mathrm{N}^{u_0,g_0}_{\mathrm{S}}}$ (see details in Theorem~\ref{caractu0derivDNT}), and where, for almost all $s\in\Gamma_{\mathrm{N}^{u_0,g_0}_{\mathrm{S}}}$, $\R_{-}\frac{\sigma_{\tau}(u_0)(s)}{g_0(s)}:=\{ y\in\R^d \mid \exists  \nu \leq0 \text{ such that } y=\nu \frac{\sigma_{\tau}(u_0)(s)}{g_0(s)}\}$. Here~$f'_{0}\in\LL^{2}(\Omega,\R^d)$ (resp.\ $h'_{0}\in~\LL^{2}(\Gamma_{\mathrm{N}})$) is the derivative at $t=0$ of the map $t\in\mathbb{R}_{+}\mapsto f_{t}\in \mathrm{L}^{2}(\Omega,\R^d)$ (resp.\ of the map $t\in\mathbb{R}_{+}\mapsto h_{t}\in \mathrm{L}^{2}(\Gamma_{\mathrm{N}})$) and~$g'_{0}\in\LL^{2}(\Gamma_{\mathrm{N}})$ is the map defined, for almost every~$s\in\Gamma_{\mathrm{N}}$, by~$g'_{0}(s):=\lim_{t \to 0^{+}}\frac{g_{t}(s)-g_{0}(s)}{t}$.

We emphasize that the boundary conditions which appear on $\Gamma_{\mathrm{N}^{u_0,g_0}_{\mathrm{S}}}$ are called the \textit{tangential Signorini's unilateral conditions}. They are close to the classical Signorini's unilateral conditions which describe a non-permeable contact (see, e.g.,~\cite{15SIG,16SIG}) except that, here, they are concerned with the tangential components (instead of the usual normal components). Roughly speaking our main result claims that the tangential Signorini's solution can be considered as first-order approximation to the perturbed Tresca's solution.

\paragraph{Application to an optimal control problem}
The above sensitivity analysis allows us to investigate the optimal control problem given by
\begin{equation*}
    \minn\limits_{ \substack{ z\in \mathcal{U}}} \; \mathcal{J}(z),
\end{equation*}
where $\mathcal{J}$ is the cost functional given by
\begin{equation*}
\fonction{\mathcal{J}}{\mathrm{V}}{\R}{z}{\mathcal{J}(z):=\frac{1}{2}\left\|u(\ell(z))\right\|^{2}_{\HH^{1}_{\mathrm{D}}(\Omega,\R^d)}+\frac{\beta}{2}\left\| \ell(z) \right\|^{2}_{\LL^{2}(\Gamma_{\mathrm{N}})},}
\end{equation*}
where~$\mathrm{V}$ is the open subset of~$\LL^{\infty}(\Gamma_{\mathrm{N}})$ defined by
$$
\mathrm{V}:=\left\{ z\in\LL^{\infty}(\Gamma_{\mathrm{N}}) \mid \exists C(z)>0\text{, } \ell(z)>C(z) \text{ \textit{\textit{a.e.}} on } \Gamma_{\mathrm{N}} \right\},
$$
where $\ell$ is the map defined by $z\in\LL^{\infty}(\Gamma_{\mathrm{N}})\mapsto \ell(z):=g_1+zg_2\in \LL^{\infty}(\Gamma_{\mathrm{N}})$, where~$g_1\in\LL^{\infty}(\Gamma_{\mathrm{N}})$ with~$g_1\geq m$ \textit{a.e.} on $\Gamma_{\mathrm{N}}$ for some positive constant~$m>0$ and $g_2\in\LL^{\infty}(\Gamma_{\mathrm{N}})$ such that~$||g_2||_{\LL^{\infty}(\Gamma_{\mathrm{N}})}>0$, and where $u(\ell(z))\in\HH^{1}_{\mathrm{D}}(\Omega,\R^d)$ stands for the unique solution to the Tresca friction problem given by
\begin{equation}\tag{CTP\ensuremath{_{\ell(z)}}}
\arraycolsep=2pt
\left\{
\begin{array}{rcll}
-\mathrm{div}(\mathrm{A}\mathrm{e}(u)) & = & f   & \text{ in } \Omega , \\
u & = & 0  & \text{ on } \Gamma_{\mathrm{D}} ,\\
\sigma_\nn(u) & = & h  & \text{ on } \Gamma_{\mathrm{N}},\\
\left\|\sigma_\tau(u)\right\|\leq \ell(z) \text{ and } u_{\tau}\cdot\sigma_{\tau}(u)+\ell(z)\left\|u_{\tau}\right\| & = & 0  & \text{ on } \Gamma_{\mathrm{N}},
\end{array}
\right.
\end{equation}
where $f\in \mathrm{L}^{2}(\Omega,\R^d)$ and $h\in \mathrm{L}^{2}(\Gamma_{\mathrm{N}})$, where~$\beta > 0$ is a positive constant and where $\mathcal{U}$ is a given nonempty convex subset of $\mathrm{V}$ such that $\mathcal{U}$ is a bounded closed subset of~$\LL^{2}(\Gamma_{\mathrm{N}})$. Note that the first term in the cost functional $\mathcal{J}$ corresponds to the compliance, while the second term is the energy consumption which is standard in optimal control problems (see, e.g.,~\cite{MANZA}).

We prove in Theorem~\ref{gradientdelafonccoutgg}
that the cost functional $\mathcal{J}$ is Gateaux differentiable on $\mathrm{V}$, and its Gateaux differential at any $z_{0}\in\mathrm{V}$, denoted by $\mathrm{d}_{G}\mathcal{J}(z_{0})$, is given by
$$
\mathrm{d}_{G}\mathcal{J}(z_{0})(z)=\int_{\Gamma_{\mathrm{N}^{u_0,\ell(z_0)}_{\mathrm{R}}}}zg_2\left(\beta\left( g_1+z_0 g_2\right)-\left\|u_{0_\tau}\right\|\right)+\int_{\Gamma_{\mathrm{N}^{u_0,\ell(z_0)}_{\mathrm{T}}}\cup\Gamma_{\mathrm{N}^{u_0,\ell(z_0)}_{\mathrm{S}}}}\beta zg_2\left(g_1+z_0g_2\right),
$$
for all $z\in\LL^{\infty}(\Gamma_{\mathrm{N}})$, where $u_0:=u(\ell(z_0))$ is the solution to the Tresca friction problem~(CTP$_{\ell(z_0)}$), and where $\Gamma_{\mathrm{N}}$ is decomposed (up to a null set) as $\Gamma_{\mathrm{N}^{u_0,\ell(z_0)}_{\mathrm{T}}}\cup\Gamma_{\mathrm{N}^{u_0,\ell(z_0)}_{\mathrm{R}}}\cup\Gamma_{\mathrm{N}^{u_0,\ell(z_0)}_{\mathrm{S}}}$.

The expression of the Gateaux differential of $\mathcal{J}$ allows us to exhibit an explicit descent direction of $\mathcal{J}$ (see Subsection~\ref{numericalsimjfkjfsdkjfsdkf} for details). Hence, using this descent direction together with a basic projected gradient algorithm, we perform numerical simulations to solve the optimal control problem on a two-dimensional example.

 \paragraph{Novelty and originality of the present paper}
We emphasize here that the previous works~\cite{4ABC,ABCJ,BCJDC} focused on scalar models, while the present work deals with the vectorial linear elasticity model, which constitutes an essential step in view of dealing with concrete cases and applications. Even if the present work is inspired by the previous papers, we want to underline that the extension to the vectorial context is not trivial and leads to several additional difficulties, especially in the investigation of the parameterized twice epi-differentiability of the parameterized Tresca friction functional which involves the tangential norm map $||\cdot_{\tau (s)}||$, for almost all $s\in\Gamma_{\mathrm{N}}$. In particular, we present in this paper a generalization (see Proposition~\ref{epitangnorm}) of a result proved by C. N. Do about the twice epi-differentiability of a support function (see~\cite[Example 2.7 p.286]{DO}) which is used next to prove that the tangential norm map is twice epi-differentiable (see Subsection~\ref{sectiontwiceepi} for details). Moreover, the methodology used in this paper allowed us to characterize the derivative of the solution to the Tresca friction problem as the solution to a non-standard boundary problem (a tangential Signorini problem), that does not appear in the literature yet and which constituted an additional difficulty. Finally we present a first application of these results to an optimal control problem, with the aim of illustrating the feasibility of the presented method.

\paragraph{Organization of the paper}
The paper is organized as follows.  Section~\ref{Mainresult1} is the core of the present work: in Subsection~\ref{BVP} we describe the functional framework and we introduce three boundary value problems that are involved all along the paper; in Subsection~\ref{section4} the sensitivity analysis of the Tresca friction problem is performed. In Section~\ref{section4cc}, we investigate an optimal control problem and numerical simulations are performed to solve it on a two-dimensional example. Finally Appendix~\ref{appendix} is dedicated to some basic recalls from convex, variational and functional analyses used throughout the paper.

\section{Main result}\label{Mainresult1}

In this section let $d\in\left\{2,3\right\}$ and $\Omega$ be a nonempty bounded connected open subset of $\R^{d}$ with a $\mathcal{C}^{1}$-boundary denoted by~$\Gamma:=\partial{\Omega}$ (see Remark~\ref{regularityofn} for comments on this~$\mathcal{C}^{1}$-regularity assumption). We denote by $\LL^{2}(\Omega,\R^d)$, $\mathrm{L}^{2}(\Gamma,\R^d)$, $\LL^{1}(\Gamma,\R^d)$, $\HH^{1}(\Omega,\R^d)$, $\HH^{1/2}(\Gamma,\R^d)$, $\HH^{-1/2}(\Gamma,\R^d)$ the usual Lebesgue and Sobolev spaces endowed with their standard norms. 
Moreover~the~no\-tation~$\MD(\Omega,\R^d)$ stands for the set of infinitely differentiable functions $\varphi  :  \Omega \rightarrow \R^{d}$ with compact support in $\Omega$, and~$\MD'(\Omega,\R^d)$ for the set of distributions on $\Omega$. Moreover, all along this paper, we denote by~$:$  the scalar product defined by~$\mathrm{B}:\mathrm{C}=\sum_{i=1}^{d}\mathrm{B}_i\cdot\mathrm{C}_i$ for all $\mathrm{B},\mathrm{C}\in\R^{d\times d}$, where $\mathrm{B}_i\in\R^d$ (resp. $\mathrm{C}_i\in\R^d$) is the~$i$-th line of~$\mathrm{B}$ (resp. $\mathrm{C}$) for all $i\in[[1,d]]$. 

Let us consider the decomposition
$$
\Gamma=:\Gamma_{\mathrm{D}}\cup\Gamma_{\mathrm{N}},
$$
where $\Gamma_{\mathrm{D}}$ and $\Gamma_{\mathrm{N}}$ are two measurable (with positive measure) pairwise disjoint subsets of $\Gamma$ such that almost every point of $\Gamma_{\mathrm{N}}$ belongs to $\mathrm{int}_{\Gamma}({\Gamma_{\mathrm{N}}})$ (see Remark~\ref{interpot} for comments on this last assumption). We introduce $\HH^{1}_{\mathrm{D}}(\Omega,\R^d)$ the linear subspace of $\HH^{1}(\Omega,\R^d)$ defined by
$$
\HH^{1}_{\mathrm{D}}(\Omega,\R^d):=\left\{w\in\HH^{1}(\Omega,\R^d)\mid w=0 \text{ \textit{a.e.} on } \Gamma_{\mathrm{D}} \right\}.
$$
Moreover, we assume that $\Omega$ is an elastic solid satisfying the linear elastic model (see, e.g.,~\cite{SALEN}), that is
$$
\sigma(w)=\mathrm{A}\mathrm{e}(w),
$$
where $\sigma$ is the Cauchy stress tensor, $\mathrm{A}$ the stiffness tensor, and $\mathrm{e}$ is the infinitesimal strain tensor defined by 
    \begin{equation*}
    \mathrm{e}(w):=\frac{1}{2}(\nabla{w}+\nabla{w}^{\top}),
    \end{equation*}
for all displacement field $w\in\HH^1(\Omega,\R^d)$. We also assume that all coefficients of $\mathrm{A}$ are measurable (denoted by $a_{ijkl}$ for all $\left(i,j,k,l\right)\in\left\{1,...,d\right\}^{4}$) and that there exist two constants $\alpha>0$ and $\gamma>0$ such that all coefficients of $\mathrm{A}$ and $e$ (denoted by $\epsilon_{ij}$ for all $\left(i,j\right)\in\left\{1,...,d\right\}^{2}$) satisfy
 \begin{equation*}
a_{ijkl}(x)=a_{jikl}(x)=a_{lkij}(x), \qquad |a_{ijkl}(x)|\leq \alpha,
\end{equation*}
and also
$$\displaystyle\sum_{i=1}^{d}\sum_{j=1}^{d}\sum_{k=1}^{d}\sum_{l=1}^{d}a_{ijkl}\epsilon_{ij}(w_1)(x)\epsilon_{kl}(w_2)(x)\geq\gamma\sum_{i=1}^{d}\sum_{j=1}^{d}\epsilon_{ij}(w_1)(x)\epsilon_{ij}(w_2)(x),
$$
for all displacement field $w_1,w_2\in\HH^1(\Omega,\R^d)$ and for almost all $x\in\Omega$. Moreover, since $\Gamma_{\mathrm{D}}$ has a positive measure, then we can deduce that 
\begin{equation*}
\fonction{\dual{\cdot}{\cdot}_{\HH^{1}_{\mathrm{D}}(\Omega,\R^{d})}}{\left(\HH^{1}_{\mathrm{D}}(\Omega,\R^{d})\right)^2}{\R}{(w_1,w_2)}{\displaystyle\int_{\Omega}\mathrm{A}\mathrm{e}(w_1):\mathrm{e}(w_2),}
\end{equation*}
is a scalar product on $\HH^{1}_{\mathrm{D}}(\Omega,\R^{d})$ (see, e.g.,~\cite[Chapter 3]{DUVAUTLIONS}) and we denote by $\left\|\cdot\right\|_{\HH^{1}_{\mathrm{D}}(\Omega,\R^{d})}$ the corresponding norm.

We denote by $\nn\in\mathcal{C}^0(\Gamma)$ the outward-pointing unit normal vector to $\Gamma$. Therefore, for any~$w\in\LL^{2}(\Gamma,\R^d)$, one has
$w=w_{\nn}\nn+w_\tau$,
where $w_{\nn}:=w\cdot\nn\in\LL^{2}(\Gamma,\R)$ 
and $w_{\tau}:=w-w_{\nn}\nn\in\LL^2(\Gamma,\R^d)$.
In particular, if the stress vector $\mathrm{A}\mathrm{e}(w)\nn$ is in $\LL^2(\Gamma_{\mathrm{N}},\R^d)$ for some
 $w\in\HH^1(\Omega,\R^d)$, then we use the notation 
$$
\mathrm{A}\mathrm{e}(w)\nn=\sigma_{\nn}(w)\nn+\sigma_\tau(w),
$$
where $\sigma_{\nn}(w)\in\LL^{2}(\Gamma_{\mathrm{N}},\R)$ is the normal stress and $\sigma_{\tau}(w)\in\LL^{2}(\Gamma_{\mathrm{N}},\R^d)$ is the shear stress. We also denote by $\left\|\cdot\right\|$ the Euclidean norm on $\R^{d}$ and, for all $w\in\LL^{2}(\Gamma,\R^d)$, $\left\|w_\tau\right\|\in\LL^{2}(\Gamma)$ is defined by
$$
\fonction{\left\|w_\tau\right\|}{\Gamma}{\R}{s}{\displaystyle \left\|w_{\tau}(s)\right\|.}
$$
The rest of this section is organized as follows. Subsection~\ref{BVP} introduces three boundary value problems involved all along the paper: a Dirichlet-Neumann problem (see Problem~\eqref{PbNeumannDirichlet}), a tangential Signorini problem (see Problem~\eqref{PbtangSignorini}) and a Tresca friction problem (see Problem~\ref{PbTresca}). In Subsection~\ref{section4}, the sensitivity analysis of the Tresca friction problem is performed and we establish the main result of this paper (see Theorem~\ref{caractu0derivDNT}).

\subsection{Three boundary value problems}\label{BVP}

For the needs of this subsection, let us fix $f\in \mathrm{L}^{2}(\Omega,\R^d)$. Only the proofs of Subsection~\ref{SectionSignorinicasscalairesansu} are detailed since the tangential Signorini problem is, to the best of our knowledge, new in the literature. For the proofs of the other problems, they are classical and close to the ones presented in~\cite{4ABC} and thus they are left to the reader.

\subsubsection{A problem with Dirichlet-Neumann conditions}
Let $z\in\LL^{2}(\Gamma_{\mathrm{N}},\R^d)$ and consider the Dirichlet-Neumann problem given by
\begin{equation}\tag{DN}\label{PbNeumannDirichlet}
\arraycolsep=2pt
\left\{
\begin{array}{rcll}
-\mathrm{div}(\mathrm{A}\mathrm{e}(F)) & = & f   & \text{ in } \Omega , \\
F & = & 0  & \text{ on } \Gamma_{\mathrm{D}} ,\\
\mathrm{A}\mathrm{e}(F)\nn & = & z  & \text{ on } \Gamma_{\mathrm{N}}.
\end{array}
\right.
\end{equation}

\begin{myDefn}[Strong solution to the Dirichlet-Neumann problem]
A (strong) solution to the Dirichlet-Neumann problem~\eqref{PbNeumannDirichlet} is a function $F\in\HH^{1}(\Omega,\R^d)$ such that $-\mathrm{div}(\mathrm{A}\mathrm{e}(F))=f$ in~$\MD'(\Omega,\R^d)$,~$F=0$ \textit{a.e.} on $\Gamma_{\mathrm{D}}$, $\mathrm{A}\mathrm{e}(F)\nn\in\LL^{2}(\Gamma_{\mathrm{N}},\R^d)$ with $\mathrm{A}\mathrm{e}(F)\nn=z$ \textit{a.e.} on $\Gamma_{\mathrm{N}}$.
\end{myDefn}

\begin{myDefn}[Weak solution to the Dirichlet-Neumann problem]
A weak solution to the Dirichlet-Neumann problem~\eqref{PbNeumannDirichlet} is a function $F\in\HH^{1}_{\mathrm{D}}(\Omega,\R^d)$ such that
\begin{equation*}
\displaystyle\int_{\Omega}\mathrm{A}\mathrm{e}(F):\mathrm{e}(w)=\int_{\Omega}f\cdot w+\int_{\Gamma_{\mathrm{N}}}z\cdot w, \qquad \forall w\in\HH^{1}_{\mathrm{D}}(\Omega,\R^d).
\end{equation*}
\end{myDefn}

\begin{myProp}\label{DNequiart}
A function $F\in\HH^{1}(\Omega,\R^d)$ is a (strong) solution to the Dirichlet-Neumann problem~\eqref{PbNeumannDirichlet} if and only if $F$ is a weak solution to the Dirichlet-Neumann problem~\eqref{PbNeumannDirichlet}.
\end{myProp}

Using the Riesz representation theorem, we obtain the following existence/uniqueness result.

\begin{myProp}\label{existenceunicitéDN}
The Dirichlet-Neumann problem~\eqref{PbNeumannDirichlet} admits a unique (strong) solution~$F\\\in\HH^{1}_{\mathrm{D}}(\Omega,\R^d)$. Moreover there exists a constant $C \geq 0$ (depending only on $\Omega$) such that
$$
\left \| F \right \|_{\HH^{1}_{\mathrm{D}}(\Omega,\R^d)}\leq C\left(\left \| f  \right \|_{\LL^{2}(\Omega,\R^d)} + \left \| z \right \|_{\LL^{2}(\Gamma_{\mathrm{N}},\R^d)}\right).
$$
\end{myProp}

\subsubsection{A tangential Signorini problem}\label{SectionSignorinicasscalairesansu}

In this part we assume that $\Gamma_{\mathrm{N}}$ is decomposed (up to a null set) as
$$
\Gamma_{\mathrm{N}}=:\Gamma_{\mathrm{N_T}}\cup\Gamma_{\mathrm{N_R}}\cup\Gamma_{\mathrm{N_S}},
$$
where $\Gamma_{\mathrm{N_T}}$, $\Gamma_{\mathrm{N_R}}$, $\Gamma_{\mathrm{N_S}}$ are three measurable pairwise disjoint subsets of $\Gamma_{\mathrm{N}}$. Moreover let~$h\in \mathrm{L}^{2}(\Gamma_{\mathrm{N}})$,~$\ell\in\LL^{2}(\Gamma_{\mathrm{N}})$, $v\in\LL^{\infty}(\Gamma_{\mathrm{N}},\R^d)$ such that $||v||_{\LL^{\infty}(\Gamma_{\mathrm{N_R}}\cup\Gamma_{\mathrm{N_S}},\R^d)}\leq1$,~$k\in\LL^{4}(\Gamma_{\mathrm{N_R}})$ such that $k>0$ \textit{a.e.} on $\Gamma_{\mathrm{N_R}}$, and we denote, for almost all $s\in\Gamma_{\mathrm{N_S}}$, $\R_{-}v_{\tau}(s):=\{ y\in\R^d \mid \exists  \nu \leq0 \text{ such that } y=\nu v_{\tau}(s)\}$. The tangential Signorini problem is given by

\begin{equation}\tag{SP}\label{PbtangSignorini}
\arraycolsep=2pt
\left\{
\begin{array}{rcll}
-\mathrm{div}(\mathrm{A}\mathrm{e}(u)) & = & f   & \text{ in } \Omega , \\
u & = & 0  & \text{ on } \Gamma_{\mathrm{D}} ,\\
\sigma_{\nn}(u) & = & h  & \text{ on } \Gamma_{\mathrm{N}} ,\\
u_{\tau} & = & 0  & \text{ on } \Gamma_{\mathrm{N_T}},\\
\sigma_{\tau}(u)+k\left(u_{\tau}-\left(u_{\tau}\cdot v_{\tau}\right)v_{\tau}\right) & = & \ell v_{\tau}  & \text{ on } \Gamma_{\mathrm{N_R}} ,\\
u_{\tau}\in\R_{-}v_{\tau}, \left(\sigma_{\tau}(u)-\ell v_{\tau}\right)\cdot v_{\tau}\leq0 \text{ and } u_{\tau}\cdot\left(\sigma_{\tau}(u)-\ell v_{\tau}\right)  & = & 0  & \text{ on } \Gamma_{\mathrm{N_S}}.
\end{array}
\right.
\end{equation}

\begin{myDefn}[Strong solution to the tangential Signorini problem]
A (strong) solution to the tangential Signorini problem~\eqref{PbtangSignorini} is a function $u\in\HH^{1}(\Omega,\R^d)$ such that $-\mathrm{div}(\mathrm{A}\mathrm{e}(u))=f$ in $\MD'(\Omega,\R^d)$,~$u=0$ \textit{a.e.} on $\Gamma_{\mathrm{D}}$,~$u_{\tau}=0$ \textit{a.e.} on $\Gamma_{\mathrm{N_T}}$, $\mathrm{A}\mathrm{e}(u)\nn\in\LL^{2}(\Gamma_{\mathrm{N}},\R^d)$ with $\sigma_{\nn}(u)=h$ \textit{a.e.} on $\Gamma_{\mathrm{N}}$, $\sigma_{\tau}(u)+k\left(u_{\tau}-(u_{\tau}\cdot v_{\tau})v_{\tau}\right)=\ell v_{\tau}$ \textit{a.e.} on $\Gamma_{\mathrm{N_R}}$, $u_{\tau}\in\R_{-}v_{\tau}, \left(\sigma_{\tau}(u)-\ell v_{\tau}\right)\cdot v_{\tau}\leq0 \text{ and } u_{\tau}\cdot\left(\sigma_{\tau}(u)-\ell v_{\tau}\right)=0$ \textit{a.e.} on $\Gamma_{\mathrm{N_S}}$.
\end{myDefn}

\begin{myDefn}[Weak solution to the tangential Signorini problem]\label{weakformuletangentSigno}
A weak solution to the tangential Signorini problem~\eqref{PbtangSignorini} is a function $u\in\mathcal{K}^{1}(\Omega,\R^d)$ such that
\begin{multline}
\label{FaibleSignorini}
\displaystyle\int_{\Omega}\mathrm{A}\mathrm{e}(u):\mathrm{e}(w-u)\geq\int_{\Omega}f\cdot (w-u)+\int_{\Gamma_{\mathrm{N}}}h(w_{\nn}-u_\nn)+\int_{\Gamma_{\mathrm{N_R}}}\left(\ell v_{\tau}-k\left(u_{\tau}-(u_{\tau}\cdot v_{\tau})v_{\tau}\right)\right)\cdot(w_\tau-u_\tau)\\+\int_{\Gamma_{\mathrm{N_S}}}\ell v_{\tau}\cdot(w_\tau-u_\tau), \qquad \forall w\in\mathcal{K}^{1}(\Omega,\R^d),
\end{multline}
where $\mathcal{K}^{1}(\Omega,\R^d)$ is the nonempty closed convex subset of $\HH^{1}_{\mathrm{D}}(\Omega,\R^d)$ given by
$$
\mathcal{K}^{1}(\Omega,\R^d) := \left\{w\in\HH^{1}_{\mathrm{D}}(\Omega,\R^d) \mid w_\tau=0 \text{ \textit{a.e.} on } \Gamma_{\mathrm{N_T}}\text{ and } w_\tau\in\R_{-}v_{\tau} \text{ \textit{a.e.} on }\Gamma_{\mathrm{N_S}} \right \}.
$$
\end{myDefn}

One can easily prove that a (strong) solution to the tangential Signorini problem~\eqref{PbtangSignorini} is also a weak solution. However, to the best of our knowledge, without additional assumptions one cannot prove the converse. To get the equivalence, one can assume, in particular, that the decomposition~$\Gamma_{\mathrm{D}}\cup\Gamma_{\mathrm{N_T}}\cup\Gamma_{\mathrm{N_R}}\cup\Gamma_{\mathrm{N_S}}$ of $\Gamma$ is \textit{consistent} in the following sense.
\begin{myDefn}[Consistent decomposition]\label{regulieresens2}
 The decomposition $\Gamma_{\mathrm{D}}\cup\Gamma_{\mathrm{N_T}}\cup\Gamma_{\mathrm{N_R}}\cup\Gamma_{\mathrm{N_S}}$ of $\Gamma$ is said to be \emph{consistent} if:
 \begin{enumerate}[label={\rm (\roman*)}]
     \item for almost all $s\in\Gamma_{\mathrm{N_S}}$,  $s\in \mathrm{int}_{\Gamma}(\Gamma_{\mathrm{N_S}})$;
     \item the nonempty closed convex subset $\mathcal{K}^{1/2}(\Gamma,\R^d)$ of $\HH^{1/2}(\Gamma,\R^d)$ defined by
    \begin{multline*}
            \mathcal{K}^{1/2}(\Gamma,\R^d):=\biggl \{ w\in \HH^{1/2}(\Gamma,\R^d) \mid w=0 \text{ \textit{a.e.} on } \Gamma_{\mathrm{D}}\text{, } w_\tau=0 \text{ \textit{a.e.} on }\Gamma_{\mathrm{N_T}} \\ \text{ and } w_\tau\in\R_{-}v_{\tau} \text{ \textit{a.e.} on }\Gamma_{\mathrm{N_S}} \biggl\},
    \end{multline*}        
is dense in the nonempty closed convex subset $\mathcal{K}^{0}(\Gamma,\R^d)$ of $\mathrm{L}^{2}(\Gamma,\R^d)$ given by
    \begin{multline*}
         \mathcal{K}^{0}(\Gamma,\R^d):=\biggl \{ w\in \mathrm{L}^{2}(\Gamma,\R^d) \mid w=0 \text{ \textit{a.e.} on } \Gamma_{\mathrm{D}}\text{, } w_\tau=0 \text{ \textit{a.e.} on }\Gamma_{\mathrm{N_T}} \\ \text{ and } w_\tau\in\R_{-}v_{\tau} \text{ \textit{a.e.} on }\Gamma_{\mathrm{N_S}} \biggl \}.
    \end{multline*}
\end{enumerate}
\end{myDefn}

\begin{myProp}\label{EquiSignorini}
Let $u\in \HH^{1}(\Omega,\R^d)$.
\begin{enumerate}[label={\rm (\roman*)}]
    \item If $u$ is a (strong) solution to the tangential Signorini problem~\eqref{PbtangSignorini}, then $u$ is a weak solution to the tangential Signorini problem~\eqref{PbtangSignorini}.
    \item If $u$ is a weak solution to the Signorini problem~\eqref{PbtangSignorini} such that $\mathrm{A}\mathrm{e}(u)\nn\in\LL^{2}(\Gamma_{\mathrm{N}},\R^d)$ and the decomposition $\Gamma_{\mathrm{D}}\cup\Gamma_{\mathrm{N_T}}\cup\Gamma_{\mathrm{N_R}}\cup\Gamma_{\mathrm{N_S}}$ of $\Gamma$ is consistent, then $u$ is a (strong) solution to the tangential Signorini problem~\eqref{PbtangSignorini}.
\end{enumerate}
\end{myProp}

\begin{proof}
  Assume that $u$ is a (strong) solution to the tangential Signorini problem~\eqref{PbtangSignorini}. Then, from the boundary conditions, $u\in \mathcal{K}^{1}(\Omega,\R^d)$. Moreover, since $-\mathrm{div}(\mathrm{A}\mathrm{e}(u))=f$ in~$\MD'(\Omega,\R^d)$ and  $f\in\LL^{2}(\Omega,\R^d)$, then $-\mathrm{div}(\mathrm{A}\mathrm{e}(u))=f$ in $\LL^{2}(\Omega,\R^d)$. Hence, from divergence formula (see Proposition~\ref{div}), one gets
$$
\displaystyle \int_{\Omega}\mathrm{A}\mathrm{e}(u):\mathrm{e}(w-u)-\dual{\mathrm{A}\mathrm{e}(u)\nn}{w-u}_{\HH^{-1/2}(\Gamma,\R^d)\times \HH^{1/2}(\Gamma,\R^d)}=\int_{\Omega}f\cdot(w-u),
$$
for all $w\in \mathcal{K}^{1}(\Omega,\R^d)$. Moreover, for all $w\in \mathcal{K}^{1}(\Omega,\R^d)$, $w\in\HH^{1/2}_{00}(\Gamma_{\mathrm{N}},\R^d)$ which can be identified to a linear subspace of $\HH^{1/2}(\Gamma,\R^d)$, hence
$$
\displaystyle \displaystyle \int_{\Omega}\mathrm{A}\mathrm{e}(u):\mathrm{e}(w-u)-\dual{\mathrm{A}\mathrm{e}(u)\nn}{w-u}_{\HH^{-1/2}_{00}(\Gamma_{\mathrm{N}},\R^d)\times \HH^{1/2}_{00}(\Gamma_{\mathrm{N}},\R^d)}=\int_{\Omega}f\cdot(w-u), 
$$
for all $w\in \mathcal{K}^{1}(\Omega,\R^d)$. Furthermore, since $\mathrm{A}\mathrm{e}(u)\nn\in\LL^{2}(\Gamma_{\mathrm{N}},\R^d)$, it follows that
$$
\dual{\mathrm{A}\mathrm{e}(u)\nn}{w-u}_{\HH^{-1/2}_{00}(\Gamma_{\mathrm{N}},\R^d)\times \HH^{1/2}_{00}(\Gamma_{\mathrm{N}},\R^d)}=\int_{\Gamma_{\mathrm{N}}}\mathrm{A}\mathrm{e}(u)\nn\cdot (w-u),
$$
for all $w\in \mathcal{K}^{1}(\Omega,\R^d)$.
Using the decomposition of $\mathrm{A}\mathrm{e}(u)\nn$ on its tangential and normal components, one has
$$
\int_{\Gamma_{\mathrm{N}}}\mathrm{A}\mathrm{e}(u)\nn\cdot (w-u)=\int_{\Gamma_{\mathrm{N}}}\sigma_\nn(u)(w_\nn-u_\nn)+\int_{\Gamma_{\mathrm{N_R}}\cup\Gamma_{\mathrm{N_S}}}\sigma_\tau(u)\cdot (w_\tau-u_\tau),
$$
for all $w\in \mathcal{K}^{1}(\Omega,\R^d)$. From the boundary conditions, one has $\sigma_{\nn}(u)=h$ \textit{a.e.} on $\Gamma_{\mathrm{N}}$ and~$\sigma_{\tau}(u)=\ell v_{\tau}- k\left(u_{\tau}-(u_{\tau}\cdot v_{\tau})v_{\tau}\right)$ \textit{a.e.} on $\Gamma_{\mathrm{N_R}}$.
Moreover one has
$$
\displaystyle\sigma_\tau(u)\cdot\left(w_\tau-u_\tau\right)=\sigma_\tau(u)\cdot w_\tau - \sigma_\tau(u)\cdot u_\tau \geq \ell v_\tau\cdot w_\tau-\ell v_\tau\cdot u_\tau=\ell v_\tau\cdot\left(w_\tau-u_\tau\right),
$$
\textit{a.e.} on $\Gamma_{\mathrm{N_S}}$. This concludes
the proof of the first item.

{\rm (ii)} Assume that $u$ is a weak solution to the tangential Signorini problem~\eqref{PbtangSignorini}. Then $u\in\mathcal{K}^{1}(\Omega,\R^d)$. For all $\varphi\in\MD(\Omega,\R^d)$, considering $w:=u\pm\varphi\in\mathcal{K}^{1}(\Omega,\R^d)$ in Inequality~\eqref{FaibleSignorini}, one gets~$-\mathrm{div}(\mathrm{A}\mathrm{e}(u))=f$ in~$\MD'(\Omega,\R^d)$, then also in $\LL^{2}(\Omega,\R^d)$ since $f\in\LL^{2}(\Omega,\R^d)$. Hence we can apply the divergence formula (see Proposition~\ref{div}) in Inequality~\eqref{FaibleSignorini} to get that
\begin{multline*}
\displaystyle\dual{\mathrm{A}\mathrm{e}(u)\nn}{w-u}_{\HH^{-1/2}(\Gamma,\R^d)\times \HH^{1/2}(\Gamma,\R^d)}\geq \int_{\Gamma_{\mathrm{N}}}h(w_{\nn}-u_\nn)\\+\int_{\Gamma_{\mathrm{N_R}}}\left(\ell v_{\tau}-k\left(u_{\tau}-(u_{\tau}\cdot v_{\tau})v_{\tau}\right)\right)\cdot(w_\tau-u_\tau)+\int_{\Gamma_{\mathrm{N_S}}}\ell v_{\tau}\cdot(w_\tau-u_\tau),
\end{multline*}
for all $w\in\mathcal{K}^{1}(\Omega,\R^d)$.
Moreover, similarly to ${\rm (i)}$ and from the assumption $\mathrm{A}\mathrm{e}(u)\nn\in\LL^{2}(\Gamma_{\mathrm{N}},\R^d)$, one gets
\begin{multline}\label{inegaliteDNS}
\displaystyle\int_{\Gamma_{\mathrm{N}}}\sigma_{\nn}(u)(w_\nn-u_\nn)+\int_{\Gamma_{\mathrm{N_R}}\cup\Gamma_{\mathrm{N_S}}}\sigma_{\tau}(u)\cdot(w_\tau-u_\tau)\geq \int_{\Gamma_{\mathrm{N}}}h(w_{\nn}-u_\nn)\\+\int_{\Gamma_{\mathrm{N_R}}}\left(\ell v_{\tau}-k\left(u_{\tau}-(u_{\tau}\cdot v_{\tau})v_{\tau}\right)\right)\cdot(w_\tau-u_\tau)+\int_{\Gamma_{\mathrm{N_S}}}\ell v_{\tau}\cdot(w_\tau-u_\tau),
\end{multline}
for all $w\in\mathcal{K}^{1}(\Omega,\R^d)$, then also for all $w\in\mathcal{K}^{1/2}(\Gamma,\R^d)$. From the assumption that the decomposition $\Gamma_{\mathrm{D}}\cup\Gamma_{\mathrm{N_T}}\cup\Gamma_{\mathrm{N_R}}\cup\Gamma_{\mathrm{N_S}}$ of $\Gamma$ is consistent, $\mathcal{K}^{1/2}(\Gamma,\R^d)$ is dense in $\mathcal{K}^{0}(\Gamma,\R^d)$. Therefore, since~$k\in\LL^{4}(\Gamma_{\mathrm{N_R}})$, $||v||_{\LL^{\infty}(\Gamma_{\mathrm{N_R}}\cup\Gamma_{\mathrm{N_S}},\R^d)}\leq1$ and from the continuous embedding 
$\HH^{1}(\Omega,\R^d){\hookrightarrow} \LL^{4}(\Gamma,\R^d)$, we deduce that Inequality~\eqref{inegaliteDNS} is still  true for all $w\in\mathcal{K}^{0}(\Gamma,\R^d)$. 

By considering the function $w:=u\pm \psi\nn\in\mathcal{K}^{0}(\Gamma,\R^d)$ in Inequality~\eqref{inegaliteDNS}, where  $\psi\in\LL^{2}(\Gamma)$ is given by
$$
\psi=\left\{
\begin{array}{rcll}
0	&     & \text{ on } \Gamma_{\mathrm{D}} , \\
\phi	&   & \text{ on } \Gamma_{\mathrm{N}},
\end{array}
\right.
$$
with $\phi$ any function in $\LL^{2}(\Gamma_{\mathrm{N}})$, one deduces that~$\sigma_{\nn}(u)=h$ \textit{a.e.} on $\Gamma_{\mathrm{N}}$.

By considering $w:=u\pm w_\phi\in\mathcal{K}^{0}(\Gamma,\R^d)$ in Inequality~\eqref{inegaliteDNS},
where  $w_\phi\in\mathrm{L}^{2}(\Gamma,\R^d)$ is given by
$$
w_\phi=\left\{
\begin{array}{rcll}
0	&     & \text{ on } \Gamma_{\mathrm{D}}\cup\Gamma_{\mathrm{N_T}}\cup\Gamma_{\mathrm{N_S}} , \\
\phi	&   & \text{ on } \Gamma_{\mathrm{N_R}},
\end{array}
\right.
$$
with $\phi$ any function in $\LL^{2}(\Gamma_{\mathrm{N_R}},\R^d)$, one gets that $\sigma_{\tau}(u)=\ell v_{\tau}-k\left(u_{\tau}-(u_{\tau}\cdot v_{\tau})v_{\tau}\right)$ \textit{a.e.} on~$\Gamma_{\mathrm{N_R}}$.
Hence Inequality~\eqref{inegaliteDNS} becomes
\begin{equation}\label{inegaliteDNS2}
\int_{\Gamma_{\mathrm{N_S}}}\sigma_{\tau}(u)\cdot(w_\tau-u_\tau)\geq\int_{\Gamma_{\mathrm{N_S}}}\ell v_{\tau}\cdot(w_\tau-u_\tau),
\end{equation}
for all $w\in\mathcal{K}^{0}(\Gamma,\R^d)$. 
Let $s\in\Gamma_{\mathrm{N_S}}$ be a Lebesgue point of $\sigma_{\tau}(u)\cdot v_\tau\in\LL^{2}(\Gamma_{\mathrm{N_R}}\cup\Gamma_{\mathrm{N_S}})$ and of~$\ell\left\|v_\tau\right\|_{2}^{2}\in\LL^{2}(\Gamma_{\mathrm{N_R}}\cup\Gamma_{\mathrm{N_S}})$, such that $s\in\mathrm{int}_{\Gamma}(\Gamma_{\mathrm{N_S}})$. By considering the function $w:=u-\psi v_\tau\in\mathcal{K}^{0}(\Gamma,\R^d)$ in Inequality~\eqref{inegaliteDNS2}, where $\psi\in\LL^{2}(\Gamma)$ is defined by
$$
\psi:=
\left\{
\begin{array}{rl}
1   & \text{ on } B_{\Gamma}(s,\varepsilon) , \\
0  & \text{ on } \Gamma\textbackslash B_{\Gamma}(s,\varepsilon) ,
\end{array}
\right.
$$
for $\varepsilon>0$ such that $B_{\Gamma}(s,\varepsilon)\subset\Gamma_{\mathrm{N_S}}$, one gets that
$$
    \displaystyle \frac{1}{\left|B_{\Gamma}(s,\varepsilon)\right|}\int_{B_{\Gamma}(s,\varepsilon)}\sigma_{\tau}(u)\cdot v_\tau\leq\frac{1}{\left|B_{\Gamma}(s,\varepsilon)\right|}\int_{B_{\Gamma}(s,\varepsilon)}\ell\left\|v_\tau\right\|_{2}^{2},
$$
and thus $(\sigma_{\tau}(u)(s)-\ell (s)v_\tau(s))\cdot v_\tau(s)\leq 0$ by letting $\varepsilon\rightarrow0^{+}$. Moreover, since almost every point of $\Gamma_{\mathrm{N_S}}$ are in~$\mathrm{int}_{\Gamma}({\Gamma_{\mathrm{N_S}}})$ and are Lesbegue points of $\sigma_{\tau}(u)\cdot v_\tau\in\LL^{2}(\Gamma_{\mathrm{N_R}}\cup\Gamma_{\mathrm{N_S}})$ and of~$\ell\left\|v_\tau\right\|_{2}^{2}\in\LL^{2}(\Gamma_{\mathrm{N_R}}\cup\Gamma_{\mathrm{N_S}})$, one deduces 
$$
(\sigma_{\tau}(u)-\ell v_\tau)\cdot v_\tau\leq 0,
$$
\textit{a.e.} on $\Gamma_{\mathrm{N_S}}$. Finally, by considering $w=0$ and $w=2u$ in Inequa\-lity~\eqref{inegaliteDNS2},
one gets
$$
\displaystyle\int_{\Gamma_{\mathrm{N_S}}}u_\tau\cdot\left(\sigma_\tau(u)-\ell v_\tau\right)=0, 
$$
therefore $u_\tau\cdot(\sigma_\tau(u)-\ell v_\tau)=0$ \textit{a.e.} on $\Gamma_{\mathrm{N_S}}$ since $u\in\mathcal{K}^{1}(\Omega,\R^d)$. The proof is complete.
\end{proof}

Now let us prove that there exists a unique solution to the tangential Signorini problem~\eqref{PbtangSignorini}. To this aim let us introduce the functional $\Psi$ defined by
\begin{equation*}
\fonction{\Psi}{\HH^{1}_{\mathrm{D}}(\Omega,\R^d)}{\R}{w}{\displaystyle \Psi(w):=\int_{\Gamma_{\mathrm{N_R}}}\frac{k}{2}\left(\left\|w_\tau\right\|^{2}-\left|w_\tau\cdot v_\tau\right|^{2}\right) .}
\end{equation*}
Note that $\Psi$ is well defined since 
$k\in\LL^4(\Gamma_{\mathrm{N}_{\mathrm{R}}})$, $||v||\leq1$ \textit{a.e.} on $\Gamma_{\mathrm{N}_{\mathrm{R}}}$ and from the continuous embedding $\HH^{1}(\Omega,\R^d){\hookrightarrow} \LL^{4}(\Gamma,\R^d)$.

\begin{myLem}\label{fonctionnelleannexe}
The functional $\Psi$ is convex and Fréchet differentiable on $\HH^{1}_{\mathrm{D}}(\Omega,\R^d)$ and, for all~$w_{0}\in\HH^{1}_{\mathrm{D}}(\Omega,\R^d)$, $\nabla{\Psi(w_{0})}\in\HH^{1}_{\mathrm{D}}(\Omega,\R^d)$ is the unique solution to the Dirichlet-Neumann problem
\begin{equation}\label{DNproblem2}
\arraycolsep=2pt
\left\{
\begin{array}{rcll}
-\mathrm{div}(\mathrm{A}\mathrm{e}(\nabla{\Psi(w_{0})})) & = & 0   & \text{ in } \Omega , \\
\nabla{\Psi(w_{0})} & = & 0  & \text{ on } \Gamma_{\mathrm{D}} ,\\
\mathrm{A}\mathrm{e}(\nabla{\Psi(w_{0})})\nn & = & 0  & \text{ on } \Gamma_{\mathrm{N_T}}\cup\Gamma_{\mathrm{N_S}},\\
\sigma_\nn(\nabla{\Psi(w_{0})}) & = & 0  & \text{ on } \Gamma_{\mathrm{N_R}},\\
\sigma_\tau(\nabla{\Psi(w_{0})}) & = & k\left( w_{0_\tau}-\left(w_{0_\tau}\cdot v_{\tau}\right) v_{\tau} \right)  & \text{ on } \Gamma_{\mathrm{N_R}}.
\end{array}
\right.
\end{equation}
\end{myLem}

\begin{proof}
Let us start with the convexity of $\Psi$. Take $w_{1},w_{2}\in\HH^{1}_{\mathrm{D}}(\Omega,\R^d)$ and $\nu\in ( 0,1 )$. Then
\begin{multline*}
\Psi(\nu w_{1}+(1-\nu) w_{2})-\nu\Psi(w_1)-(1-\nu)\Psi(w_{2})=\\[10pt]
\int_{\Gamma_{\mathrm{N_R}}}-\frac{k}{2}\nu(1-\nu)[\left [\left\|w_{1_\tau}\right\|^{2}+\left\|w_{2_\tau}\right\|^{2}+2w_{1_\tau}\cdot w_{2_\tau}-\left|w_{1_\tau}\cdot v_\tau\right|^{2}-\left|w_{2_\tau}\cdot v_\tau\right|^{2}-2(w_{1_\tau}\cdot v_\tau)(w_{2_\tau}\cdot v_\tau)  \right ]\\=\int_{\Gamma_{\mathrm{N_R}}}-\frac{k}{2}\nu(1-\nu)\left\|w_{1_\tau}+w_{2_\tau}\right\|^{2}+\int_{\Gamma_{\mathrm{N_R}}}\frac{k}{2}\nu(1-\nu)\left|(w_{1_\tau}+w_{2_\tau})\cdot v_\tau\right|^{2}.
\end{multline*}
Since $k>0$ and $\left\|v\right\|\leq1$ \textit{a.e.} on $\Gamma_{\mathrm{N_R}}$, one deduces
\begin{multline*}
\Psi(\nu w_{1}+(1-\nu) w_{2})-\nu\Psi(w_1)-(1-\nu)\Psi(w_{2})\leq\\[10pt]
\int_{\Gamma_{\mathrm{N_R}}}-\frac{k}{2}\nu(1-\nu)\left\|w_{1_\tau}+w_{2_\tau}\right\|^{2}+\int_{\Gamma_{\mathrm{N_R}}}\frac{k}{2}\nu(1-\nu)\left\|w_{1_\tau}+w_{2_\tau}\right\|^{2}\left\|v_\tau\right\|^{2}\leq0.
\end{multline*}
Thus $\Psi$ is convex on $\HH^{1}_{\mathrm{D}}(\Omega,\R^d)$. Now let us prove that $\Psi$ is Fréchet differentiable. For $w_{0}\in\HH^{1}_{\mathrm{D}}(\Omega,\R^d)$ and $w\in\HH^{1}_{\mathrm{D}}(\Omega,\R^d)$, it holds that
\begin{equation*}
\Psi(w_0+w)=\Psi(w_0)+\int_{\Gamma_{\mathrm{N_R}}}k\left( w_{0_\tau}-\left(w_{0_\tau}\cdot v_{\tau}\right) v_{\tau} \right)\cdot w_{\tau}+\int_{\Gamma_{\mathrm{N_R}}}\frac{k}{2}\left(\left\|w_{\tau}\right\|^2-\left|w_\tau\cdot v_\tau\right|^2
\right).
\end{equation*}
Moreover one has
$$
\int_{\Gamma_{\mathrm{N_R}}}\frac{k}{2}\left(\left\|w_{\tau}\right\|^2-\left|w_\tau\cdot v_\tau\right|^2\right)=o(w),
$$
where $o$ stands for the standard Bachmann-Landau notation for the $\HH^{1}_{\mathrm{D}}(\Omega,\R^d)$-norm. Moreover the map
$$
\displaystyle w\in\HH^{1}_{\mathrm{D}}(\Omega,\R^d)\mapsto \int_{\Gamma_{\mathrm{N_R}}}k\left( w_{0_\tau}-\left(w_{0_\tau}\cdot v_{\tau}\right) v_{\tau} \right)\cdot w_{\tau}\in\R,
$$
is linear and continuous. Therefore $\Psi$ is Fréchet differentiable in $w_{0}\in\HH^{1}_{\mathrm{D}}(\Omega,\R^d)$ and
$$
\dual{\nabla{\Psi(w_{0})}}{w}_{\HH^{1}_{\mathrm{D}}(\Omega,\R^d)}=\int_{\Gamma_{\mathrm{N_R}}}k\left( w_{0_\tau}-\left(w_{0_\tau}\cdot v_{\tau}\right) v_{\tau} \right)\cdot w_{\tau}, \qquad \forall w\in\HH^{1}_{\mathrm{D}}(\Omega,\R^d).
$$
In other words $\nabla{\Psi(w_{0})}\in\HH^{1}_{\mathrm{D}}(\Omega,\R^d)$ is the unique solution to the Dirichlet-Neumann problem~\eqref{DNproblem2}. The proof is complete.
\end{proof}

\begin{myProp}\label{existenceSignorinitang}
The tangential Signorini problem~\eqref{PbtangSignorini} admits a unique weak solution $u\in\HH^{1}_{\mathrm{D}}(\Omega,\R^d)$ which is given by
$$
    \displaystyle u=\mathrm{prox}_{\Psi+\iota_{\mathcal{K}^{1}(\Omega,\R^d)}}(F),
$$
where $F\in\HH^{1}_{\mathrm{D}}(\Omega,\R^d)$ is the unique solution to the Dirichlet-Neumann problem~\eqref{PbNeumannDirichlet} with $z:=h\nn+\ell v_\tau\in\LL^2(\Gamma_\mathrm{N},\R^d)$, and $\mathrm{prox}_{\Psi+\iota_{\mathcal{K}^{1}(\Omega,\R^d)}}$ stands for the proximal operator associated with the functional $\Psi+\iota_{\mathcal{K}^{1}(\Omega,\R^d)}$.
\end{myProp}

\begin{proof}
Let $F\in\HH^{1}_{\mathrm{D}}(\Omega,\R^d)$ be the solution to the Dirichlet-Neumann problem~\eqref{PbNeumannDirichlet} with~$z:=h\nn+\ell v_\tau\in\LL^2(\Gamma_\mathrm{N},\R^d)$. Then
$$
\dual{F}{w}_{\HH^{1}_{\mathrm{D}}(\Omega,\R^d)}=\int_{\Omega}f\cdot w+\int_{\Gamma_{\mathrm{N}}}hw_{\nn}+\int_{\Gamma_{\mathrm{N}}}\ell v_\tau\cdot w_\tau, \qquad \forall w\in\HH^{1}_{\mathrm{D}}(\Omega,\R^d).
$$
Let $u\in\HH^{1}_{\mathrm{D}}(\Omega,\R^d)$ and note that $\Psi+\iota_{\mathcal{K}^{1}(\Omega,\R^d)}$ is a proper lower semi-continuous convex function on $\HH^{1}_{\mathrm{D}}(\Omega,\R^d)$. Then $u$ is the weak solution to the tangentiel Signorini problem~\eqref{PbtangSignorini} if and only if~$u\in\mathcal{K}^{1}(\Omega,\R^d)$ and
\begin{multline*}
\dual{u}{w-u}_{\HH^{1}_{\mathrm{D}}(\Omega,\R^d)}\geq\int_{\Omega}f\cdot \left(w-u\right)+\int_{\Gamma_{\mathrm{N}}}h\left(w_{\nn}-u_\nn\right)\\+\int_{\Gamma_{\mathrm{N_R}}}\left(\ell v_\tau-k\left( u_{\tau}-\left(u_{\tau}\cdot v_{\tau}\right) v_{\tau}\right) \right)\cdot\left(w_{\tau}-u_\tau\right)+\int_{\Gamma_{\mathrm{N_S}}}\ell v_\tau\cdot \left(w_\tau-u_\tau\right), \qquad \forall w\in\mathcal{K}^{1}(\Omega,\R^d),
\end{multline*}
i.e. if and only if
\begin{equation*}
\int_{\Gamma_{\mathrm{N_R}}}k\left( u_{\tau}-\left(u_{\tau}\cdot v_{\tau}\right) v_{\tau} \right)\cdot\left(w_{\tau}-u_\tau\right)\geq\dual{F-u}{w-u}_{\HH^{1}_{\mathrm{D}}(\Omega,\R^d)},\qquad\forall w\in\mathcal{K}^{1}(\Omega,\R^d),
\end{equation*}
i.e. if and only if (see~Proposition~\ref{gradientequi})
\begin{equation*}
\Psi(w)-\Psi(u)\geq \dual{F-u}{w-u}_{\HH^{1}_{\mathrm{D}}(\Omega,\R^d)}, \qquad\forall w\in\mathcal{K}^{1}(\Omega,\R^d),
\end{equation*}
i.e. if and only if
$$
\dual{F-u}{w-u}_{\HH^{1}_{\mathrm{D}}(\Omega,\R^d)}\leq \Psi(w)-\Psi(u)+\iota_{\mathcal{K}^{1}(\Omega,\R^d)}(w)-\iota_{\mathcal{K}^{1}(\Omega,\R^d)}(u), \qquad\forall w\in\HH^{1}_{\mathrm{D}}(\Omega,\R^d),
$$
i.e. if and only if
$
F-u\in\partial{\left(\Psi+\iota_{\mathcal{K}^{1}(\Omega,\R^d)}\right)}(u)
$, 
i.e. if and only if
$
u=\mathrm{prox}_{\Psi+\iota_{\mathcal{K}^{1}(\Omega,\R^d)}}(F)
$,
which concludes the proof.
\end{proof}

\subsubsection{A Tresca friction problem}\label{subtresca}
Let $h\in \mathrm{L}^{2}(\Gamma_{\mathrm{N}})$ and $g\in \mathrm{L}^{2}(\Gamma_{\mathrm{N}})$ such that $g>0$ \textit{a.e.} on $\Gamma_{\mathrm{N}}$. Consider the Tresca friction problem given by
\begin{equation}\tag{TP}\label{PbTresca}
\arraycolsep=2pt
\left\{
\begin{array}{rcll}
-\mathrm{div}(\mathrm{A}\mathrm{e}(u)) & = & f   & \text{ in } \Omega , \\
u & = & 0  & \text{ on } \Gamma_{\mathrm{D}} ,\\
\sigma_\nn(u) & = & h  & \text{ on } \Gamma_{\mathrm{N}},\\
\left\|\sigma_\tau(u)\right\|\leq g \text{ and } u_{\tau}\cdot\sigma_{\tau}(u)+g\left\|u_{\tau}\right\| & = & 0  & \text{ on } \Gamma_{\mathrm{N}}.
\end{array}
\right.
\end{equation}

\begin{myDefn}[Strong solution to the Tresca friction problem]
A (strong) solution to the Tresca friction problem~\eqref{PbTresca} is a function $u\in\HH^{1}(\Omega,\R^d)$ such that $-\mathrm{div}(\mathrm{A}\mathrm{e}(u))=f$ in $\MD'(\Omega,\R^d)$,~$u=0$ \textit{a.e.} on $\Gamma_{\mathrm{D}}$, $\mathrm{A}\mathrm{e}(u)\nn\in\LL^{2}(\Gamma_{\mathrm{N}},\R^d)$ with $\sigma_\nn(u)=h$, $\left\|\sigma_\tau(u)\right\|\leq g$ and $u_{\tau}\cdot\sigma_{\tau}(u)+g\left\|u_{\tau}\right\|=0$ \textit{a.e.} on~$\Gamma_{\mathrm{N}}$.
\end{myDefn}

\begin{myDefn}[Weak solution to the Tresca friction problem]
A weak solution to the Tresca friction problem~\eqref{PbTresca} is a function $u\in\HH^{1}_{\mathrm{D}}(\Omega,\R^d)$ such that
\begin{multline*}
\displaystyle\int_{\Omega}\mathrm{A}\mathrm{e}(u):\mathrm{e}(w-u)+\int_{\Gamma_{\mathrm{N}}}g\left\|w_\tau\right\|-\int_{\Gamma_{\mathrm{N}}}g\left\|u_\tau\right\| \geq\int_{\Omega}f\cdot\left(w-u\right)\\+\int_{\Gamma_{\mathrm{N}}}h\left(w_\nn-u_\nn\right), \qquad \forall w\in\HH^{1}_{\mathrm{D}}(\Omega,\R^d).
\end{multline*}
\end{myDefn}

\begin{myProp}\label{Trescaequivalent}
A function $u\in \HH^{1}(\Omega,\R^d)$ is a (strong) solution to the Tresca friction problem~\eqref{PbTresca} if and only if $u$ is a weak solution to the Tresca friction problem~\eqref{PbTresca}.
\end{myProp}

From definition of the proximal operator (see Definition~\ref{proxi}), one deduces the following existence/uniqueness result.
\begin{myProp}\label{existenceuniciteTresca}
The Tresca friction problem~\eqref{PbTresca} admits a unique (strong) solution $u\in\HH^{1}_{\mathrm{D}}(\Omega,\R^d)$ given by
$$
\displaystyle u=\mathrm{prox}_{\phi}(F),
$$
where $F\in\HH^{1}_{\mathrm{D}}(\Omega,\R^d)$ is the solution to the Dirichlet-Neumann problem~\eqref{PbNeumannDirichlet} with $z:=h\nn\in\LL^{2}(\Gamma_{\mathrm{N}},\R^d)$, and where $\mathrm{prox}_{\phi}$ stands for the proximal operator associated with the Tresca friction functional $\phi$ defined by 
\begin{equation*}
\displaystyle\fonction{\phi}{\HH^{1}_{\mathrm{D}}(\Omega,\R^d)}{\R}{w}{\displaystyle \phi(w):=\int_{\Gamma_{\mathrm{N}}}g\left\|w_\tau\right\|.}
\end{equation*}
\end{myProp}

\begin{myRem}\normalfont\label{interpot}
The assumption that almost every point of $\Gamma_{\mathrm{N}}$ is in $\mathrm{int}_{\Gamma}({\Gamma_{\mathrm{N}}})$ is only used to prove that a weak solution to the Tresca friction problem~\eqref{PbTresca} is also a (strong) solution (precisely to get the Tresca friction law pointwisely on~$\Gamma_{\mathrm{N}}$). Of course, some sets do not satisfy this assumption, for instance the well-known Smith–Volterra–Cantor set (see, e.g,~\cite[Example 6.15 Section 6 Chapter 1]{ALIP}). Nevertheless it is trivially satisfied in most of standard cases found in practice. Furthermore, if this assumption is not satisfied, one can also prove that the weak solution to the Tresca friction problem~\eqref{PbTresca} is a (strong) solution by adding the assumption that~$g\in\LL^{\infty}(\Gamma_{\mathrm{N}})$, and by using the isometry between the dual of~$(\LL^{1}(\Gamma_{\mathrm{N}},\R^d), \| \cdot  \|_{\LL^{1}(\Gamma_{\mathrm{N}},\R^d)_{g}})$ and~$\LL^{\infty}(\Gamma_{\mathrm{N}},\R^d)$ (with its standard norm~$\| \cdot \|_{\LL^{\infty}(\Gamma_{\mathrm{N}},\R^d)}$) where~$\| \cdot \|_{\LL^{1}(\Gamma_{\mathrm{N}},\R^d)_{g}}$ is the norm defined by
$$
\fonction{\left \| \cdot \right \|_{\LL^{1}(\Gamma_{\mathrm{N}},\R^d)_{g}}}{\LL^{1}(\Gamma_{\mathrm{N}},\R^d)}{\R}{w}{\displaystyle\int_{\Gamma_{\mathrm{N}}}g\left\| w \right \|.}
$$
We refer to~\cite[Chapitre 3]{DUVAUTLIONS} for details in a similar context.
\end{myRem}

\subsection{Sensitivity analysis of the Tresca friction problem}\label{section4}
In this section we perform the sensitivity analysis of the Tresca friction problem. To this aim we consider the parameterized Tresca friction problem given by

 \begin{equation}\tag{TP\ensuremath{_{t}}} \label{PbNeumannDirichletTrescaPara}\arraycolsep=2pt
\left\{
\begin{array}{rcll}
-\mathrm{div}(\mathrm{A}\mathrm{e}(u_{t})) & = & f_{t}   & \text{ in } \Omega , \\
u_{t} & = & 0  & \text{ on } \Gamma_{\mathrm{D}} ,\\
\sigma_\nn(u_{t}) & = & h_{t}  & \text{ on } \Gamma_{\mathrm{N}},\\
\left\|\sigma_\tau(u_{t})\right\|\leq g_{t} \text{ and } u_{t_\tau}\cdot\sigma_{\tau}(u_{t})+g_t\left\|u_{t_\tau}\right\| & = & 0  & \text{ on } \Gamma_{\mathrm{N}},
\end{array}
\right.
\end{equation}
where $f_{t}\in\LL^{2}(\Omega,\R^d)$, $h_{t}\in\LL^{2}(\Gamma_{\mathrm{N}})$ and $g_{t}\in\LL^{2}(\Gamma_{\mathrm{N}})$ such that~$g_{t}>0$ \textit{a.e.} on $\Gamma_{\mathrm{N}}$, for all $t\geq0$. 

\subsubsection{Parameterized Tresca friction functional and twice epi-differentiability}\label{sectiontwiceepi}
Let us introduce the parameterized Tresca friction functional given by
\begin{equation}\label{fonctionnelledeTrescaparacas2}
\displaystyle\fonction{\Phi}{\mathbb{R}_{+}\times \HH^{1}_{\mathrm{D}}(\Omega,\R^d)}{\R}{(t,w)}{\displaystyle \Phi(t,w):=\int_{\Gamma_{\mathrm{N}}}g_{t}\left\|w_\tau\right\|.}
\end{equation}
From Proposition~\ref{existenceuniciteTresca}, the unique solution to the parameterized Tresca friction problem~\eqref{PbNeumannDirichletTrescaPara} is given by
\begin{equation*}
   \displaystyle u_{t}=\mathrm{prox}_{\Phi(t,\mathord{\cdot})}(F_{t}),
\end{equation*}
where $F_{t}$ is the unique solution to the parameterized Dirichlet-Neumann problem
\begin{equation}\tag{DN\ensuremath{_{t}}} \label{PbNeumannDirichletPara}
\arraycolsep=2pt
\left\{
\begin{array}{rcll}
-\mathrm{div}(\mathrm{A}\mathrm{e}(F_{t})) & = & f_{t}   & \text{ in } \Omega , \\
F_{t} & = & 0  & \text{ on } \Gamma_{\mathrm{D}} ,\\
\mathrm{A}\mathrm{e}(F_{t})\nn & = & h_{t}\nn  & \text{ on } \Gamma_{\mathrm{N}},
\end{array}
\right.
\end{equation}
for all $t\geq0$.
Similarly to the scalar case (see~\cite{BCJDC}), since the parameterized Tresca friction functional depends on a parameter $t\geq0$, we have to use the notion of twice epi-differentiability depending on a parameter (see Definition~\ref{epidiffpara}), in order to apply Theorem~\ref{TheoABC2018}.
Let us prepare the background for the twice epi-differentiability of the parameterized Tresca friction functional. More specifically, let us start with the characterization of the convex subdifferential of $\Phi(0,\cdot)$ (see Definition~\ref{sousdiff}). To this aim, for all $s\in\Gamma_{\mathrm{N}}$, we introduce the tangential norm map defined by
\begin{equation*}
\displaystyle\fonction{\left\|\cdot_{\tau(s)}\right\|}{\R^d}{\R}{x}{\displaystyle \left\|x_{\tau(s)}\right\|,}
\end{equation*}
and we introduce an auxiliary problem defined, for all~$u\in\HH^{1}_{\mathrm{D}}(\Omega,\R^d)$, by
\begin{equation}\tag{AP\ensuremath{_{u}}}\label{PbannexesousdiffDNT}
\arraycolsep=2pt
\left\{
\begin{array}{rl}
-\mathrm{div}(\mathrm{A}\mathrm{e}(v)) =  0   & \text{ in } \Omega , \\
v  =  0 & \text{ on } \Gamma_{\mathrm{D}} , \\
\sigma_{\nn}(v)  =  0 & \text{ on } \Gamma_{\mathrm{N}} , \\
\sigma_{\tau}(v)(s) \in  g_{0}(s)\partial \left\|\mathord{\cdot}_{\tau(s)}\right\|(u(s))  & \text{ on } \Gamma_{\mathrm{N}},\\
\end{array}
\right.
\end{equation}
where, for almost all $s\in\Gamma_{\mathrm{N}}$, $\partial\left\|\cdot_{\tau(s)}\right\|(u(s))$ stands for the convex subdifferential of the tangential norm map $\left\|\cdot_{\tau(s)}\right\|$ at $u(s)\in\R^{d}$.
For a given $u\in \HH^{1}_{\mathrm{D}}(\Omega,\R^d)$, a solution to this problem~\eqref{PbannexesousdiffDNT} is a function~$v\in\HH^{1}(\Omega,\R^d)$ such that $-\mathrm{div}(\mathrm{A}\mathrm{e}(v))=0$ in~$\mathcal{D}'
(\Omega,\R^d)$, $v=0$ \textit{a.e.} on $\Gamma_{\mathrm{D}}$,~$\mathrm{A}\mathrm{e}(v)\nn\in\mathrm{L}^{2}(\Gamma_{\mathrm{N}},\R^d)$ with $\sigma_{\nn}(v)=0$ \textit{a.e.} on $\Gamma_{\mathrm{N}}$ and $\sigma_{\tau}(v)(s)\in g_{0}(s)\partial ||\mathord{\cdot}_{\tau(s)}||(u(s))$ for almost all $s\in\Gamma_{\mathrm{N}}$.

\begin{myLem}\label{lemmeAnnexeSousdiff}
Let $u\in \HH^{1}_{\mathrm{D}}(\Omega,\R^d)$. Then

\begin{center}
    $\partial \Phi(0,\cdot)(u)=$ the set of solutions to Problem~\eqref{PbannexesousdiffDNT}.
\end{center}
\end{myLem}

\begin{proof}
Let $u\in \HH^{1}_{\mathrm{D}}(\Omega,\R^d)$ and let us prove the two inclusions. Firstly, let $v\in\HH^{1}(\Omega,\R^d)$ be a solution to Problem~\eqref{PbannexesousdiffDNT}. Then $v\in\HH^{1}_{\mathrm{D}}(\Omega,\R^d)$, $\mathrm{A}\mathrm{e}(v)\nn\in\mathrm{L}^{2}(\Gamma_{\mathrm{N}},\R^d)$ with~$\sigma_{\tau}(v)(s)\in g_{0}(s)\partial ||\mathord{\cdot}_{\tau(s)}||(u(s))$ for almost all $s\in\Gamma_{\mathrm{N}}$. Hence one has
$$
\displaystyle \sigma_{\tau}(v)(s)\cdot\left(w_\tau(s)-u_{\tau}(s)\right)\leq g_{0}(s)(\left\|w_\tau(s)\right\|-\left\|u_\tau(s)\right\|),
$$
for all $w\in \HH^{1}_{\mathrm{D}}(\Omega,\R^d)$ and for almost all $s\in\Gamma_{\mathrm{N}}$. It follows that
$$
\displaystyle\int_{\Gamma_{\mathrm{N}}}\displaystyle \sigma_{\tau}(v)\cdot\left(w_\tau-u_{\tau}\right)\leq \int_{\Gamma_{\mathrm{N}}}g_{0}\left\|w_\tau\right\|-\int_{\Gamma_{\mathrm{N}}}g_{0}\left\|u_\tau\right\|,
$$
for all $w\in\HH^{1}_{\mathrm{D}}(\Omega,\R^d)$. Moreover $-\mathrm{div}(\mathrm{A}\mathrm{e}(v))=0$ in~$\mathcal{D}'
(\Omega,\R^d)$, thus it holds $-\mathrm{div}(\mathrm{A}\mathrm{e}(v))=0$ in~$\LL^{2}(\Omega,\R^d)$. Hence, from divergence formula (see Proposition~\ref{div}), one gets
$$
\dual{v}{w-u}_{\HH^{1}_{\mathrm{D}}(\Omega,\R^d)}=\dual{\mathrm{A}\mathrm{e}(v)\nn}{w-u}_{\HH^{-1/2}_{00}(\Gamma_{\mathrm{N}},\R^d)\times \HH^{1/2}_{00}(\Gamma_{\mathrm{N}},\R^d)},
$$
for all $w\in\HH^{1}_{\mathrm{D}}(\Omega,\R^d)$. Since $\mathrm{A}\mathrm{e}(v)\nn\in\mathrm{L}^{2}(\Gamma_{\mathrm{N}},\R^d)$ and $\sigma_{\nn}(v)=0$ \textit{a.e.} on $\Gamma_{\mathrm{N}}$, one deduces that
$$
\dual{v}{w-u}_{\HH^{1}_{\mathrm{D}}(\Omega,\R^d)}=\int_{\Gamma_{\mathrm{N}}}\sigma_\tau(v)\cdot\left(w_\tau-u_\tau\right),
$$
for all $w\in\HH^{1}_{\mathrm{D}}(\Omega,\R^d)$. Therefore it follows that
$$
\dual{v}{w-u}_{\HH^{1}_{\mathrm{D}}(\Omega,\R^d)}\leq
\int_{\Gamma_{\mathrm{N}}}g_{0}\left\|w_\tau\right\|-\int_{\Gamma_{\mathrm{N}}}g_{0}\left\|u_\tau\right\|,
$$
for all $w\in\HH^{1}_{\mathrm{D}}(\Omega,\R^d)$. Thus $v\in\partial\Phi(0,\cdot)(u)$ and the first inclusion is proved. Conversely let~$v\in\partial\Phi(0,\cdot)(u)$. Then one has
\begin{equation}\label{inegalitesousdiffDNT}
\displaystyle\dual{v}{w-u}_{\HH^{1}_{\mathrm{D}}(\Omega,\R^d)}\leq\int_{\Gamma_{\mathrm{N}}}g_{0}\left\|w_\tau\right\|-\int_{\Gamma_{\mathrm{N}}}g_{0}\left\|u_\tau\right\|,
\end{equation}
for all $w\in\HH^{1}_{\mathrm{D}}(\Omega,\R^d)$. Considering the function $w:=u\pm\psi\in\HH^{1}_{\mathrm{D}}(\Omega,\R^d)$ with any function~$\psi\in\MD(\Omega,\R^d)$, one deduces from Inequality~\eqref{inegalitesousdiffDNT} that $-\mathrm{div}(\mathrm{A}\mathrm{e}(v))=0$ in~$\MD'(\Omega,\R^d)$, thus it holds $-\mathrm{div}(\mathrm{A}\mathrm{e}(v))=0$ in $\LL^{2}(\Omega,\R^d)$. Hence, from divergence formula (see Proposition~\ref{div}) and Inequality~\eqref{inegalitesousdiffDNT}, it follows that
$$
\dual{\mathrm{A}\mathrm{e}(v)\nn}{w-u}_{\HH^{-1/2}_{00}(\Gamma_{\mathrm{N}},\R^d)\times \HH^{1/2}_{00}(\Gamma_{\mathrm{N}},\R^d)}\leq\int_{\Gamma_{\mathrm{N}}}g_{0}\left\|w_\tau\right\|-\int_{\Gamma_{\mathrm{N}}}g_{0}\left\|u_\tau\right\|,
$$
for all $w\in\HH^{1}_{\mathrm{D}}(\Omega,\R^d)$, and thus also for all $w\in\HH^{1/2}_{00}(\Gamma_{\mathrm{N}},\R^d)$. Now, by considering~$w:=u+\varphi\in\HH^{1/2}_{00}(\Gamma_{\mathrm{N}},\R^d)$, for any $\varphi\in\HH^{1/2}_{00}(\Gamma_{\mathrm{N}},\R^d)$, one gets
$$
\dual{\mathrm{A}\mathrm{e}(v)\nn}{\varphi}_{\HH^{-1/2}_{00}(\Gamma_{\mathrm{N}},\R^d)\times \HH^{1/2}_{00}(\Gamma_{\mathrm{N}},\R^d)}\leq \int_{\Gamma_{\mathrm{N}}}g_{0}\left\|\varphi_\tau\right\|\leq\left \| g_{0} \right \|_{\LL^{2}(\Gamma_{\mathrm{N}})}\left \| \varphi \right \|_{\LL^{2}(\Gamma_{\mathrm{N}},\R^d)},
$$
for all $\varphi\in\HH^{1/2}_{00}(\Gamma_{\mathrm{N}},\R^d)$.
From Proposition~\ref{Ident}, one deduces that $\mathrm{A}\mathrm{e}(v)\nn\in \mathrm{L}^{2}(\Gamma_{\mathrm{N}},\R^d)$ and also that
\begin{multline}
\label{inegalitesousdiffDNTsensi}
\displaystyle\int_{\Gamma_{\mathrm{N}}}\mathrm{A}\mathrm{e}(v)\nn\cdot \left(w-u\right)=\int_{\Gamma_{\mathrm{N}}}\sigma_{\tau}(v)\cdot\left(w_\tau-u_\tau\right)+\int_{\Gamma_{\mathrm{N}}}\sigma_\nn(v)\left(w_{\nn}-u_{\nn}\right)\\\leq
\int_{\Gamma_{\mathrm{N}}}g_{0}\left\|w_\tau\right\|-\int_{\Gamma_{\mathrm{N}}}g_{0}\left\|u_\tau\right\|, 
\end{multline}
for all $w\in\HH^{1/2}_{00}(\Gamma_{\mathrm{N}},\R^d)$,
and thus for all $w\in\LL^{2}(\Gamma_{\mathrm{N}},\R^d)$ by density.
By considering~$w:=u\pm\psi\nn\in\LL^{2}(\Gamma_{\mathrm{N}},\R^d)$ in Inequality~\eqref{inegalitesousdiffDNTsensi}, for any $\psi\in\LL^{2}(\Gamma_{\mathrm{N}})$, one gets
$$
\int_{\Gamma_{\mathrm{N}}}\sigma_\nn(v)\psi=0.
$$
Therefore $\sigma_{\nn}(v)=0$ \textit{a.e.} on $\Gamma_{\mathrm{N}}$ and Inequality~\eqref{inegalitesousdiffDNTsensi} becomes
\begin{equation}\label{secondeinegalite}
\int_{\Gamma_{\mathrm{N}}}\sigma_{\tau}(v)\cdot\left(w_\tau-u_\tau\right)\leq
\int_{\Gamma_{\mathrm{N}}}g_{0}\left\|w_\tau\right\|-\int_{\Gamma_{\mathrm{N}}}g_{0}\left\|u_\tau\right\|, 
\end{equation}
for all $w\in\LL^{2}(\Gamma_{\mathrm{N}},\R^d)$. Now let $s_{0}\in\Gamma_{\mathrm{N}}$ be a Lebesgue point of~$(\sigma_{\tau}(v))_{i}\in\LL^{2}(\Gamma_{\mathrm{N}})$ for $i\in\left[\left[1,d\right] \right]$,~$\sigma_\tau(v)\cdot u_\tau\in\LL^1(\Gamma_{\mathrm{N}})$, $g_{0}\in\LL^{2}(\Gamma_{\mathrm{N}})$ and of~$g_{0}\left\|u_\tau\right\|\in\LL^{1}(\Gamma_{\mathrm{N}})$, such that $s_{0}\in\mathrm{int}_{\Gamma}({\Gamma_{\mathrm{N}}})$. Let us consider the function $w\in \mathrm{L}^{2}(\Gamma_{\mathrm{N}},\R^d)$ defined by
$$w:=
\left\{
\begin{array}{rl}
x   & \text{ on } B_{\Gamma}(s_{0},\varepsilon) , \\
u   & \text{ on } \Gamma_{\mathrm{N}}\textbackslash B_{\Gamma}(s_{0},\varepsilon) ,
\end{array}
\right.
$$
with $x\in\R^{d}$ and $\varepsilon>0$ such that $B_{\Gamma}(s_{0},\varepsilon)\subset\Gamma_{\mathrm{N}}$. Then one has from Inequality~\eqref{secondeinegalite}
$$
\displaystyle \frac{1}{\left|B_{\Gamma}(s_{0},\varepsilon)\right|}\int_{B_{\Gamma}(s_{0},\varepsilon)}\sigma_{\tau}(v)\cdot(x_\tau-u_\tau)\leq\frac{1}{\left|B_{\Gamma}(s_{0},\varepsilon)\right|}\int_{B_{\Gamma}(s_{0},\varepsilon)}g_{0}\left\|x_\tau\right\|-\frac{1}{\left|B_{\Gamma}(s_{0},\varepsilon)\right|}\int_{B_{\Gamma}(s_{0},\varepsilon)}g_{0}\left\|u_\tau\right\|.
$$
The map $s\in\Gamma\mapsto \left\|x_{\tau(s)}\right\|\in\mathbb{R}_{+}$ is continuous since $\nn\in\mathcal{C}^{0}(\Gamma)$, thus $s_{0}$ is a Lebesgue point of~$g_{0}\left\|x_{\tau}\right\|\in\LL^{2}(\Gamma_{\mathrm{N}})$, then $\sigma_{\tau}(v)(s_{0})\cdot(x_{\tau(s_{0})}-u_\tau(s_{0}))\leq g_{0}(s_{0})\left\|x_{\tau(s_0)}\right\|-g_{0}(s_{0})\left\|u_\tau(s_{0})\right\|$ by letting~$\varepsilon\rightarrow0^{+}$. This inequality is true for any~$x \in \R^{d}$, therefore $\sigma_{\tau}(v)(s_{0})\in g_{0}(s_{0})\partial ||\mathord{\cdot}_{\tau(s_{0})}||(u(s_{0}))$.
Moreover, almost every point of $\Gamma_{\mathrm{N}}$ are in~$\mathrm{int}_{\Gamma}({\Gamma_{\mathrm{N}}})$ and are Lesbegue points of $(\sigma_{\tau}(v))_{i}\in\LL^{2}(\Gamma_{\mathrm{N}})$ for $i\in\left[\left[1,d\right] \right]$, $\sigma_\tau(v)\cdot u_\tau\in\LL^1(\Gamma_{\mathrm{N}})$, $g_{0}\in\LL^{2}(\Gamma_{\mathrm{N}})$ and of~$g_{0}\left\|u_\tau\right\|\in\LL^{1}(\Gamma_{\mathrm{N}})$, hence one deduces 
$$
\sigma_{\tau}(v)(s)\in g_{0}(s)\partial ||\mathord{\cdot}_{\tau(s)}||(u(s)),
$$
for almost all $s\in\Gamma_{\mathrm{N}}$, and this proves the second inclusion.
\end{proof}

\begin{myRem}\normalfont\label{regularityofn}
    As one can see in the proof of Lemma~\ref{lemmeAnnexeSousdiff}, the assumption that $\Gamma$ is of class~$\mathcal{C}^1$ is only used to ensure that $\nn\in \mathcal{C}^{0}(\Gamma)$, and thus to characterize the convex subdifferential of~$\Phi(0,\cdot)$ as the set of solutions to Problem~\eqref{PbannexesousdiffDNT}.
\end{myRem}

Since the twice epi-differentiability is defined using the second-order difference quotient functions, let us compute the second-order difference quotient functions of $\Phi$ at $u\in\HH^{1}_{\mathrm{D}}(\Omega,\R^d)$ for~$v\in\partial\Phi(0,\cdot)(u)$.

\begin{myProp}\label{epidiffoffunctionG}
For all $t>0$, all $u\in \HH^{1}_{\mathrm{D}}(\Omega,\R^d)$ and all $v\in\partial\Phi(0,\cdot)(u)$, it holds that
\begin{equation}\label{Delta2}
      \displaystyle\Delta_{t}^{2}\Phi(u\mid v)(w)=\int_{\Gamma_{\mathrm{N}}}\Delta_{t}^{2}G(s)(u(s)\mid\sigma_{\tau}(v)(s))(w(s)) \, \mathrm{d}s,
\end{equation}
for all $w\in \HH^{1}_{\mathrm{D}}(\Omega,\R^d)$, where, for almost all $s\in\Gamma_{\mathrm{N}}$, $\Delta_{t}^{2}G(s)(u(s)\mid\sigma_{\tau}(v)(s))$ stands for the second-order difference quotient function of $G(s)$ at $u(s)\in\R^{d}$ for $\sigma_{\tau}(v)(s)\in g_{0}(s)\partial ||\mathord{\cdot}_{\tau(s)}||(u(s))$, with~$G(s)$ defined by
$$ 
\fonction{G(s)}{\mathbb{R}_{+}\times\mathbb{R}^{d}}{\R}{(t,x)}{G(s)(t,x):=g_{t}(s)\left\|x_{\tau(s)}\right\|.}
$$
\end{myProp}
\begin{myRem}\normalfont
Note that, for almost all $s\in\Gamma_{\mathrm{N}}$ and all $t\geq 0$, $G(s)(t,\cdot):=g_{t}(s)\left\|\cdot_{\tau(s)}\right\|$ is a proper lower semi-continuous convex function on $\mathbb{R}^{d}$. Moreover, since $g_{0}>0$ \textit{a.e.} on~$\Gamma_{\mathrm{N}}$, it follows that
$$
\partial\left[G(s)(0,\mathord{\cdot})\right](x)=g_{0}(s)\partial{||\mathord{\cdot}_{\tau(s)}||}(x),
$$
for all $x\in\R^{d}$ and for almost all $s\in\Gamma_{\mathrm{N}}$.
\end{myRem}

\begin{proof}[Proof of Proposition~\ref{epidiffoffunctionG}]
Let $t>0$, $u\in \HH^{1}_{\mathrm{D}}(\Omega,\R^d)$ and $v\in\partial\Phi(0,\cdot)(u)$. From Lemma~\ref{lemmeAnnexeSousdiff} and the divergence formula (see Proposition~\ref{div}), one deduces that
$$
\displaystyle \dual{ v}{w}_{\HH^{1}_{\mathrm{D}}(\Omega,\R^d)}=\int_{\Gamma_{\mathrm{N}}}\sigma_{\tau}(v)\cdot w,
$$
for all $w\in\HH^{1}_{\mathrm{D}}(\Omega,\R^d)$. It follows that
$$
        \displaystyle \Delta_{t}^{2}\Phi(u\mid v)(w)=\int_{\Gamma_{\mathrm{N}}}\frac{g_{t}(s)\left\|u_\tau(s)+t w_\tau(s)\right\|-g_{t}(s)\left\|u_\tau(s)\right\|-t\sigma_{\tau}(v)(s)\cdot w(s)}{t^{2}} \, \mathrm{d}s,
$$
for all $w\in\HH^{1}_{\mathrm{D}}(\Omega,\R^d)$. Moreover, since $\sigma_{\tau}(v)(s)\in g_{0}(s)\partial{||\mathord{\cdot}_{\tau(s)}||}(u(s))$ for almost all $s\in\Gamma_{\mathrm{N}}$, one deduces that
$$
\displaystyle \Delta_{t}^{2}\Phi(u\mid v)(w)=\int_{\Gamma_{\mathrm{N}}}\Delta_{t}^{2}G(s)(u(s)\mid\sigma_{\tau}(v)(s))(w(s)) \, \mathrm{d}s,
$$
for all $w\in \HH^{1}_{\mathrm{D}}(\Omega,\R^d)$, which concludes the proof.
\end{proof}

From Proposition~\ref{epidiffoffunctionG}, it is clear that the twice epi-differentiability of the parameterized Tresca friction functional~$\Phi$ is related to the twice epi-differentiability of the parameterized function $G(s)$. Hence we have to compute the second-order epi-derivative of $G(s)$ for almost all~$s\in\Gamma_{\mathrm{N}}$. To this aim, let us start with the investigation of the twice epi-differentiability of the tangential norm map. Since, for almost all~$s\in\Gamma_{\mathrm{N}}$, $\left\|\cdot_{\tau(s)}\right\|=\xi_{\mathrm{\overline{\mathrm{B}(0,1)}\cap\left(\R\nn(s)\right)^{\perp}}}$, where $\xi_{\mathrm{\overline{\mathrm{B}(0,1)}\cap\left(\R\nn(s)\right)^{\perp}}}$ is the support function of $\mathrm{\overline{\mathrm{B}(0,1)}\cap\left(\R\nn(s)\right)^{\perp}}$ (see Definition~\ref{supportfunc}), we have to extend a result proved in~\cite[Example 2.7 p.286]{DO} about the twice epi-differentiability of a support function.

\begin{myProp}\label{epitangnorm}
Let $\xi_{\mathrm{C}}$ be the support function of a nonempty closed convex subset $\mathrm{C}$ of~$\R^d$.
Then, for all $x\in\mathrm{C}^{\perp}$, one has
$
\partial{\xi_{\mathrm{C}}}(x)=\mathrm{C}
$
and $\xi_{\mathrm{C}}$ is twice epi-differentiable at $x$ for any $y\in\mathrm{C}$ with
$$
\mathrm{d}_{e}^{2}\xi_{\mathrm{C}}(x \mid y)=\iota_{\mathrm{N}_{\mathrm{C}}(y)},
$$
where $\mathrm{N}_{\mathrm{C}}(y):=\left\{z\in\R^d\mid z\cdot(c-y)\leq0, \forall c\in\mathrm{C}\right\}$ is the \textit{normal cone} to $\mathrm{C}$ at $y\in\mathrm{C}$ and $\iota_{\mathrm{N}_{\mathrm{C}}(y)}$ stands for the indicator function of $\mathrm{N}_{\mathrm{C}}(y)$ defined by $\iota_{\mathrm{N}_{\mathrm{C}}(y)}(z):=0$ if $z\in\mathrm{N}_{\mathrm{C}}(y)$, and~$\iota_{\mathrm{N}_{\mathrm{C}}(y)}(z):=+\infty$ otherwise.
\end{myProp}

\begin{proof}
Let $x\in\mathrm{C}^{\perp}$. From~\cite[Lesson E]{URRU}, it holds that: 
\begin{enumerate}[label=(\roman*)]
    \item $y\in\partial{\xi_{\mathrm{C}}}(x)\Leftrightarrow x\in\partial{\iota_\mathrm{C}}(y)\Leftrightarrow y\in\mathrm{C}$;
    \item if $y\in\mathrm{C}$, then $z\in\mathrm{N}_{\mathrm{C}}(y)\Leftrightarrow \xi_{\mathrm{C}}(z)=z\cdot y$.
\end{enumerate}
From the first item one deduces that~$\partial{\xi_{\mathrm{C}}}(x)=\mathrm{C}$.
Let~$y\in\mathrm{C}$ and let us prove that $h_\mathrm{C}$ is twice epi-differentiable at~$x$ for $y$. To this aim we use Proposition~\ref{caractMosco}. Consider $z\in\mathrm{N}_{\mathrm{C}}(y)$ and thus~$\xi_{\mathrm{C}}(z)=y\cdot z$. By considering the sequence $z_{t}:=z$ for all $t>0$, one gets
\begin{multline*}
\displaystyle\delta_{t}^{2}\xi_{\mathrm{C}}(x \mid y)(z_{t})=\frac{\xi_{\mathrm{C}}(x+tz)-\xi_{\mathrm{C}}(x)-t y\cdot z}{t^{2}}= \\
\frac{\sup_{c\in\mathrm{C}}\left(x+tz \right) \cdot c-t y\cdot z}{t^{2}}=\frac{\xi_{\mathrm{C}}(z)-\xi_{\mathrm{C}}(z)}{t}=0.
\end{multline*}
Moreover, since $\displaystyle\delta_{t}^{2}\xi_{\mathrm{C}}(x \mid y)(w)\geq0$ for all $w\in\R^d$, one deduces that
$\mathrm{d}_{e}^{2}\xi_{\mathrm{C}}(x \mid y)(z)=0$.
Now consider~$z\notin\mathrm{N}_{\mathrm{C}}(y)$. There exists $c_{0}\in\mathrm{C}$ such that $z\cdot c_{0}>z\cdot y$, thus $\xi_{\mathrm{C}}(z)>z \cdot y$. Consider any sequence $\left(z_{t}\right)_{t>0}\rightarrow z$. Since $\xi_{\mathrm{C}}$ is lower semi-continuous one has
$$
\mathrm{lim}\inf \xi_{\mathrm{C}}(z_t)\geq \xi_{\mathrm{C}}(z) \quad \text{and} \quad y\cdot z_t{\longrightarrow} y\cdot z,
$$
when $t\rightarrow 0^{+}$. Therefore there exists $\eps>0$ such that
$$
\xi_{\mathrm{C}}(z_t)\geq \xi_{\mathrm{C}}(z) -\frac{\xi_{\mathrm{C}}(z)-y\cdot z}{4} \quad \text{and} \quad -y\cdot z_t\geq-y\cdot z-\frac{\xi_{\mathrm{C}}(z)-y\cdot z}{4},
$$
for all $t\leq\eps$, and thus
$$
\displaystyle\delta_{t}^{2}\xi_{\mathrm{C}}(x \mid y)(z_{t})=\frac{\xi_{\mathrm{C}}(z_{t})-y \cdot z_{t}}{t}\geq\frac{\xi_{\mathrm{C}}(z)-y \cdot z}{2t}{\longrightarrow}+\infty,
$$
when $t\rightarrow 0^{+}$. Hence $\mathrm{d}_{e}^{2}\xi_{\mathrm{C}}(x \mid y)(z)=+\infty
$ which concludes the proof.
\end{proof}

\begin{myRem}\normalfont
    Note that if we consider a general real Hilbert space~$(\mathcal{H}, \dual{\cdot}{\cdot}_{\mathcal{H}})$, then Proposition~\ref{epitangnorm} is still true since the support function of a nonempty closed convex subset of $\mathcal{H}$ is convex and lower semi-continuous, thus weakly lower semi-continuous.
\end{myRem}

\begin{myLem}\label{epinormtang}
For all $s\in\Gamma_{\mathrm{N}}$, the map $\left\|\cdot_{\tau(s)}\right\|$ is twice epi-differentiable at any $x\in\R^{d}$ for any~$y\in\partial{\left\|\cdot_{\tau(s)}\right\|}(x)$ and its second-order epi-derivative is given by
\begin{equation*}
\mathrm{d}_{e}^{2}\left\|\cdot_{\tau(s)}\right\|(x \mid y)(z) =
\left\{
\begin{array}{lcll}
\frac{1}{2\left\|x_{\tau(s)}\right\|}\left(\left\|z_{\tau(s)}\right\|^{2}-\left|z_{\tau(s)}\cdot\frac{x_{\tau(s)}}{\left\|x_{\tau(s)}\right\|}\right|^2\right)	&   & \text{ if } x_{\tau(s)}\neq0 , \\
\iota_{\mathrm{N}_{\overline{\mathrm{B}(0,1)}\cap\left(\R\nn(s)\right)^{\perp}}(y)}(z)	&   & \text{ if } x_{\tau(s)}=0,
\end{array}
\right.
\end{equation*}
for all $z\in\R^{d}$, where $\mathrm{N}_{\overline{\mathrm{B}(0,1)}\cap\left(\R\nn(s)\right)^{\perp}}(y)$ is the normal cone to $\overline{\mathrm{B}(0,1)}\cap\left(\R\nn(s)\right)^{\perp}$ at $y$.
\end{myLem}

\begin{proof}
Let $s\in\Gamma_{\mathrm{N}}$. Note that
$$
\partial{\left\|\cdot_{\tau(s)}\right\|}(x):=
\left\{
\begin{array}{lcll}
\left\{\frac{x_{\tau(s)}}{\left\|x_{\tau(s)}\right\|}\right\}	&   & \text{ if } x_{\tau(s)}\neq0 , \\
\mathrm{\overline{\mathrm{B}(0,1)}\cap\left(\R\nn(s)\right)^{\perp}}	&   & \text{ if } x_{\tau(s)}=0.
\end{array}
\right.
$$
Since
$$
 (\mathrm{\overline{\mathrm{B}(0,1)}\cap\left(\R\nn(s)\right)^{\perp}})^{\perp}=\R\nn(s),
$$ 
one can apply Proposition~\ref{epitangnorm} to get that
$$
\mathrm{d}_{e}^{2}\left\|\cdot_{\tau(s)}\right\|(x \mid y)=\iota_{\mathrm{N}_{\overline{\mathrm{B}(0,1)}\cap\left(\R\nn(s)\right)^{\perp}}(y)},
$$
for all $x\in\R\nn(s)$ and all $y\in\overline{\mathrm{B}(0,1)}\cap\left(\R\nn(s)\right)^{\perp}$.
In the case where $x\notin\R\nn(s)$ (i.e. $x_{\tau(s)}\neq0$), one can easily prove that $\left\|\cdot_{\tau(s)}\right\|$ is twice Fréchet differentiable at $x$ with
\begin{multline*}
\mathrm{D}^{2}\left\|\cdot_{\tau(s)}\right\|(x)(z_1,z_2)=\\\frac{1}{\left\|x_{\tau(s)}\right\|}\left(z_{1_{\tau(s)}}\cdot z_{2_{\tau(s)}}-\left(x_{\tau(s)}\cdot z_{2_{\tau(s)}}\right)z_{1_{\tau(s)}}\cdot\frac{x_{\tau(s)}}{\left\|x_{\tau(s)}\right\|^{2}}\right),\qquad\forall\left(z_1,z_2\right)\in\R^{d}\times\R^{d}.
\end{multline*}
From Remark~\ref{diffsecond}, one gets
$$
\mathrm{d}_{e}^{2}\left\|\cdot_{\tau(s)}\right\| \left( x\mid\frac{x_{\tau(s)}}{\left\|x_{\tau(s)}\right\|} \right)(z)=\frac{1}{2}\mathrm{D}^{2}\left\|\cdot_{\tau(s)}\right\|(x)(z,z)=\frac{1}{2\left\|x_{\tau(s)}\right\|}\left(\left\|z_{\tau(s)}\right\|^{2}-\left|z_{\tau(s)}\cdot\frac{x_{\tau(s)}}{\left\|x_{\tau(s)}\right\|}\right|^2\right),
$$
for all $z\in\R^d$, which concludes the proof.
\end{proof}

Now, with additional assumptions, let us compute the second-order epi-derivative of $G(s)$ for almost all $s\in\Gamma_{\mathrm{N}}$.
\begin{myProp}\label{épidiffgabs}
Assume that, for almost all $s\in\Gamma_{\mathrm{N}}$, the map $t\in\mathbb{R}_{+}\mapsto g_{t}(s)\in\mathbb{R}_{+}$ is differentiable at $t=0$, with its derivative denoted by $g'_{0}(s)$. Then, for almost all $s\in\Gamma_{\mathrm{N}}$, the map~$G(s)$ is twice epi-differentiable at any~$x\in\mathbb{R}^{d}$ for all $y\in 
g_{0}(s)\partial{||\mathord{\cdot}_{\tau(s)}||}(x)$ with
\begin{equation*}
\small{
\mathrm{D}_{e}^{2}G(s)(x \mid y)(z):=
\left\{
\begin{array}{lcll}
\frac{g_{0}(s)}{2\left\|x_{\tau(s)}\right\|}\left(\left\|z_{\tau(s)}\right\|^{2}-\left|z_{\tau(s)}\cdot\frac{x_{\tau(s)}}{\left\|x_{\tau(s)}\right\|}\right|^2\right)+g'_{0}(s)\frac{x_{\tau(s)}}{\left\|x_{\tau(s)}\right\|}\cdot z	&   & \text{ if } x_{\tau(s)}\neq0 , \\
\iota_{\mathrm{N}_{\overline{\mathrm{B}(0,1)}\cap\left(\R\nn(s)\right)^{\perp}}(\frac{y}{g_{0}(s)})}(z)+g'_{0}(s)\frac{y}{g_{0}(s)}\cdot z	&   & \text{ if } x_{\tau(s)}=0,
\end{array}
\right.}
\end{equation*}
for all $z\in\mathbb{R}^{d}$.
\end{myProp}

\begin{proof}
We use the same notations as in Definitions~\ref{epidiff} and~\ref{epidiffpara}.
Let $x\in\mathbb{R}^{d}$. Then, for almost all~$s\in\Gamma_{\mathrm{N}}$, for all $y\in g_{0}(s)\partial{||\mathord{\cdot}_{\tau(s)}||}(x)$ and all $z\in\mathbb{R}^{d}$, one has
\begin{multline*}
\displaystyle\Delta_{t}^{2}G(s)(x \mid y)(z)=\frac{g_{t}(s)\left\|x_{\tau(s)}+t z_{\tau(s)}\right\|-g_{t}(s)\left\|x_{\tau(s)}\right\|-ty\cdot z}{t^{2}}\\=g_{t}(s)\frac{\left\|x_{\tau(s)}+t z_{\tau(s)}\right\|-\left\|x_{\tau(s)}\right\|-t \frac{y}{g_{0}(s)}\cdot z}{t^{2}}+\frac{\left(g_{t}(s)-g_{0}(s)\right)}{tg_{0}(s)}y\cdot z,
\end{multline*}
that is
$$
\displaystyle \Delta_{t}^{2}G(s)(x \mid y)(z)=g_{t}(s)\delta_{t}^{2}||\mathord{\cdot}_{\tau(s)}|| \left( x \mid \frac{y}{g_{0}(s)} \right)(z)+\frac{\left(g_{t}(s)-g_{0}(s)\right)}{tg_{0}(s)}y\cdot z,
$$ 
with $\frac{y}{g_{0}(s)}\in\partial||\mathord{\cdot}_{\tau(s)}||(x)$, and where $\delta_{t}^{2}||\mathord{\cdot}_{\tau(s)}||(x|\frac{y}{g_{0}(s)})$ is the second-order difference quotient function of~$||\mathord{\cdot}_{\tau(s)}||$ at $x$ for $\frac{y}{g_{0}(s)}$ (see Definition~\ref{epidiff} since $||\mathord{\cdot}_{\tau(s)}||$ is a $t$-independent function). Using the characterization of Mosco epi-convergence (see Proposition~\ref{caractMosco}) and Lemma~\ref{epinormtang}, one gets
$$
\displaystyle\mathrm{D}_{e}^{2}G(s)(x \mid y)(z)=g_{0}(s)\mathrm{d}_{e}^{2}\left\|\cdot_{\tau(s)}\right\| \left( x \mid \frac{y}{g_{0}(s)} \right)+g'_{0}(s)\frac{y}{g_{0}(s)}\cdot z.
$$
The proof is complete.
\end{proof}

To conclude this part, let us characterize $\mathrm{N}_{\overline{\mathrm{B}(0,1)}\cap\left(\R\nn(s)\right)^{\perp}}(y)$ for all~$y\in\overline{\mathrm{B}(0,1)}\cap\left(\R\nn(s)\right)^{\perp}$ and for almost all $s\in\Gamma_{\mathrm{N}}$.

\begin{myLem}\label{conenormalde}
Let $s\in\Gamma_{\mathrm{N}}$. It holds that
\begin{equation*}
\mathrm{N}_{\overline{\mathrm{B}(0,1)}\cap\left(\R\nn(s)\right)^{\perp}}(y)=
\left\{
\begin{array}{lcll}
\R\nn(s)	&   & \text{ if } y\in\mathrm{B}(0,1)\cap\left(\R\nn(s)\right)^{\perp}, \\
\R\nn(s)+\R_{+}y	&   & \text{ if } y\in\partial{\mathrm{B}(0,1)}\cap\left(\R\nn(s)\right)^{\perp},
\end{array}
\right.
\end{equation*}
for all $y\in\overline{\mathrm{B}(0,1)}\cap\left(\R\nn(s)\right)^{\perp}$, where $\R_{+}y:=\{ z\in\R^d \mid \exists  \nu \geq0 \text{ such that } z=\nu y\}$.
\end{myLem}

\begin{proof}
Let $s\in\Gamma_{\mathrm{N}}$ and $y\in\overline{\mathrm{B}(0,1)}\cap\left(\R\nn(s)\right)^{\perp}.$
\begin{enumerate}
    \item[{\rm (i)}] First, let $y\in\mathrm{B}(0,1)\cap\left(\R\nn(s)\right)^{\perp}$.
If $v\in\R\nn(s)$, then 
$$
v\cdot(y-z)=0, \qquad\forall z\in \overline{\mathrm{B}(0,1)}\cap\left(\R\nn(s)\right)^{\perp},
$$
thus $v\in\mathrm{N}_{\overline{\mathrm{B}(0,1)}\cap\left(\R\nn(s)\right)^{\perp}}(y)$. Since this is true for any $v\in\R\nn(s)$, one deduces that $$\R\nn(s)\subset\mathrm{N}_{\overline{\mathrm{B}(0,1)}\cap\left(\R\nn(s)\right)^{\perp}}(y).
$$ 
Consider $v\in\mathrm{N}_{\overline{\mathrm{B}(0,1)}\cap\left(\R\nn(s)\right)^{\perp}}(y)$. Then it holds that
$$
v\cdot\left(z-y\right)\leq0,\qquad\forall z \in\overline{\mathrm{B}(0,1)}\cap\left(\R\nn(s)\right)^{\perp}. 
$$
Moreover there exists $\eps>0$ such that $\mathrm{B}(y,\eps)\cap\left(\R\nn(s)\right)^{\perp}\subset\mathrm{B}(0,1)\cap\left(\R\nn(s)\right)^{\perp}$. Therefore by considering $z:=y+\eps\frac{w}{2\left\|w\right\|}$ for any $w\in\left(\R\nn(s)\right)^{\perp}$, one deduces that
$$
v\cdot w=0,\qquad\forall w\in\left(\R\nn(s)\right)^{\perp}.
$$
Thus $v\in((\R\nn(s))^{\perp})^{\perp}=\R\nn(s)$. Since this is true for any $v\in\mathrm{N}_{\overline{\mathrm{B}(0,1)}\cap\left(\R\nn(s)\right)^{\perp}}(y)$, one deduces that
$$
\mathrm{N}_{\overline{\mathrm{B}(0,1)}\cap\left(\R\nn(s)\right)^{\perp}}(y)\subset\R\nn(s).
$$
\item[{\rm (ii)}]  Let $y\in\partial{\mathrm{B}}(0,1)\cap\left(\R\nn(s)\right)^{\perp}$. If $v\in\R\nn(s)+\R_{+}y$, then
$$
v\cdot\left(z-y\right)=v_{\tau(s)}\cdot\left(z-y\right)\leq\left\|v_{\tau(s)}\right\| \left\|z\right\|-v_{\tau(s)}\cdot y\leq\left\|v_{\tau(s)}\right\|-\left\|v_{\tau(s)}\right\|=0,
$$
for all $z\in \overline{\mathrm{B}(0,1)}\cap\left(\R\nn(s)\right)^{\perp}$. Thus it follows that
$$
\R\nn(s)+\R_{+}y\subset\mathrm{N}_{\overline{\mathrm{B}(0,1)}\cap\left(\R\nn(s)\right)^{\perp}}(y).
$$
Let $v\in\mathrm{N}_{\overline{\mathrm{B}(0,1)}\cap\left(\R\nn(s)\right)^{\perp}}(y)$, and consider $z:=\frac{1}{2}\left(\frac{v_{\tau(s)}}{\left\|v_{\tau(s)}\right\|}\left\|y\right\|+y\right)\in\overline{\mathrm{B}(0,1)}\cap\left(\R\nn(s)\right)^{\perp}$. One deduces that
$$
0\geq v\cdot(z-y)=v_{\tau(s)}\cdot\frac{1}{2}\left(\frac{v_{\tau(s)}}{\left\|v_{\tau(s)}\right\|}\left\|y\right\|-y\right)=\frac{1}{2}\left(\left\|v_{\tau(s)}\right\|\left\|y\right\|-v_{\tau(s)}\cdot y\right)\geq0,
$$
thus $\left\|v_{\tau(s)}\right\|\left\|y\right\|=v_{\tau(s)}\cdot y$, hence $v_{\tau(s)}\in\R_{+}y$. This is true for any~$v\in\mathrm{N}_{\overline{\mathrm{B}(0,1)}\cap\left(\R\nn(s)\right)^{\perp}}(y)$, thus one deduces that
$$
\mathrm{N}_{\overline{\mathrm{B}(0,1)}\cap\left(\R\nn(s)\right)^{\perp}}(y)\subset\R\nn(s)+\R_{+}y.
$$
\end{enumerate}
The proof is complete.
\end{proof}

\subsubsection{The derivative of the solution to the parameterized Tresca friction problem}
From the previous results and some additional assumptions detailed below, we are now in a position to state and prove the main result of this paper which characterizes the derivative of the solution to the parameterized Tresca friction problem~\eqref{PbNeumannDirichletTrescaPara}.

\begin{myTheorem}\label{caractu0derivDNT}
Let $u_{t}\in\HH^{1}_{\mathrm{D}}(\Omega,\R^d)$ be the unique solution to the parameterized Tresca friction problem~\eqref{PbNeumannDirichletTrescaPara} for all~$t \geq 0$. Let us assume that:
\begin{enumerate}[label={\rm (\roman*)}]
    \item the map $t\in\mathbb{R}_{+}\mapsto f_{t}\in \mathrm{L}^{2}(\Omega,\R^d)$ is differentiable at $t=0$, with its derivative denoted by~$f'_{0}\in\LL^{2}(\Omega,\R^d)$;\label{hypo1}
    \item the map $t\in\mathbb{R}_{+}\mapsto h_{t}\in \LL^{2}(\Gamma_{\mathrm{N}})$ is differentiable at $t=0$, with its derivative denoted by~$h'_{0}\in\LL^{2}(\Gamma_{\mathrm{N}})$;\label{hypo2}
    \item for almost all $s\in\Gamma_{\mathrm{N}}$, the map $t\in\mathbb{R}_{+}\mapsto g_{t}(s)\in\mathbb{R}_{+}$ is differentiable at $t=0$, with its derivative denoted by $g'_{0}(s)$, and also $g_{0}'\in \mathrm{L}^{2}(\Gamma_{\mathrm{N}})$;\label{hypo3}
    \item the map $s\in\Gamma_{\mathrm{N}^{u_0,g_0}_{\mathrm{R}}}\mapsto \frac{g_{0}(s)}{\left\|u_{0_\tau}(s)\right\|}\in\mathbb{R}_{+}$ belongs to $\LL^{4}(\Gamma_{\mathrm{N}^{u_0,g_0}_{\mathrm{R}}})$ (see below for definition of the set $\Gamma_{\mathrm{N}^{u_0,g_0}_{\mathrm{R}}}$);\label{hyposup}
    \item the parameterized Tresca friction functional $\Phi$ defined in~\eqref{fonctionnelledeTrescaparacas2} is twice epi-differentiable (see Definition~\ref{epidiffpara}) at $u_{0}$ for $F_{0}-u_{0}\in\partial \Phi(0,\cdot)(u_{0})$, with\label{hypo4}
\begin{equation}\label{hypoth1}
\displaystyle\mathrm{D}_{e}^{2}\Phi(u_{0}\mid F_{0}-u_{0})(w)=\int_{\Gamma_{\mathrm{N}}}\mathrm{D}_{e}^{2}G(s)(u_{0}(s)\mid\sigma_{\tau}(F_{0}-u_{0})(s))(w(s)) \, \mathrm{d}s,
\end{equation}    
for all $w\in \HH^{1}_{\mathrm{D}}(\Omega,\R^d)$, where $F_{0}\in\HH^{1}_{\mathrm{D}}(\Omega,\R^d)$ is the unique solution to the parameterized Dirichlet-Neumann problem~\eqref{PbNeumannDirichletPara} for the parameter~$t=0$.
\end{enumerate}
Then the map $t\in\mathbb{R}_{+}\mapsto u_{t}\in\HH^{1}_{\mathrm{D}}(\Omega,\R^d)$ is differentiable at $t=0$, and its derivative denoted by~$u'_{0}\in\HH^{1}_{\mathrm{D}}(\Omega,\R^d)$ is the unique weak solution to the tangential Signorini problem
\begin{equation}\tag{SP\ensuremath{_{0}'}}\label{caractu0DNT}
{\arraycolsep=2pt
\left\{
\begin{array}{rcll}
-\mathrm{div}(\mathrm{A}\mathrm{e}(u'_{0})) & = & f'_0   & \text{ in } \Omega , \\[5pt]
u'_{0} & = & 0  & \text{ on } \Gamma_{\mathrm{D}} ,\\[5pt]
\sigma_{\nn}(u'_{0}) & = & h'_0  & \text{ on } \Gamma_{\mathrm{N}} ,\\[5pt]
u'_{0_\tau} & = & 0  & \text{ on } \Gamma_{\mathrm{N}^{u_0,g_0}_{\mathrm{T}}},\\[5pt]
\sigma_{\tau}(u'_{0})+\frac{g_{0}}{\left\|u_{0_\tau}\right\|}\left(u'_{0_\tau}-\left(u'_{0_\tau}\cdot \frac{u_{0_\tau}}{\left\|u_{0_\tau}\right\|}\right)\frac{u_{0_\tau}}{\left\|u_{0_\tau}\right\|}\right) & = & -g'_0\frac{u_{0_\tau}}{\left\|u_{0_\tau}\right\|}  & \text{ on } \Gamma_{\mathrm{N}^{u_0,g_0}_{\mathrm{R}}} ,\\[15pt]
u'_{0_\tau}\in\R_{-}\frac{\sigma_{\tau}(u_0)}{g_0}, \left(\sigma_{\tau}(u'_0)-g'_0 \frac{\sigma_{\tau}(u_0)}{g_0}\right)\cdot \frac{\sigma_{\tau}(u_0)}{g_0}\leq0 \\ \text{ and } u'_{0_\tau}\cdot\left(\sigma_{\tau}(u'_0)-g'_0 \frac{\sigma_{\tau}(u_0)}{g_0}\right)  & = & 0  & \text{ on } \Gamma_{\mathrm{N}^{u_0,g_0}_{\mathrm{S}}},
\end{array}
\right.}
\end{equation}
where $\Gamma_{\mathrm{N}}$ is decomposed (up to a null set) as
$\Gamma_{\mathrm{N}^{u_0,g_0}_{\mathrm{T}}}\cup\Gamma_{\mathrm{N}^{u_0,g_0}_{\mathrm{R}}}\cup\Gamma_{\mathrm{N}^{u_0,g_0}_{\mathrm{S}}}$ with
$$
\begin{array}{l}
\Gamma_{\mathrm{N}^{u_0,g_0}_{\mathrm{R}}}:=\left\{s\in\Gamma_{\mathrm{N}} \mid  u_{0_\tau}(s)\neq0\right \}, \\
\Gamma_{\mathrm{N}^{u_0,g_{0}}_{\mathrm{T}}}:=\left\{s\in\Gamma_{\mathrm{N}} \mid  u_{0_\tau}(s)=0 \text{ and } \frac{\sigma_{\tau}(u_{0})(s)}{g_{0}(s)}\in\mathrm{B}(0,1)\cap\left(\R\nn(s)\right)^{\perp}\right\}, \\
\Gamma_{\mathrm{N}^{u_0,g_0}_{\mathrm{S}}}:=\left\{s\in\Gamma_{\mathrm{N}} \mid  u_{0_\tau}(s)=0 \text{ and } \frac{\sigma_{\tau}(u_{0})(s)}{g_{0}(s)}\in\partial{\mathrm{B}(0,1)}\cap\left(\R\nn(s)\right)^{\perp}\right\}.
\end{array}
$$
\end{myTheorem}

\begin{myRem}\normalfont\label{Remarquenotwice}
   As mentioned in papers~\cite{4ABC,BCJDC}, one can naturally expect from Proposition~\ref{epidiffoffunctionG} that the second-order epi-derivative of the parameterized Tresca friction functional $\Phi$ at~$u_{0}$ for~$F_{0}-u_{0}$ is given by Equality~\eqref{hypoth1}, which corresponds to the inversion of the symbols~$\mathrm{ME}\text{-}\mathrm{lim}$ and $\int_{\Gamma_{\mathrm{N}}}$  in Equality~\eqref{Delta2}. Nevertheless, to the best of our knowledge, the validity of this inversion is an open question in the literature. Precisely, we do not know, in general, if the parameterized Tresca friction functional is twice epi-differentiable at~$u_{0}$ for $F_{0}-u_{0}$. Nevertheless, similarly to~\cite[Appendix A]{BCJDC}, one can prove it in some practical situations.
\end{myRem}

\begin{proof}[Proof of Theorem~\ref{caractu0derivDNT}]
From Hypotheses~\ref{hypo3},~\ref{hypo4} and Proposition~\ref{épidiffgabs}, it follows that
\begin{multline*}
\displaystyle \mathrm{D}_{e}^{2}\Phi(u_{0}\mid F_{0}-u_{0})(w)=\int_{\Gamma_{\mathrm{N}^{u_0,g_0}_{\mathrm{R}}}}\left(\frac{g_{0}}{2\left\|u_{0_\tau}\right\|}\left(\left\|w_\tau\right\|^{2}-\left|w_\tau\cdot\frac{u_{0_\tau}}{\left\|u_{0_\tau}\right\|}\right|^2\right)+g'_{0}\frac{u_{0_\tau}}{\left\|u_{0_\tau}\right\|}\cdot w\right)\\+\int_{\Gamma_{\mathrm{N}}\backslash\Gamma_{\mathrm{N}^{u_0,g_0}_{\mathrm{R}}}}\iota_{\mathrm{N}_{\overline{\mathrm{B}(0,1)}\cap\left(\R\nn(s)\right)^{\perp}}(\frac{\sigma_{\tau}\left(F_0-u_0\right)(s)}{g_{0}(s)})}(w(s))\mathrm{d}s+\int_{\Gamma_{\mathrm{N}}\backslash\Gamma_{\mathrm{N}^{u_0,g_0}_{\mathrm{R}}}}g'_{0}\frac{\sigma_{\tau}\left(F_0-u_0\right)}{g_{0}}\cdot w,
\end{multline*}
for all $w\in \HH^{1}_{\mathrm{D}}(\Omega,\R^d)$, which can be rewritten as
\begin{multline*}
\displaystyle \mathrm{D}_{e}^{2}\Phi(u_{0}\mid F_{0}-u_{0})(w)=\\\Psi(w)+\int_{\Gamma_{\mathrm{N}^{u_0,g_0}_{\mathrm{R}}}}g'_{0}\frac{u_{0_\tau}}{\left\|u_{0_\tau}\right\|}\cdot w_\tau+\iota_{\mathcal{K}_{u_{0},\frac{\sigma_{\tau}(F_{0}-u_{0})}{g_{0}}}}(w)+\int_{\Gamma_{\mathrm{N}}\backslash\Gamma_{\mathrm{N}^{u_0,g_0}_{\mathrm{R}}}}g'_{0}\frac{\sigma_{\tau}\left(F_0-u_0\right)}{g_{0}}\cdot w_\tau,
\end{multline*}
for all $w\in \HH^{1}_{\mathrm{D}}(\Omega,\R^d)$, where
$\Psi$ is defined by
\begin{equation*}
\fonction{\Psi}{\HH^{1}_{\mathrm{D}}(\Omega,\R^d)}{\R}{w}{\displaystyle \Psi(w):=\int_{\Gamma_{\mathrm{N}^{u_0,g_0}_{\mathrm{R}}}}\frac{g_{0}}{2\left\|u_{0_\tau}\right\|}\left(\left\|w_\tau\right\|^{2}-\left|w_\tau\cdot\frac{u_{0_\tau}}{\left\|u_{0_\tau}\right\|}\right|^2\right),}
\end{equation*}
which is well defined from the continuous embedding $\HH^{1}(\Omega,\R^d){\hookrightarrow} \LL^{4}(\Gamma,\R^d)$ and from Hypothesis~\ref{hyposup}, and where $\mathcal{K}_{u_{0},\frac{\sigma_{\tau}(F_{0}-u_{0})}{g_{0}}}$ is the nonempty closed convex subset of $\HH^{1}_{\mathrm{D}}(\Omega,\R^d)$ defined by 
\begin{multline*}
\mathcal{K}_{u_{0},\frac{\sigma_{\tau}(F_{0}-u_{0})}{g_{0}}}:=\biggl\{ w\in \HH^{1}_{\mathrm{D}}(\Omega,\R^d)\mid w(s)\in \mathrm{N}_{\overline{\mathrm{B}(0,1)}\cap\left(\R\nn(s)\right)^{\perp}}\left(\frac{\sigma_{\tau}\left(F_0-u_0\right)(s)}{g_{0}(s)}\right) \\ \text{ for almost all }s\in\Gamma_{\mathrm{N}}\backslash\Gamma_{\mathrm{N}^{u_0,g_0}_{\mathrm{R}}} \biggl\}.
\end{multline*}
Moreover, from Lemma~\ref{conenormalde} and since $\sigma_{\tau}(F_0)=0$ \textit{a.e.} on $\Gamma_{\mathrm{N}}$, it follows that
\begin{equation*}
    \mathcal{K}_{u_{0},\frac{\sigma_{\tau}(F_{0}-u_{0})}{g_{0}}}=\left\{ w\in \HH^{1}_{\mathrm{D}}(\Omega,\R^d)\mid w_\tau=0 \text{ \textit{a.e.} on } \Gamma_{\mathrm{N}^{u_0,g_{0}}_{\mathrm{T}}}  \text{ and } w_\tau\in\R_{-}\frac{\sigma_\tau(u_0)}{g_0} \text{ \textit{a.e.} on } \Gamma_{\mathrm{N}^{u_0,g_0}_{\mathrm{S}}} \right\}.
\end{equation*}
Since $\frac{g_{0}}{||u_{0_\tau}||}>0$ \textit{a.e.} on $\Gamma_{\mathrm{N}^{u_0,g_0}_{\mathrm{R}}}$ and from Lemma~\ref{fonctionnelleannexe}, one deduces that $\Psi$ is convex and Fréchet differentiable on $\HH^{1}_{\mathrm{D}}(\Omega,\R^d)$. In particular we get that 
$\mathrm{D}_{e}^{2}\Phi(u_{0}|F_{0}-u_{0})$ is a proper lower semi-continuous convex function on $\HH^{1}_{\mathrm{D}}(\Omega,\R^d)$. Moreover, from Hypotheses~\ref{hypo1} and~\ref{hypo2} and from the linearity of the Dirichlet-Neumann problem~\eqref{PbNeumannDirichlet} and Proposition~\ref{existenceunicitéDN}, we can easily prove that the map $t\in\mathbb{R}_{+}\mapsto F_{t}\in \HH^{1}_{\mathrm{D}}(\Omega,\R^d)$ is differentiable at $t=0$, with its
derivative $F'_0\in\HH^{1}_{\mathrm{D}}(\Omega,\R^d)$ being the unique solution to the Dirichlet-Neumann problem
\begin{equation*}
\arraycolsep=2pt
\left\{
\begin{array}{rcll}
-\mathrm{div}(\mathrm{A}\mathrm{e}(F'_{0})) & = & f'_{0}   & \text{ in } \Omega , \\
F'_{0} & = & 0  & \text{ on } \Gamma_{\mathrm{D}} ,\\
\mathrm{A}\mathrm{e}(F'_{0})\nn & = & h'_{0}\nn  & \text{ on } \Gamma_{\mathrm{N}}.
\end{array}
\right.
\end{equation*}
Thus one can apply Theorem~\ref{TheoABC2018} to deduce that the map $t\in\mathbb{R}_{+}\mapsto u_{t}\in \HH^{1}_{\mathrm{D}}(\Omega,\R^d)$ is differentiable at $t=0$, and its derivative $u_{0}'\in\HH^{1}_{\mathrm{D}}(\Omega,\R^d)$ satisfies
$$
\displaystyle u_{0}'=\mathrm{prox}_{\mathrm{D}_{e}^{2}\Phi(u_{0}\mid F_{0}-u_{0})}(F_{0}'),
$$
which, from the definition of the proximal operator (see Proposition~\ref{proxi}), leads to
$$
\displaystyle F_{0}'-u_{0}'\in\partial \mathrm{D}_{e}^{2}\Phi(u_{0}\mid F_{0}-u_{0})(u_{0}'),
$$
which means that
$$
\displaystyle \dual{ F_{0}'-u'_{0}}{w-u_{0}'}_{\HH^{1}_{\mathrm{D}}(\Omega,\R^d)}\leq \mathrm{D}_{e}^{2}\Phi(u_{0}\mid F_{0}-u_{0})(w) -\mathrm{D}_{e}^{2}\Phi(u_{0}\mid F_{0}-u_{0})(u_{0}'),
$$
for all $w\in \HH^{1}_{\mathrm{D}}(\Omega,\R^d)$. Hence we get that
\begin{multline*}
    \dual{F_{0}'-u'_{0}}{w-u_{0}'}_{\HH^{1}_{\mathrm{D}}(\Omega,\R^d)} \leq \Psi(w)-\Psi(u'_0)+\iota_{\mathcal{K}_{u_{0},\frac{\sigma_{\tau}(F_{0}-u_{0})}{g_{0}}}}(w)-\iota_{\mathcal{K}_{u_{0},\frac{\sigma_{\tau}(F_{0}-u_{0})}{g_{0}}}}(u'_0)\\+\int_{\Gamma_{\mathrm{N}^{u_0,g_0}_{\mathrm{R}}}}g'_{0}\frac{u_{0_\tau}}{\left\|u_{0_\tau}\right\|}\cdot \left(w_\tau-u'_{0_\tau}\right)+\int_{\Gamma_{\mathrm{N}}\backslash\Gamma_{\mathrm{N}^{u_0,g_0}_{\mathrm{R}}}}g'_{0}\frac{\sigma_{\tau}\left(F_0-u_0\right)}{g_{0}}\cdot \left(w_\tau-u'_{0_\tau}\right),
\end{multline*}
for all $w\in\HH^{1}_{\mathrm{D}}(\Omega,\R^d)$.
Moreover, since $\sigma_{\tau}(F_{0})=0$ \textit{a.e.} on $\Gamma_{\mathrm{N}}$, and for all $w\in\mathcal{K}_{u_{0},\frac{\sigma_{\tau}(F_{0}-u_{0})}{g_{0}}}$,~$w_\tau=0$ \textit{a.e.} $\Gamma_{\mathrm{N}^{u_0,g_{0}}_{\mathrm{T}}}$, one deduces that $u'_0\in\mathcal{K}_{u_{0},\frac{\sigma_{\tau}(F_{0}-u_{0})}{g_{0}}}$ and
\begin{multline*}
 \dual{u'_{0}}{w-u_{0}'}_{\HH^{1}_{\mathrm{D}}(\Omega,\R^d)}+\Psi(w)-\Psi(u'_0)\geq\int_{\Omega}f'_0\cdot\left(w-u'_0\right)+\int_{\Gamma_{\mathrm{N}}}h'_0\left(w_\nn-u'_{0_\nn}\right)\\-\int_{\Gamma_{\mathrm{N}^{u_0,g_0}_{\mathrm{R}}}}g'_{0}\frac{u_{0_\tau}}{\left\|u_{0_\tau}\right\|}\cdot \left(w_\tau-u'_{0_\tau}\right)+\int_{\Gamma_{\mathrm{N}^{u_0,g_0}_{\mathrm{S}}}}g'_{0}\frac{\sigma_{\tau}\left(u_0\right)}{g_{0}}\cdot \left(w_\tau-u'_{0_\tau}\right),
\end{multline*}
for all $w\in\mathcal{K}_{u_{0},\frac{\sigma_{\tau}(F_{0}-u_{0})}{g_{0}}}$.
Moreover, since $\Psi$ is convex and Fréchet differentiable on $\HH^{1}_{\mathrm{D}}(\Omega,\R^d)$ (see Lemma~\ref{fonctionnelleannexe}), one gets from Proposition~\ref{gradientequi} that
\begin{multline*}
    \dual{\nabla{\Psi}(u'_0)}{w-u'_0}_{\HH^{1}_{\mathrm{D}}(\Omega,\R^d)}\geq -\dual{u'_{0}}{w-u_{0}'}_{\HH^{1}_{\mathrm{D}}(\Omega,\R^d)}+\int_{\Omega}f'_0\cdot\left(w-u'_0\right)+\int_{\Gamma_{\mathrm{N}}}h'_0\left(w_\nn-u'_{0_\nn}\right)\\-\int_{\Gamma_{\mathrm{N}^{u_0,g_0}_{\mathrm{R}}}}g'_{0}\frac{u_{0_\tau}}{\left\|u_{0_\tau}\right\|}\cdot \left(w_\tau-u'_{0_\tau}\right)+\int_{\Gamma_{\mathrm{N}^{u_0,g_0}_{\mathrm{S}}}}g'_{0}\frac{\sigma_{\tau}\left(u_0\right)}{g_{0}}\cdot \left(w_\tau-u'_{0_\tau}\right),
\end{multline*}
for all $w\in\mathcal{K}_{u_{0},\frac{\sigma_{\tau}(F_{0}-u_{0})}{g_{0}}}$. Finally, using the expression of $\nabla{\Psi}(u'_0)\in\HH^{1}_{\mathrm{D}}(\Omega,\R^d)$, one gets
\begin{multline*}
    \dual{u'_{0}}{w-u_{0}'}_{\HH^{1}_{\mathrm{D}}(\Omega,\R^d)}\geq\int_{\Omega}f'_0\cdot\left(w-u'_0\right)+\int_{\Gamma_{\mathrm{N}}}h'_0\left(w_\nn-u'_{0_\nn}\right)+\int_{\Gamma_{\mathrm{N}^{u_0,g_0}_{\mathrm{S}}}}g'_{0}\frac{\sigma_{\tau}\left(u_0\right)}{g_{0}}\cdot \left(w_\tau-u'_{0_\tau}\right)\\+\int_{\Gamma_{\mathrm{N}^{u_0,g_0}_{\mathrm{R}}}}\left(-g'_{0}\frac{u_{0_\tau}}{\left\|u_{0_\tau}\right\|}-\frac{g_0}{\left\|u_{0_\tau}\right\|}\left( u'_{0_\tau}-\left(u'_{0_\tau}\cdot \frac{u_{0_\tau}}{\left\|u_{0_\tau}\right\|}\right)\frac{u_{0_\tau}}{\left\|u_{0_\tau}\right\|} \right)\right)\cdot \left(w_\tau-u'_{0_\tau}\right),
\end{multline*}
for all $w\in\mathcal{K}_{u_{0},\frac{\sigma_{\tau}(F_{0}-u_{0})}{g_{0}}}$.  From Definition~\ref{weakformuletangentSigno} one deduces that $u'_{0}$ is the unique weak solution to the tangential Signorini problem~\eqref{caractu0DNT} which concludes the proof.
\end{proof}

\begin{myRem}\normalfont
Consider the framework of Theorem~\ref{caractu0derivDNT}. Note that $u'_{0}$ is the unique weak solution to the tangential Signorini problem~\eqref{caractu0DNT}, but is not necessarily a strong solution. Nevertheless, in the case where $\mathrm{A}\mathrm{e}(u'_0)\nn\in\LL^{2}(\Gamma_{\mathrm{N}},\R^d)$ and the decomposition $\Gamma_{\mathrm{D}}\cup\Gamma_{\mathrm{N}^{u_0,g_0}_{\mathrm{T}}}\cup\Gamma_{\mathrm{N}^{u_0,g_0}_{\mathrm{R}}}\cup\Gamma_{\mathrm{N}^{u_0,g_0}_{\mathrm{S}}}$ of $\Gamma$ is consistent (see Definition~\ref{regulieresens2}), then~$u'_{0}$ is a strong solution to the tangential Signorini problem~\eqref{caractu0DNT}.
\end{myRem}

\begin{myRem}\label{remarextensigno}\normalfont
In this paper a prescribed normal stress $\sigma_{\nn}(u_t)=h_t$ on $\Gamma_{\mathrm{N}}$ has been considered in the parameterized Tresca friction problem~\eqref{PbNeumannDirichletTrescaPara}. Nevertheless, as mentioned in Remark~\ref{RemCondBordNeumann1}, it is possible to consider a contact problem by constraining the normal displacement, that is, by replacing $\sigma_{\nn}(u_t)=h_t$ by ${u_t}_{\nn}=0$ on $\Gamma_{\mathrm{N}}$. This would lead in Theorem~\ref{caractu0derivDNT} to a derivative $u'_0$ satisfying ${u'_0}_\nn=0$ on~$\Gamma_{\mathrm{N}}$, instead of~$\sigma_{\nn}({u'_0})=h'_0$. One can also consider the Signorini's unilateral conditions given by ${u_t}_{\nn}\leq 0,~\sigma_{\nn}({u_t})\leq 0 \text{ and } {u_t}_{\nn} \sigma_{\nn}({u_t}) =0 $ on $\Gamma_{\mathrm{N}}$. In that case, for all~$t\geq0$, the solution~$u_t$ is given by $u_t:=\mathrm{prox}_{\iota_{\mathcal{K}}+\Phi(t,\cdot)}(F_t)$, where $F_t$ is the solution to the parameterized Dirichlet-Neumann problem~\eqref{PbNeumannDirichletPara} with $h_t=0$, and~$\iota_{\mathcal{K}}$ is the indicator function associated with the closed convex subset~$\mathcal{K}$ of $\HH^{1}_{\mathrm{D}}(\Omega,\R^{d})$ given by $\mathcal{K} := \left\{v\in\HH^{1}_{\mathrm{D}}(\Omega,\R^{d}) \mid v_{\nn}\leq0 \text{ \textit{a.e.} on }\Gamma_{\mathrm{N}} \right \}$. 
To develop our strategy in that context, one should investigate the twice epi-differentiablity of~$\iota_{\mathcal{K}}$. This nontrivial part is done in the submitted paper~\cite{jdc} where our methodology has been applied to a contact problem with the Signorini's unilateral conditions in a shape optimization context.
\end{myRem}

\section{Application to optimal control}\label{section4cc}

Consider the functional framework introduced at the beginning of Section~\ref{Mainresult1}. Let $f\in \mathrm{L}^{2}(\Omega,\R^d)$, $h\in \mathrm{L}^{2}(\Gamma_{\mathrm{N}})$, $g_1\in\LL^{\infty}(\Gamma_{\mathrm{N}})$ such that $g_1\geq m$ \textit{a.e.} on $\Gamma_{\mathrm{N}}$ for some positive constant $m>0$ and~$g_2\in\LL^{\infty}(\Gamma_{\mathrm{N}})$ such that~$||g_2||_{\LL^{\infty}(\Gamma_{\mathrm{N}})}>0$. In this section we consider the optimal control problem given by
\begin{equation}\label{problemenergyfcontr}
    \minn\limits_{ \substack{ z\in \mathcal{U}}} \; \mathcal{J}(z),
\end{equation}
where $\mathcal{J}$ is the cost functional defined by
\begin{equation*}
\fonction{\mathcal{J}}{\mathrm{V}}{\R}{z}{\mathcal{J}(z):=\frac{1}{2}\left\|u(\ell(z))\right\|^{2}_{\HH^{1}_{\mathrm{D}}(\Omega,\R^d)}+\frac{\beta}{2}\left\| \ell(z) \right\|^{2}_{\LL^{2}(\Gamma_{\mathrm{N}})},}
\end{equation*}
where~$\mathrm{V}$ is the open subset of~$\LL^{\infty}(\Gamma_{\mathrm{N}})$ defined by
$$
\mathrm{V}:=\left\{ z\in\LL^{\infty}(\Gamma_{\mathrm{N}}) \mid \exists C(z)>0\text{, } \ell(z)>C(z) \text{ \textit{\textit{a.e.}} on } \Gamma_{\mathrm{N}} \right\},
$$
where $\ell$ is the map defined by $z\in\LL^{\infty}(\Gamma_{\mathrm{N}})\mapsto \ell(z):=g_1+zg_2\in \LL^{\infty}(\Gamma_{\mathrm{N}})$, and where $u(\ell(z))\in\HH^{1}_{\mathrm{D}}(\Omega,\R^d)$ stands for the unique solution to the Tresca friction problem given by
\begin{equation}\label{controlefortrescaopti}\tag{CTP\ensuremath{_{\ell(z)}}}
\arraycolsep=2pt
\left\{
\begin{array}{rcll}
-\mathrm{div}(\mathrm{A}\mathrm{e}(u)) & = & f   & \text{ in } \Omega , \\
u & = & 0  & \text{ on } \Gamma_{\mathrm{D}} ,\\
\sigma_\nn(u) & = & h  & \text{ on } \Gamma_{\mathrm{N}},\\
\left\|\sigma_\tau(u)\right\|\leq \ell(z) \text{ and } u_{\tau}\cdot\sigma_{\tau}(u)+\ell(z)\left\|u_{\tau}\right\| & = & 0  & \text{ on } \Gamma_{\mathrm{N}},
\end{array}
\right.
\end{equation}
where~$\beta > 0$ is a positive constant and where~$\mathcal{U}$ is a given nonempty convex subset of $\mathrm{V}$ such that~$\mathcal{U}$ is a bounded closed subset of~$\LL^{2}(\Gamma_{\mathrm{N}})$. Note that the first term in the cost functional $\mathcal{J}$ corresponds to the compliance, while the second term is the energy consumption which is standard in optimal control problems (see, e.g.,~\cite{MANZA}).

This section is organized as follows. In Subsection~\ref{existencemaispasunitforelatvv} we prove the existence of a solution to Problem~\eqref{problemenergyfcontr}. In Subsection~\ref{gradofthecost} we prove, under some assumptions, that $\mathcal{J}$ is Gateaux differentiable on $\mathrm{V}$ and we characterize its gradient. Finally, in Subsection~\ref{numericalsimjfkjfsdkjfsdkf}, numerical simulations are performed to solve Problem~\eqref{problemenergyfcontr} on a two-dimensional example.

\subsection{Existence of a solution}\label{existencemaispasunitforelatvv}

This section is dedicated to the following existence result.

\begin{myProp}
There exists $z^{*}\in\mathcal{U}$ such that $\mathcal{J}(z^{*})\leq \mathcal{J}(z)$ for all $z\in\mathcal{U}$.
\end{myProp}
\begin{proof}
In this proof the strong (resp.\ weak) convergence in Hilbert spaces is denoted by~$\rightarrow$ (resp.~\ $\rightharpoonup$) and all limits with respect to the index~$i$ will be considered for~$i \to +\infty$. Since~$0\leq \mathcal{J}(z)<+\infty$ for all $z\in\mathcal{U}$, we get that $\inf_{z\in\mathcal{U}} \mathcal{J}(z)\in\R_{+}$. Considering a minimizing sequence~$(z_{i})_{i\in\N}$, there exists $N\in\N$ such that~$\mathcal{J}(z_{i})\leq 1+ \inf_{z\in\mathcal{U}} \mathcal{J}(z)$ for all $i\geq N$, that is
$$
\frac{1}{2}\left\|u(\ell(z_i))\right\|^{2}_{\HH^{1}_{\mathrm{D}}(\Omega,\R^d)}+\frac{\beta}{2}\left\| \ell(z_i) \right\|^{2}_{\LL^{2}(\Gamma_{\mathrm{N}})}\leq 1+ \underset{z\in\mathcal{U}}{\inf}\mathcal{J}(z),
$$
for all $i\geq N$. Thus the sequence $(\ell(z_i))_{i\in\N}$ is bounded in $\LL^{2}(\Gamma_{\mathrm{N}})$ and thus, up to a subsequence that we do not relabel, weakly converges to some $g^{*}\in\LL^{2}(\Gamma_{\mathrm{N}})$. Moreover, since~$\mathcal{U}$ is a bounded closed convex subset of $\LL^{2}(\Gamma_{\mathrm{N}})$ (and thus weakly closed in $\LL^{2}(\Gamma_{\mathrm{N}})$), we know that, up to a subsequence that we do not relabel, the sequence $(z_{i})_{i\in\N}$ weakly converges to some~$z^{*}\in\mathcal{U}$. Moreover one has
$$
\left| \int_{\Gamma_{\mathrm{N}}}\left( \ell(z_i)-g_1-z^{*}g_2\right)w \right|=\int_{\Gamma_{\mathrm{N}}}\left(z_{i}-z^{*}\right)g_2w,
$$
for all~$w\in\LL^2(\Gamma_{\mathrm{N}})$, and, since $g_2\in\LL^{\infty}(\Gamma_{\mathrm{N}})$, it holds that $g_2 w\in\LL^{2}(\Gamma_{\mathrm{N}})$ and one deduces that~$\ell(z_i) \rightharpoonup g_1+z^*g_2$ in $\LL^{2}(\Gamma_{\mathrm{N}})$ and thus $g^*=g_1+z^*g_2$.
In a similar way, up to a subsequence that we do not relabel, the sequence $(u(\ell(z_i)))_{i\in\N}$ weakly converges in $\HH^{1}_{\mathrm{D}}(\Omega,\R^d)$ to some $u^{*}\in\HH^{1}_{\mathrm{D}}(\Omega,\R^d)$, thus $u(\ell(z_i))\rightarrow u^{*}$ in~$\LL^{2}(\Gamma,\R^d)$ from the compact embedding~$\HH^{1}_{\mathrm{D}}(\Omega,\R^d)\hookdoubleheadrightarrow~\mathrm{L}^{2}(\Gamma,\R^d)$ (see Proposition~\ref{injections}). Let us prove that $u(\ell(z_i))\rightarrow u^{*}$ in $\HH^{1}_{\mathrm{D}}(\Omega,\R^d)$. It holds that
$$\left\|u^{*}-u(\ell(z_i))\right\|^{2}_{\HH^{1}_{\mathrm{D}}(\Omega,\R^d)}=\dual{u^{*}}{u^{*}-u(\ell(z_i))}_{\HH^{1}_{\mathrm{D}}(\Omega,\R^d)}-\dual{u(\ell(z_i))}{u^{*}-u(\ell(z_i))}_{\HH^{1}_{\mathrm{D}}(\Omega,\R^d)},
$$
for all $i\in\N$. Using the weak formulation satisfied by $u(\ell(z_i))$, we get that
\begin{multline*}
    \left\|u^{*}-u(\ell(z_i))\right\|^{2}_{\HH^{1}_{\mathrm{D}}(\Omega,\R^d)}\leq \dual{u^{*}}{u^{*}-u(\ell(z_i))}_{\HH^{1}_{\mathrm{D}}(\Omega,\R^d)}-\int_{\Omega}f\cdot(u^{*}-u(\ell(z_i)))\\-\int_{\Gamma_{\mathrm{N}}}h\left(u^{*}_\nn-u(\ell(z_i))_\nn\right)+\int_{\Gamma_{\mathrm{N}}}\ell(z_i)\left(\left\|u^{*}_{\tau}\right\|-\left\|u(\ell(z_i))_\tau\right\|\right)\\\leq \dual{u^{*}}{u^{*}-u(\ell(z_i))}_{\HH^{1}_{\mathrm{D}}(\Omega,\R^d)}-\int_{\Omega}f\cdot(u^{*}-u(\ell(z_i)))-\int_{\Gamma_{\mathrm{N}}}h\left(u^{*}_\nn-u(\ell(z_i))_\nn\right)\\+C\left\|u^{*}-u(\ell(z_i))\right\|_{\LL^{2}(\Gamma,\R^d)}\longrightarrow0,
\end{multline*}
where $C \geq 0$ is a constant (depending only on $\Omega$ and on $\max_{i\in\N}||\ell(z_i)||_{\LL^2(\Gamma_{\mathrm{N}})}
$).
Now let us prove that $u^{*}=u(g_1+z^{*}g_2)$. For $w\in\HH^{1}_{\mathrm{D}}(\Omega,\R^d)$ fixed, it holds that 
\begin{multline}\label{equationsecondtrescabisk96}
    \dual{u(\ell(z_i))}{w-u(\ell(z_i))}_{\HH^{1}_{\mathrm{D}}(\Omega,\R^d)}+\int_{\Gamma_{\mathrm{N}}}\ell(z_i)\left\|w_\tau\right\|-\int_{\Gamma_{\mathrm{N}}}\ell(z_i)\left\|u(\ell(z_i))_\tau\right\|\\\geq \int_{\Omega}f\cdot(w-u(\ell(z_i)))+\int_{\Gamma_{\mathrm{N}}}h\left(w_\nn-u(\ell(z_i))_\nn\right),
\end{multline}
for all $i\in\N$. Note that:\\
\begin{enumerate}[label={\rm (\roman*)}] \small{
    \item $\displaystyle\left|\dual{u(\ell(z_i))}{w-u(\ell(z_i))}_{\HH^{1}_{\mathrm{D}}(\Omega,\R^d)}-\dual{u^{*}}{w-u^{*}}_{\HH^{1}_{\mathrm{D}}(\Omega,\R^d)}\right|\leq D\left\|u^{*}-u(\ell(z_i))\right\|_{\HH^{1}_{\mathrm{D}}(\Omega,\R^d)}\longrightarrow~0$;
    \item $\displaystyle\left|\int_{\Omega}f\cdot(w-u(\ell(z_i)))-\int_{\Omega}f\cdot(w-u^{*})\right|\leq D\left\|f\right\|_{\LL^{2}(\Omega,\R^d)}\left\|u^{*}-u(\ell(z_i))\right\|_{\HH^{1}_{\mathrm{D}}(\Omega,\R^d)}\longrightarrow 0$;
     \item $\displaystyle\left|\int_{\Gamma_{\mathrm{N}}}h\left(w_\nn-u(\ell(z_i))_\nn\right)-\int_{\Gamma_{\mathrm{N}}}h\left(w_\nn-u^{*}_\nn\right)\right|\leq D\left\|h\right\|_{\LL^{2}(\Gamma_{\mathrm{N}})}\left\|u^{*}-u(\ell(z_i))\right\|_{\LL^{2}(\Gamma,\R^d)}\longrightarrow 0$;
    \item  $\displaystyle \left| \int_{\Gamma_{\mathrm{N}}}\ell(z_i)\left(\left\|w_\tau\right\|-\left\|u(\ell(z_i))_\tau\right\|\right)-\int_{\Gamma_{\mathrm{N}}}g^{*}\left(\left\|w_\tau\right\|-\left\|u^{*}_{\tau}\right\|\right)\right|\leq\\\left|\int_{\Gamma_{\mathrm{N}}}(\ell(z_i)-g^{*})\left\|w_\tau\right\|\right|+\left|\int_{\Gamma_{\mathrm{N}}}(\ell(z_i)-g^{*})\left\|u^{*}_{\tau}\right\|\right|+D\left\|u^{*}-u(\ell(z_i))\right\|_{\LL^{2}(\Gamma,\R^d)}\longrightarrow 0;$}\\
\end{enumerate}
where $D\geq0$ is a constant (depending only on $\Omega$, $\mathrm{A}$ and $w$). Therefore it follows in~\eqref{equationsecondtrescabisk96} when~$i\to +\infty$ that
$$
 \dual{u^{*}}{w-u^{*}}_{\HH^{1}_{\mathrm{D}}(\Omega,\R^d)}+\int_{\Gamma_{\mathrm{N}}}g^{*}\left\|w_\tau\right\|-\int_{\Gamma_{\mathrm{N}}}g^{*}\left\|u^{*}_{\tau}\right\|\geq\int_{\Omega}f\cdot(w-u^{*})+\int_{\Gamma_{\mathrm{N}}}h\left(w_\nn-u^{*}_\nn\right).
$$
Since this inequality is true for all $w\in\HH^{1}_{\mathrm{D}}(\Omega,\R^d)$ and $g^*=g_1+z^{*}g_2$, one deduces that~$u^{*}=u(g_1+z^{*}g_2)$, and then
\begin{multline*}
    \mathcal{J}(z^{*})=\frac{1}{2}\left\|u(g_1+z^{*}g_2)\right\|^{2}_{\HH^{1}_{\mathrm{D}}(\Omega,\R^d)}+\frac{\beta}{2}\left\| g_1+z^{*}g_2 \right\|^{2}_{\LL^{2}(\Gamma_{\mathrm{N}})}\leq\\    \liminf\limits_{i\to +\infty}\left(\frac{1}{2}\left\|u(\ell(z_i))\right\|^{2}_{\HH^{1}_{\mathrm{D}}(\Omega,\R^d)}+\frac{\beta}{2}\left\| \ell(z_i) \right\|^{2}_{\LL^{2}(\Gamma_{\mathrm{N}})}\right)\leq \liminf\limits_{i\to +\infty} \mathcal{J}(z_{i})=\underset{z\in\mathcal{U}}{\inf}\mathcal{J}(z),
\end{multline*}
which concludes the proof.
\end{proof}

\begin{myRem}\normalfont
Since the solution to the Tresca friction problem is not linear with respect to the friction term, note that $\mathcal{J}$ is not a strictly convex functional (and thus the uniqueness of the solution to Problem~\eqref{problemenergyfcontr} is not guaranteed).
\end{myRem}

\subsection{Gateaux differentiability of the cost functional}\label{gradofthecost}
Consider the auxiliary functional
\begin{equation*}
    \fonction{\mathrm{J}}{\HH^{1}_{\mathrm{D}}(\Omega,\R^d)\times\LL^{\infty}(\Gamma_{\mathrm{N}})}{\R}{(v,g)}{\mathrm{J}(v,g):=\frac{1}{2}\left\|v\right\|^{2}_{\HH^{1}_{\mathrm{D}}(\Omega,\R^d)}+\frac{\beta}{2}\left\| g \right\|^{2}_{\LL^{2}(\Gamma_{\mathrm{N}}).}}
\end{equation*}
One can easily prove that $\mathrm{J}$ is Fréchet differentiable on $\HH^{1}_{\mathrm{D}}(\Omega,\R^d)\times\LL^{\infty}(\Gamma_{\mathrm{N}})$ and its Fréchet differential at some $(v,g)\in\HH^{1}_{\mathrm{D}}(\Omega,\R^d)\times\LL^{\infty}(\Gamma_{\mathrm{N}})$, denoted by $\mathrm{d}\mathrm{J}(v,g)$, is given by
$$
\mathrm{d}\mathrm{J}(v,g)(\tilde{v},\tilde{g})=\dual{v}{\tilde{v}}_{\HH^{1}_{\mathrm{D}}(\Omega,\R^d)}+\beta\dual{g}{\tilde{g}}_{\LL^{2}(\Gamma_{\mathrm{N}})},
$$
for all $(\tilde{v},\tilde{g})\in\HH^{1}_{\mathrm{D}}(\Omega,\R^d)\times\LL^{\infty}(\Gamma_{\mathrm{N}}).$
Now let us introduce the map 
\begin{equation*}
    \fonction{\mathcal{F}}{\mathrm{V}}{\HH^{1}_{\mathrm{D}}(\Omega,\R^d)\times \LL^{\infty}(\Gamma_{\mathrm{N}}) }{z}{\mathcal{F}(z):=\left(u(\ell(z)),\ell(z)\right),}
\end{equation*}
where $u(\ell(z))\in\HH^{1}_{\mathrm{D}}(\Omega,\R^d)$ is the unique solution to the Tresca friction problem~\eqref{controlefortrescaopti}. Hence the cost functional $\mathcal{J}$ is given by the composition $\mathcal{J}=\mathrm{J}\circ\mathcal{F}$. 

\begin{myTheorem}\label{gradientdelafonccoutgg}
Let $z_0\in\mathrm{V}$ be fixed and let us denote by $u_0:=u(\ell(z_0))$. Assume that:
\begin{enumerate}[label={\rm (\roman*)}]
        \item\label{bxwww45} the map $s\in\Gamma_{\mathrm{N}^{u_0,\ell(z_0)}_{\mathrm{R}}}\mapsto \frac{\ell(z_0)(s)}{\left\|{u_0}_{\tau}(s)\right\|}\in\mathbb{R^{*}_{+}}$ belongs to $\LL^{4}(\Gamma_{\mathrm{N}^{u_0,\ell(z_0)}_{\mathrm{R}}})$ (see below for definition of the set $\Gamma_{\mathrm{N}^{u_0,\ell(z_0)}_{\mathrm{R}}}$);
        \item\label{bxwww456} the parameterized Tresca friction functional $\Phi$ defined in~\eqref{fonctionnelledeTrescaparacas2} is twice epi-differentiable at~$u_{0}$ for $F-u_{0}\in\partial \Phi(0,\cdot)(u_{0})$, with
\begin{equation*}
\displaystyle\mathrm{D}_{e}^{2}\Phi(u_{0})|F-u_{0})(w)=\int_{\Gamma_{\mathrm{N}}}\mathrm{D}_{e}^{2}G(s)(u_{0}(s)|\sigma_{\tau}(F-u_{0})(s))(w(s))\, \mathrm{d}s, \qquad \forall w\in \HH^{1}_{\mathrm{D}}(\Omega,\R^d),
\end{equation*} 
where, for almost all $s\in\Gamma_{\mathrm{N}}$, the map $G(s)$ is defined in Proposition~\ref{epidiffoffunctionG}, and $F\in\HH^{1}_{\mathrm{D}}(\Omega,\R^d)$ is the unique solution to the Dirichlet-Neumann problem
\begin{equation}\label{FFFFFF111}
\arraycolsep=2pt
\left\{
\begin{array}{rcll}
-\mathrm{div}(\mathrm{A}\mathrm{e}(F)) & = & f   & \text{ in } \Omega , \\
F & = & 0  & \text{ on } \Gamma_{\mathrm{D}} ,\\
\mathrm{A}\mathrm{e}(F)\nn & = & h\nn  & \text{ on } \Gamma_{\mathrm{N}}.
\end{array}
\right.
\end{equation}
\end{enumerate}
Then the cost functional $\mathcal{J}$ is Gateaux differentiable at~$z_{0}$ and its differential~$\mathrm{d}_{G}\mathcal{J}(z_{0})$ is given by
$$
\mathrm{d}_{G}\mathcal{J}(z_{0})(z)=\int_{\Gamma_{\mathrm{N}^{u_0,\ell(z_0)}_{\mathrm{R}}}}zg_2\left(\beta\left( g_1+z_0 g_2\right)-\left\|u_{0_\tau}\right\|\right)+\int_{\Gamma_{\mathrm{N}^{u_0,\ell(z_0)}_{\mathrm{T}}}\cup\Gamma_{\mathrm{N}^{u_0,\ell(z_0)}_{\mathrm{S}}}}\beta zg_2\left(g_1+z_0g_2\right),
$$
for all $z\in\LL^{\infty}(\Gamma_{\mathrm{N}})$, where $\Gamma_{\mathrm{N}}$ is decomposed (up to a null set) as $\Gamma_{\mathrm{N}^{u_0,\ell(z_0)}_{\mathrm{T}}}\cup\Gamma_{\mathrm{N}^{u_0,\ell(z_0)}_{\mathrm{R}}}\cup\Gamma_{\mathrm{N}^{u_0,\ell(z_0)}_{\mathrm{S}}}$ with
$$
\begin{array}{l}
\Gamma_{\mathrm{N}^{u_0,\ell(z_0)}_{\mathrm{R}}}:=\left\{s\in\Gamma_{\mathrm{N}} \mid  u_{0_\tau}(s)\neq0\right \}, \\
\Gamma_{\mathrm{N}^{u_0,\ell(z_0)}_{\mathrm{T}}}:=\left\{s\in\Gamma_{\mathrm{N}} \mid  u_{0_\tau}(s)=0 \text{ and } \frac{\sigma_{\tau}(u_0)(s)}{\ell(z_0)(s)}\in\mathrm{B}(0,1)\cap\left(\R\nn(s)\right)^{\perp}\right\}, \\
\Gamma_{\mathrm{N}^{u_0,\ell(z_0)}_{\mathrm{S}}}:=\left\{s\in\Gamma_{\mathrm{N}} \mid  u_{0_\tau}(s)=0 \text{ and } \frac{\sigma_{\tau}(u_0)(s)}{\ell(z_0)(s)}\in\partial{\mathrm{B}(0,1)}\cap\left(\R\nn(s)\right)^{\perp}\right\}.
\end{array}
$$
\end{myTheorem}
\begin{proof}
Let $z\in\LL^{\infty}(\Gamma_{\mathrm{N}})$ and $t>0$ be sufficiently small such that $z_{t}:=z_{0}+tz\in \mathrm{V}$. We denote by~$u_t:=u(\ell(z_t))\in\HH^{1}_{\mathrm{D}}(\Omega,\R^d)$. From Subsection~\ref{subtresca},~$u_t\in\HH^{1}_{\mathrm{D}}(\Omega,\R^d)$ is given by~$u_{t}=\mathrm{prox}_{\Phi(t,\cdot)}(F)$, where $\Phi$ is the parameterized Tresca friction functional defined in~\eqref{fonctionnelledeTrescaparacas2} and~$F$ is the unique solution to the Dirichlet-Neumann problem~\eqref{FFFFFF111}. From Hypotheses~\ref{bxwww45},~\ref{bxwww456} and since the map~$t\in\mathbb{R}_{+}\mapsto \ell(z_t)\in \LL^{\infty}(\Gamma_{\mathrm{N}})$ is differentiable at~$t=0$, with its derivative given by~$\ell'(z_0):=zg_2$, one can apply Theorem~\ref{caractu0derivDNT} to deduce that the map~$t\in\mathbb{R}_{+}\mapsto u_{t}\in\HH^{1}_{\mathrm{D}}(\Omega,\R^d)$ is differentiable at~$t=0$ and its derivative, denoted by~$u'_{0}\in\mathcal{K}_{u_{0},\frac{\sigma_{\tau}(F_{0}-u_{0})}{\ell(z_0)}}\subset\HH^{1}_{\mathrm{D}}(\Omega,\R^d)$, is the unique solution to the variational inequality (which is the weak formulation of a tangential Signorini problem) given by
\begin{multline*}
    \dual{u'_0}{w-u'_0}_{\HH^{1}_{\mathrm{D}}(\Omega,\R^d)}\geq\int_{\Gamma_{\mathrm{N}^{u_0,\ell(z_0)}_{\mathrm{S}}}}\ell'(z_0)\frac{\sigma_{\tau}\left(u_0\right)}{\ell(z_0)}\cdot \left(w_\tau-u'_{0_{\tau}}\right)\\+\int_{\Gamma_{\mathrm{N}^{u_0,\ell(z_0)}_{\mathrm{R}}}}\left(-\ell'(z_0)\frac{u_{{0_\tau}}}{\left\|u_{{0_\tau}}\right\|}-\frac{\ell(z_0)}{\left\|u_{{0_\tau}}\right\|}\left( u'_{0_{\tau}}-\left(u'_{0_\tau}\cdot \frac{u_{{0_\tau}}}{\left\|u_{{0_\tau}}\right\|}\right)\frac{u_{{0_\tau}}}{\left\|u_{{0_\tau}}\right\|} \right)\right)\cdot \left(w_\tau-u'_{0_{\tau}}\right),
\end{multline*}
for all $w\in\mathcal{K}_{u_{0},\frac{\sigma_{\tau}(F_{0}-u_{0})}{\ell(z_0)}}$, where
\begin{multline*}
    \mathcal{K}_{u_{0},\frac{\sigma_{\tau}(F_{0}-u_{0})}{\ell(z_0)}}:=\biggl\{ w\in \HH^{1}_{\mathrm{D}}(\Omega,\R^d)\mid w_\tau=0 \text{ \textit{\textit{a.e.}} on } \Gamma_{\mathrm{N}^{u_0,\ell(z_0)}_{\mathrm{T}}} \\ \text{ and } w_\tau\in\R_{-}\frac{\sigma_{\tau}(u_0)}{\ell(z_0)}\text{ \textit{\textit{a.e.}} on } \Gamma_{\mathrm{N}^{u_0,\ell(z_0)}_{\mathrm{S}}} \biggl\}.
\end{multline*}
Since~$\mathcal{J}=\mathrm{J}\circ \mathcal{F}$, with $\mathrm{J}$ Fréchet differentiable on $\HH^{1}_{\mathrm{D}}(\Omega,\R^d)\times \mathrm{V}$, and
$$ \frac{\left\|\mathcal{F}(z_{0}+tz)-\mathcal{F}(z_{0})-t\left(u'_{0},\ell'(z_0)\right)\right\|_{\HH^{1}_{\mathrm{D}}(\Omega,\R^d)\times\LL^{\infty}(\Gamma_{\mathrm{N}})}}{t}=\frac{\left\|u_{t}-u_{0}-tu'_{0}\right\|_{\HH^{1}_{\mathrm{D}}(\Omega,\R^d)}}{t}\longrightarrow 0,
$$
when $t\rightarrow0^+$, we deduce that $\mathcal{J}$ has a right derivative at $z_{0}$ in the direction $z$ given by
$$
\mathcal{J}'(z_{0})(z)=\dual{u'_{0}}{{u_{0}}}_{\HH^{1}_{\mathrm{D}}(\Omega,\R^d)}+\beta\dual{\ell(z_0)}{\ell'(z_0)}_{\LL^{2}(\Gamma_{\mathrm{N}})}.
$$
Furthermore, since $u'_{0}\pm u_{0}\in\mathcal{K}_{u_{0},\frac{\sigma_{\tau}(F-u_{0})}{\ell(z_0)}}$, one deduces that
$$
    \dual{u'_{0}}{u_0}_{\HH^{1}_{\mathrm{D}}(\Omega,\R^d)}=\int_{\Gamma_{\mathrm{N}^{u_0,\ell(z_0)}_{\mathrm{R}}}}\left(-\ell'(z_0)\frac{u_{{0_\tau}}}{\left\|u_{{0_\tau}}\right\|}-\frac{\ell(z_0)}{\left\|u_{{0_\tau}}\right\|}\left( {u'_{0}}_\tau-\left({u'_{0}}_\tau\cdot \frac{u_{{0_\tau}}}{\left\|u_{{0_\tau}}\right\|}\right)\frac{u_{{0_\tau}}}{\left\|u_{{0_\tau}}\right\|} \right)\right)\cdot u_{{0_{\tau}}}.
$$
Since
$$
\int_{\Gamma_{\mathrm{N}^{u_0,\ell(z_0)}_{\mathrm{R}}}}\frac{\ell(z_0)}{\left\|u_{{0_\tau}}\right\|}\left( {u'_{0}}_\tau-\left({u'_{0}}_\tau\cdot \frac{u_{{0_\tau}}}{\left\|u_{{0_\tau}}\right\|}\right)\frac{u_{{0_\tau}}}{\left\|u_{{0_\tau}}\right\|} \right)\cdot u_{{0_{\tau}}}=0,
$$
we get that
$$
\dual{u'_{0}}{u_0}_{\HH^{1}_{\mathrm{D}}(\Omega,\R^d)}=-\int_{\Gamma_{\mathrm{N}^{u_0,\ell(z_0)}_{\mathrm{R}}}}\ell'(z_0)\left\|u_{{0_\tau}}\right\|,
$$
and we can rewrite the right derivative of $\mathcal{J}$ at $z_{0}$ in the direction $z$ as
\begin{multline*}
    \mathcal{J}'(z_{0})(z)=-\int_{\Gamma_{\mathrm{N}^{u_0,\ell(z_0)}_{\mathrm{R}}}}\ell'(z_0)\left\|u_{{0_\tau}}\right\|+\int_{\Gamma_{\mathrm{N}}}\beta \ell'(z_0)\ell(z_0) \\ = \int_{\Gamma_{\mathrm{N}^{u_0,\ell(z_0)}_{\mathrm{R}}}}\ell'(z_0)\left(\beta \ell(z_0)-\left\|u_{{0_\tau}}\right\|\right)+\int_{\Gamma_{\mathrm{N}^{u_0,\ell(z_0)}_{\mathrm{T}}}\cup\Gamma_{\mathrm{N}^{u_0,\ell(z_0)}_{\mathrm{S}}}}\beta \ell'(z_0)\ell(z_0),
\end{multline*}
and thus
$$
\mathcal{J}'(z_{0})(z)=\int_{\Gamma_{\mathrm{N}^{u_0,\ell(z_0)}_{\mathrm{R}}}}zg_2\left(\beta \left(g_1+z_0g_2\right)-\left\|u_{{0_\tau}}\right\|\right)+\int_{\Gamma_{\mathrm{N}^{u_0,\ell(z_0)}_{\mathrm{T}}}\cup\Gamma_{\mathrm{N}^{u_0,\ell(z_0)}_{\mathrm{S}}}}\beta zg_2\left(g_1+z_0g_2\right).
$$
Note that $\mathcal{J}'(z_{0})$ is linear and continuous on $\LL^{\infty}(\Gamma_{\mathrm{N}})$. Thus $\mathcal{J}$ is Gateaux differentiable at $z_{0}$ with its Gateaux differential given by $\mathrm{d}_{G}\mathcal{J}(z_{0}):=\mathcal{J}'(z_{0})$. The proof is complete.
\end{proof}

\begin{myRem}\normalfont
    In the proof of Theorem~\ref{gradientdelafonccoutgg}, note that the derivative~$u'_0$ depends on the pair $(\ell(z_0),\ell'(z_0)) = (g_1+z_0 g_2,zg_2)$ and thus on the term~$z \in\LL^{\infty}(\Gamma_{\mathrm{N}})$. Therefore let us denote by~$u'_0 := u'_0(z)$. Note that~$u'_0(z)$ is not linear with respect to~$z$. However one can observe that the scalar product~$ \dual{u'_{0}(z)}{u_0}_{\HH^{1}_{\mathrm{D}}(\Omega,\R^d)}$, that appears in the proof of Theorem~\ref{gradientdelafonccoutgg}, is linear with respect to~$z$. Therefore it leads to an expression of $\mathcal{J}'(z_{0})$ that is linear with respect to~$z$, and thus to the Gateaux differentiability of $\mathcal{J}$ at $z_{0}$.
\end{myRem}

\subsection{Numerical simulations}\label{numericalsimjfkjfsdkjfsdkf}
Let us assume that~$||g_2||_{\LL^{\infty}(\Gamma_{\mathrm{N}})}<m$, where~$m>0$ is the constant introduced at the beginning of Section~\ref{section4cc}, and consider the admissible set~$ \mathcal{U}$ given by
 $$
 \mathcal{U}:=\left\{ z\in\LL^{2}(\Gamma_{\mathrm{N}}) \mid -1\leq z\leq 1 \text{ \textit{\textit{a.e.}} on } \Gamma_{\mathrm{N}} \right\},
$$
which is a nonempty convex subset of $\mathrm{V}$ and is a bounded closed subset of $\LL^{2}(\Gamma_{\mathrm{N}})$. In this subsection our aim is to numerically solve an example of Problem~\eqref{problemenergyfcontr} in the two-dimensional case $d=2$, by making use of our theoretical result obtained in Theorem~\ref{gradientdelafonccoutgg}. 

\subsubsection{Numerical methodology}\label{methodnum22XXw4545}
Starting with an initial control $z_0\in\mathcal{U}$, we compute $z_d\in\LL^{\infty}(\Gamma_{\mathrm{N}})$ given by
\begin{equation*}
z_d:=
\left\{
\begin{array}{lcll}
-g_2\left(\beta\left(g_1+z_0 g_2\right)-\left\|{u_0}_\tau\right\|\right)	&   & \text{ on } \Gamma_{\mathrm{N}^{u_0,\ell(z_0)}_{\mathrm{R}}}, \\
-\beta g_2 \left(g_1+z_0 g_2\right)	&   & \text{ on } \Gamma_{\mathrm{N}^{u_0,\ell(z_0)}_{\mathrm{T}}}\cup\Gamma_{\mathrm{N}^{u_0,\ell(z_0)}_{\mathrm{S}}},
\end{array}
\right.
\end{equation*}
which is, from Theorem~\ref{gradientdelafonccoutgg}, a descent direction of the functional~$\mathcal{J}$ at $z_0$ since it satisfies
\begin{multline*}
    \mathrm{d}_{G}\mathcal{J}(z_0)(z_d)=-||g_2\left(\beta\left(g_1+z_0 g_2\right)-\left\|{u_0}_\tau\right\|\right)||^2_{\LL^2(\Gamma_{\mathrm{N}^{u_0,\ell(z_0)}_{\mathrm{R}}})}\\-||\beta g_2 \left(g_1+z_0 g_2\right) ||^2_{\LL^2(\Gamma_{\mathrm{N}^{u_0,\ell(z_0)}_{\mathrm{T}}}\cup\Gamma_{\mathrm{N}^{u_0,\ell(z_0)}_{\mathrm{S}}})}\leq0.
\end{multline*}
Then the control is updated as $z_1=\mathrm{proj}_{\mathcal{U}}\left(z_0+\eta z_d\right)$, where $\eta>0$ is a fixed parameter and $\mathrm{proj}_{\mathcal{U}}$ is the classical projection operator onto $\mathcal{U}$ considered in $\LL^2(\Gamma_{\mathrm{N}})$. Then the algorithm restarts with~$z_1$, and so on.

 Let us mention that the numerical simulations have been performed using Freefem++ software~\cite{HECHT} with P1-finite elements and standard affine mesh. The Tresca friction problem is numerically solved using an adaptation of iterative switching algorithms (this adaptation is close to the one described in~\cite[Appendix C]{4ABC} which concerns a scalar Tresca friction problem). We also precise that, for all~$i\in~\mathbb{N}^{*}$, the difference between the cost functional $\mathcal{J}$ at the iteration~$20\times i$ and at the iteration~$20\times (i-1)$ is computed. The smallness of this difference is used as a stopping criterion for the algorithm.

\begin{myRem}\normalfont
In this paper, to numerically solve the Tresca friction problem, we used an iterative switching algorithm since it is an easily implementable method. Nevertheless there exist many different algorithms in the literature to numerically solve the Tresca friction problem: Nitsche's method (see, e.g.,~\cite{CHOULYTRESCA,CHOULY55}), mixed methods (see, e.g.,~\cite{HAS}), etc. These algorithms could be more efficient and investigations to compare them with the iterative switching algorithm need to be carried out in order to study their advantages and drawbacks. However this interesting perspective for further research works is beyond the scope of the present paper.
\end{myRem}

\subsubsection{Example and numerical results}\label{exempldelasimupourcontroltres}
In this subsection take~$d=2$ and let~$\Omega$ be the unit disk of $\R^{2}$ with its boundary $\Gamma:=\partial\Omega$ decomposed as $\Gamma=\Gamma_{\mathrm{D}}\cup\Gamma_{\mathrm{N}}$ (see Figure~\ref{figurebcxhfhgfq}), where
$$
\begin{array}{l}
\Gamma_{\mathrm{D}}:=\left\{(\cos\theta,\sin\theta)\in\Gamma \mid 0\leq\theta\leq\frac{\pi}{2}\right\}, \\
\Gamma_{\mathrm{N}}:=\left\{(\cos\theta,\sin\theta)\in\Gamma \mid \frac{\pi}{2}<\theta<2\pi\right\}.
\end{array}
$$
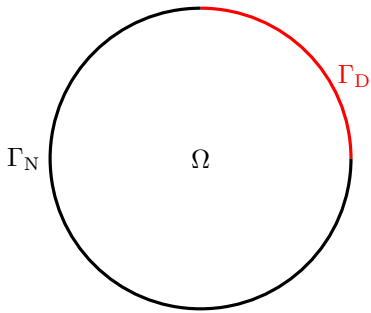
\begin{figure}[ht]
    \centering
\begin{tikzpicture}
\draw (0,0) node{$\Omega$};
\draw [color=red, very thick] (2,0) arc (0:90:2);
\draw [color=black, very thick](0,2) arc(90:360:2);
\draw (1.7,1.1) [color=red] node[right]{$\Gamma_{\mathrm{D}}$};
\draw (-2,0) [color=black] node[left]{$\Gamma_{\mathrm{N}}$};
\end{tikzpicture}
    \caption{Unit disk $\Omega$ and its boundary $\Gamma=\Gamma_{\mathrm{D}}\cup\Gamma_{\mathrm{N}}$.}\label{figurebcxhfhgfq}
\end{figure}

We assume that $\Omega$ is \textit{isotropic}, in the sense that the Cauchy stress tensor is given by
 $$
 \sigma(w)=2\mu\mathrm{e}(w)+\lambda \mathrm{tr}\left(\mathrm{e}(w)\right)\mathrm{I},
 $$
 for all~$w\in\HH^{1}_{\mathrm{D}}(\Omega,\R^d)$, where $\mathrm{tr}\left(\mathrm{e}(w)\right)$ is the trace of the matrix $\mathrm{e}(w)$ and where $\mu\geq0$ and~$\lambda\geq0$ are Lamé parameters (see, e.g.,~\cite{SALEN}). In what follows we take $\mu=0.3846$ and~$\lambda=0.5769$. This corresponds to a Young's modulus equal to $1$ and to a Poisson's ratio equal to $0.3$, which is a typical value for a large variety of materials. Let us consider the arbitrary functions $h:=0$ \textit{\textit{a.e.}} on $\Gamma_{\mathrm{N}}$, $g_1:=2$ \textit{\textit{a.e.}} on $\Gamma_{\mathrm{N}}$, $g_2\in\LL^2(\Gamma_{\mathrm{N}})$ be the function defined by
$$
\fonction{g_2}{\Gamma_{\mathrm{N}}}{\R}{(x,y)}{\displaystyle g_2(x,y):= x^2-y^2,}
$$
and $f\in\LL^{2}(\Omega,\R^2)$ be the function defined by 
$$
\fonction{f}{\Omega}{\R^2}{(x,y)}{\displaystyle f(x,y):= \begin{pmatrix}
\frac{5-x^{2}-y^{2}+xy}{4} & \frac{5-x^{2}-y^{2}+xy}{4}
\end{pmatrix}.}
$$
With $m:=2$, one has $g_1\geq m$ \textit{\textit{a.e.}} on $\Gamma_{\mathrm{N}}$ and $0<||g_2||_{\LL^{\infty}(\Gamma_{\mathrm{N}})}< m$, thus the assumptions from the beginning of Section~\ref{section4cc} and from Subsection~\ref{numericalsimjfkjfsdkjfsdkf} are satisfied. We consider the initial control~$z_0\in\mathcal{U}$ given by
$$
\fonction{z_0}{\Gamma_{\mathrm{N}}}{\R}{(x,y)}{\displaystyle z_0(x,y):= \cos{(x^2-y^2)}.}
$$

We present now the numerical results obtained for the above two-dimensional example using the numerical methodology described in Subsection~\ref{methodnum22XXw4545}. Figure~\ref{problemenergyfcofdfsntr} depicts the control which solves Problem~\eqref{problemenergyfcontr}. It is a \textit{bang–bang} optimal control, that takes exclusively the two values $-1$ and $1$ on the boundary $\Gamma_{\mathrm{N}}$. Figure~\ref{figure2contr} shows the evolution of the value of $\mathcal{J}$ with respect to the iteration. We observe an usual decreasing of the cost functional $\mathcal{J}$ with respect to the iteration. 

\begin{figure}[h!]
    \centering
    \includegraphics[scale=0.5]{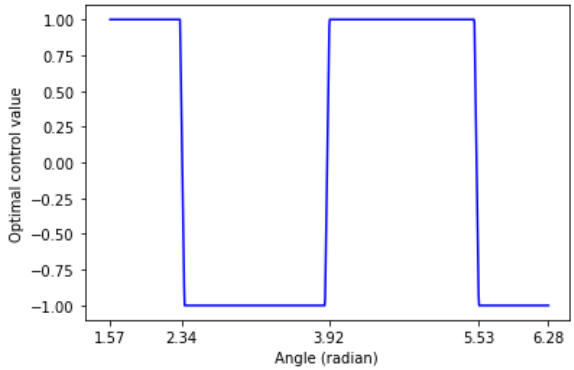}
    \caption{Values of the optimal control on the boundary $\Gamma_{\mathrm{N}}:=\left\{(\cos\theta,\sin\theta)\in\Gamma \mid \frac{\pi}{2}<\theta<2\pi\right\}$.}\label{problemenergyfcofdfsntr}
\end{figure}
\begin{figure}[h!]
    \centering
    \includegraphics[scale=0.6]{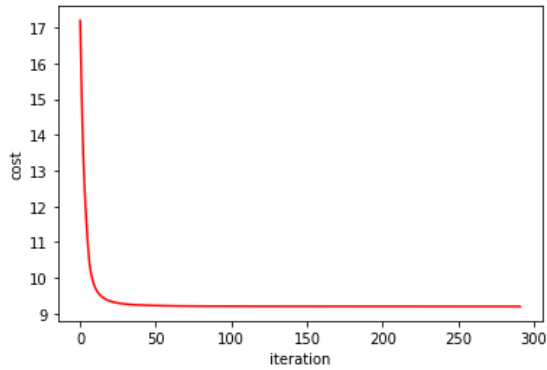}
    \caption{Values of the cost functional $\mathcal{J}$ with respect to the iterations.}\label{figure2contr}
\end{figure}

\newpage

\appendix

\section{Notions from convex, variational and functional analyses}\label{appendix}

\subsection{Reminders from convex and variational analyses}\label{rappel}

For notions and results presented in this section, we refer to standard references such as~\cite{BREZ2,MINTY,ROCK2} and~\cite[Chapter~12]{ROCK}. In what follows $(\mathcal{H}, \dual{\cdot}{\cdot}_{\mathcal{H}})$ stands for a general real Hilbert space.

\begin{myDefn}[Domain and epigraph]\label{epidom}
Let  $\phi \,: \, \mathcal{H}\rightarrow \mathbb{R}\cup\left\{\pm \infty \right\}$.
The \emph{domain} and the \emph{epigraph} of~$\phi$ are respectively defined by
$$
\mathrm{dom}\left(\phi\right):=\left\{x\in \mathcal{H} \mid \phi(x)<+\infty \right\} \quad \text{and} \quad
\mathrm{epi}\left(\phi\right):=\left\{(x,\nu)\in \mathcal{H}\times\mathbb{R}\mid \phi(x)\leq \nu \right\}.
$$
\end{myDefn}
Recall that $\phi \,: \, \mathcal{H}\rightarrow \mathbb{R}\cup\left\{\pm \infty \right\}$ is said to be \textit{proper} if $\mathrm{dom}(\phi)\neq \emptyset$ and $\phi(x)>-\infty$ 
for all~$x\in\mathcal{H}$. Moreover, $\phi$ is a convex (resp.\ lower semi-continuous) function on $\mathcal{H}$ if and only if~$\mathrm{epi}(\phi)$ is a convex (resp.\ closed) subset of~$\mathcal{H}\times\R$.

\begin{myDefn}[Support function]\label{supportfunc}
Let $\mathrm{C}$ be a nonempty closed convex subset of $\mathcal{H}$. The support function $\xi_{\mathrm{C}}$ of $\mathrm{C}$ is the map defined by
$$
\fonction{\xi_{\mathrm{C}}}{\mathcal{H}}{\mathbb{R}\cup\left\{+\infty\right\}}{x}{\displaystyle \xi_{\mathrm{C}}(x):=\underset{y\in\mathrm{C}}{\sup}\dual{x}{y}_{\mathcal{H}}.}
$$
\end{myDefn}

\begin{myDefn}[Convex subdifferential operator]\label{sousdiff}
 Let $\phi  :  \mathcal{H} \rightarrow \R\cup\left\{+\infty\right\}$ be a proper function. We denote by $\partial{\phi}  :  \mathcal{H} \rightrightarrows \mathcal{H}$ the \emph{convex subdifferential operator} of $\phi$, defined by 
 $$
 \partial{\phi}(x):=\left\{y\in\mathcal{H} \mid \forall z\in\mathcal{H}\text{, } \dual{y}{z-x}_{\mathcal{H}}\leq \phi(z)-\phi(x)\right\},
 $$
for all $x\in\mathcal{H}$.
\end{myDefn} 

\begin{myDefn}[Proximal operator]\label{proxi}
 Let $\phi  :  \mathcal{H} \rightarrow \R\cup\left\{+\infty\right\}$ be a proper lower semi-continuous convex function. The \emph{proximal operator} associated with $\phi$ is the map~$\mathrm{prox}_{\phi}  :  \mathcal{H} \rightarrow \mathcal{H}$ defined by
 $$
      \mathrm{prox}_{\phi}(x):=\underset{y\in \mathcal{H}}{\argmin}\left[ \phi(y)+\frac{1}{2}\left \| y-x \right \|^{2}_{\mathcal{H}}\right]=(\mathrm{I}+\partial \phi)^{-1}(x),
 $$
for all $x\in\mathcal{H}$, where $\mathrm{I}  :  \mathcal{H}\rightarrow \mathcal{H}$ stands for the identity operator.
\end{myDefn}
The proximal operator have been introduced by J.-J. Moreau in 1965 (see~\cite{MOR}) and can be seen as a generalization of the classical projection operators onto nonempty closed convex subsets. It is well-known that, if $\phi  :  \mathcal{H} \rightarrow \R\cup\left\{+\infty\right\}$ is a proper lower semi-continuous convex function, then~$\partial{\phi}$ is a maximal monotone operator (see, e.g.,~\cite{ROCK2}), and thus the proximal operator~$\mathrm{prox}_{\phi}$ is well-defined and a single-valued map (see, e.g.,~\cite[Chapter II]{BREZ2}).

We pursue with the following classical result which is crucial to prove the existence of a unique weak solution to the tangential Signorini problem (see Proposition~\ref{existenceSignorinitang}).

\begin{myProp}\label{gradientequi}
Let $\phi \,: \, \mathcal{H}\rightarrow \mathbb{R}$ be a Fréchet differentiable convex function and~$\mathrm{C}$ be a nonempty convex subset of $\mathcal{H}$. Let $y\in\mathrm{C}$ and $x\in\mathcal{H}$. Then the following variational inequalities are equivalent:
\begin{enumerate}[label={\rm (\roman*)}]
\item $\displaystyle \varphi(z)-\varphi(y)\geq\dual{x-y}{z-y}_{\mathcal{H}},\qquad \forall z\in\mathrm{C};$
\item $\dual{\nabla{\varphi(y)}}{z-y}_{\mathcal{H}}\geq\dual{x-y}{z-y}_{\mathcal{H}}, \qquad\forall z\in\mathrm{C}.$
\end{enumerate}
\end{myProp}

In what follows, some definitions related to the notion of twice epi-differentiability are recalled (for more details, see~\cite[Chapter 7, section B p.240]{ROCK} for the finite-dimensional case and~\cite{DO} for the infinite-dimensional one). The strong (resp.\ weak) convergence of a sequence in~$\mathcal{H}$ will be denoted by~$\rightarrow$ (resp.\ $\rightharpoonup$) and note that all limits with respect to~$t$ will be considered for~$t \to 0^+$.
\begin{myDefn}[Mosco-convergence]\label{limitemuch}
The \emph{outer}, \emph{weak-outer}, \emph{inner} and \emph{weak-inner limits} of a parameterized family~$(A_{t})_{t>0}$ of subsets of $\mathcal{H}$ are respectively defined by
\begin{eqnarray*}
      \mathrm{lim}\sup A_{t}&:=&\left\{ x\in \mathcal{H} \mid \exists (t_{n})_{n\in\mathbb{N}}\rightarrow 0^{+}, \exists \left(x_{n}\right)_{n\in\mathbb{N}}\rightarrow x, \forall n\in\mathbb{N}, x_{n}\in A_{t_{n}}\right\},\\
     \mathrm{w}\text{-}\mathrm{lim}\sup A_{t}&:=&\left\{ x\in \mathcal{H} \mid \exists (t_{n})_{n\in\mathbb{N}}\rightarrow 0^{+}, \exists \left(x_{n}\right)_{n\in\mathbb{N}}\rightharpoonup x, \forall n\in\mathbb{N}, x_{n}\in A_{t_{n}}\right\},\\
     \mathrm{lim}\inf A_{t}&:=&\left\{ x\in \mathcal{H} \mid \forall (t_{n})_{n\in\mathbb{N}}\rightarrow 0^{+}, \exists \left(x_{n}\right)_{n\in\mathbb{N}}\rightarrow x, \exists N\in\mathbb{N}, \forall n\geq N, x_{n}\in A_{t_{n}}\right\},\\
     \mathrm{w}\text{-}\mathrm{lim}\inf A_{t}&:=&\left\{ x\in \mathcal{H} \mid \forall (t_{n})_{n\in\mathbb{N}}\rightarrow 0^{+}, \exists \left(x_{n}\right)_{n\in\mathbb{N}}\rightharpoonup x, \exists N\in\mathbb{N}, \forall n\geq N, x_{n}\in A_{t_{n}}\right\}.
\end{eqnarray*}
The family~$(A_{t})_{t>0}$ is said to be \emph{Mosco-convergent} if~$
\mathrm{w}\text{-}\mathrm{lim}\sup A_{t}\subset\mathrm{lim}\inf A_{t}
$. In that case all the previous limits are equal and we write
$$
     \mathrm{M}\text{-}\mathrm{lim} \,  A_{t}:=\mathrm{lim}\inf A_{t}=\mathrm{lim}\sup A_{t}=\mathrm{w}\text{-}\mathrm{lim}\inf A_{t}=\mathrm{w}\text{-}\mathrm{lim}\sup A_{t}.
$$
\end{myDefn}
\begin{myDefn}[Mosco epi-convergence]
  Let $(\phi_{t})_{t>0}$ be a parameterized family of functions $\phi_{t}  : \mathcal{H}\rightarrow \mathbb{R}\cup\left\{\pm \infty \right\}$ for all $t>0$.
 We say that $(\phi_{t})_{t>0}$ is \emph{Mosco epi-convergent} if~$(\mathrm{epi}(\phi_{t}))_{t>0}$ is Mosco-convergent in~$\mathcal{H} \times \R$. Then we denote by $\mathrm{ME}\text{-}\mathrm{lim}~ \phi_{t}  :  \mathcal{H}\rightarrow \mathbb{R}\cup\left\{\pm \infty \right\}$ the function characterized by its epigraph~$\mathrm{epi}\left(\mathrm{ME}\text{-}\mathrm{lim}~\phi_{t}\right):=\mathrm{M}\text{-}\mathrm{lim}$ $\displaystyle \mathrm{epi}\left(\phi_{t}\right)$ and we say that $(\phi_{t})_{t>0}$ Mosco epi-converges to~$\mathrm{ME}\text{-}\mathrm{lim}~\phi_{t}$.
 \end{myDefn}
 
The proof of the next proposition can be found in~\cite[Proposition 3.19 p.297]{ATTOUCH}.
\begin{myProp}[Characterization of Mosco epi-convergence]\label{caractMosco}
Let $(\phi_{t})_{t>0}$ be a parameterized family of functions~$\phi_{t}  :  \mathcal{H}\rightarrow \mathbb{R}\cup\left\{\pm \infty \right\}$ for all $t>0$ and let~$\phi  :  \mathcal{H}\rightarrow \mathbb{R}\cup\left\{\pm \infty \right\}$. Then $(\phi_{t})_{t>0}$ Mosco epi-converges to $\phi$ if and only if, for all $x\in \mathcal{H}$, the two conditions:
\begin{enumerate}[label={\rm (\roman*)}]
    \item there exists $(x_{t})_{t>0}\rightarrow x$ such that $\mathrm{lim}\sup \phi_{t}(x_{t})\leq \phi(x)$;
    \item for all $(x_{t})_{t>0}\rightharpoonup  x$, $\mathrm{lim}\inf \phi_{t}(x_{t})\geq \phi(x)$;
\end{enumerate}
are both satisfied.
\end{myProp}

Now let us recall the notion of twice epi-differentiability introduced by R.T.~Rockafellar in~1985 (see~\cite{Rockafellar}) that generalizes the classical notion of second-order derivative to nonsmooth convex functions.

 \begin{myDefn}[Twice epi-differentiability]\label{epidiff}
   A proper lower semi-continuous convex function~$\phi  : \mathcal{H}\rightarrow \mathbb{R}\cup\left\{+\infty \right\}$ is said to be \emph{twice epi-differentiable} at $x\in\mathrm{dom}(\phi)$ for $y\in\partial\phi(x)$ if the family of second-order difference quotient functions $(\delta_{t}^{2}\phi(x \mid y))_{t>0}$ defined by
$$
  \fonction{ \delta_{t}^{2}\phi(x\mid y) }{\mathcal{H}}{\mathbb{R}\cup\left\{+\infty\right\}}{z}{\displaystyle\frac{\phi(x+t z)-\phi(x)-t\dual{ y}{z}_{\mathcal{H}}}{t^{2}},}
$$
for all $t>0$, is Mosco epi-convergent. In that case we denote by
$$
\mathrm{d}_{e}^{2}\phi(x \mid y):=\mathrm{ME}\text{-}\mathrm{lim}~\delta_{t}^{2}\phi(x \mid y),
$$
which is called the \emph{second-order epi-derivative} of $\phi$ at $x$ for $y$.
\end{myDefn}

\begin{myRem}\normalfont\label{diffsecond}
In the case where $\phi$ is twice Fréchet differentiable at $x\in\mathcal{H}$, then $\phi$ is twice epi-differentiable at $x$ for $\nabla{\phi}(x)$ and
$$
\mathrm{d}_{e}^{2}\phi(x\mid\nabla{\phi}(x))(z)=\frac{1}{2}\mathrm{D}^{2}\phi(x)(z,z),\qquad\forall z\in\mathcal{H},
$$
where $\mathrm{D}^2\phi(x)$ stands for the second-order Fréchet differential of $\phi$ at $x$. Note that the factor~$\frac{1}{2}$ could be removed if the family of second-order difference quotient functions is defined
with a factor~$\frac{1}{2}$ in the denominator (see the original definition in~\cite{Rockafellar}). 
\end{myRem}

In the above classical definition of twice epi-differentiability, the function $\phi$ does not depend on the parameter~$t$. However, in this paper, the parameterized Tresca friction functional does. Therefore we use an extended version of twice epi-differentiability which has been recently introduced in~\cite{8AB}. To this aim, when considering a function $\Phi  :  \mathbb{R}_{+}\times \mathcal{H}\rightarrow \mathbb{R}\cup\left\{+\infty\right\}$ such that, for all~$t\geq0$, $\Phi(t,\cdot)$ is a proper function on $\mathcal{H}$, we will make use of the two following notations: $\partial \Phi(0,\mathord{\cdot} )(x)$ stands for the convex subdifferential operator at~$x\in\mathcal{H}$ of the map~$w\in\mathcal{H} \mapsto \Phi(0,w)\in \R\cup\left\{+\infty\right\}$, and~$\Phi^{-1}(\mathord{\cdot} , \mathbb{R}):=\left\{ x\in\mathcal{H}\mid \forall t\geq0, \; \Phi(t,x)\in\R \right\}$.

 \begin{myDefn}[Twice epi-differentiability depending on a parameter]\label{epidiffpara}
Let~$\Phi  :  \mathbb{R}_{+}\times \mathcal{H}\rightarrow \mathbb{R} \cup\left\{+\infty\right\}$ be a function such that, for all $t\geq0$, $\Phi(t,\cdot)$ is a proper lower semi-continuous convex function on $\mathcal{H}$. The function $\Phi$ is said to be \emph{twice epi-differentiable} at $x\in \Phi^{-1}(\mathord{\cdot} , \mathbb{R})$ for~$y\in\partial \Phi(0,\mathord{\cdot} )(x)$ if the family of second-order difference quotient functions $(\Delta_{t}^{2}\Phi(x \mid y))_{t>0}$ defined by
$$
  \fonction{\Delta_{t}^{2}\Phi(x \mid y) }{\mathcal{H}}{\mathbb{R}\cup\left\{+\infty\right\}}{z}{\displaystyle\frac{\Phi(t,x+t z)-\Phi(t,x)-t\dual{ y}{z}_{\mathcal{H}}}{t^{2}},}
$$
for all $t>0$, is Mosco epi-convergent. In that case, we denote by
$$
\mathrm{D}_{e}^{2}\Phi(x \mid y):=\mathrm{ME}\text{-}\mathrm{lim}~\Delta_{t}^{2}\Phi(x \mid y) ,
$$
which is called the \emph{second-order epi-derivative} of $\Phi$ at $x$ for $y$.
\end{myDefn}
Note that, if the function $\Phi$ is $t$-independent in Definition~\ref{epidiffpara}, then we recover Definition~\ref{epidiff}. Finally the following theorem is the key point in order to derive our main result in this paper. It is a particular case of a more general theorem that can be found in~\cite[Theorem 4.15 p.1714]{8AB}.
\begin{myTheorem}\label{TheoABC2018}
Let~$\Phi  :  \mathbb{R}_{+}\times \mathcal{H}\rightarrow \mathbb{R} \cup \left\{+\infty\right\}$ be a function such that, for all $t\geq0$, $\Phi(t,\cdot)$ is a proper lower semi-continuous convex function on $\mathcal{H}$. Let $F  :  \mathbb{R}_{+}\rightarrow \mathcal{H}$ and let~$u  :  \mathbb{R}_{+}\rightarrow \mathcal{H}$ be defined by
$$
    u(t):=\mathrm{prox}_{\Phi(t,\mathord{\cdot} )}(F(t)),
$$
for all~$t\geq 0$. If the conditions: 
\begin{enumerate}[label={\rm (\roman*)}]
    \item $F$ is differentiable at $t=0$;
    \item $\Phi$ is twice epi-differentiable at $u(0)$ for $F(0)-u(0)\in\partial \Phi(0,\mathord{\cdot} )(u(0))$;
    \item $\mathrm{D}_{e}^{2}\Phi(u(0)|F(0)-u(0))$ is a proper function on $\mathcal{H}$;
\end{enumerate}
are satisfied, then $u$ is differentiable at $t=0$ with
$$
u'(0)=\mathrm{prox}_{\mathrm{D}_{e}^{2}\Phi(u(0)|F(0)-u(0))}(F'(0)).
$$
\end{myTheorem}

\subsection{Reminders from functional analysis}\label{rappelfunct}
Let $d\in\N^*$ be a positive integer, $\Omega$ be a nonempty bounded connected open subset of $\R^{d}$ with a~$\mathcal{C}^{1}$-boundary~$\Gamma:=\partial{\Omega}$ and $\nn$ be the outward-pointing unit normal vector to $\Gamma$. In what follows we consider a decomposition~$\Gamma=:\Gamma_{1}\cup\Gamma_{2}$ where~$\Gamma_{1}$ and $\Gamma_{2}$ are two measurable disjoint subsets of $\Gamma$. Let us recall some embeddings useful in this work, that can be found for instance in~\cite[Chapter~4, p.79]{ADAMS},~\cite{MarkusB},~\cite{BREZ}, and~\cite[Chapter~7, Section~2 p.395]{DAUTLIONS}.

\begin{myProp}\label{injections}
The continuous and dense embeddings:
\begin{itemize}
    \item $\HH^{1}(\Omega,\R^d){\hookrightarrow} \HH^{1/2}(\Gamma,\R^d){\hookrightarrow} \LL^{2}(\Gamma,\R^d){\hookrightarrow} \HH^{-1/2}(\Gamma,\R^d)$;
    \item $\LL^{2}(\Gamma,\R^d){\hookrightarrow} \LL^{1}(\Gamma,\R^d)$;
    \item $\HH^{1}(\Omega,\R^d){\hookrightarrow} \mathrm{L}^{2}(\Omega,\R^d)$;
    \item $\HH^{1/2}_{00}(\Gamma_{1},\R^d){\hookrightarrow} \LL^{2}(\Gamma_{1},\R^d){\hookrightarrow} \HH^{-1/2}_{00}(\Gamma_{1},\R^d)$;
\end{itemize}  are satisfied, where $\HH^{1/2}_{00}(\Gamma_{1},\R^d)$ can be identified to a linear subspace of $\HH^{1/2}(\Gamma,\R^d)$ defined by
$$
\HH^{1/2}_{00}(\Gamma_{1},\R^d):=\left\{w\in\LL^{2}(\Gamma_{1},\R^d) \mid \exists v\in\HH^{1}(\Omega,\R^d), \; v=w \text{ \textit{a.e.} on }\Gamma_{1} \text{ and } v=0 \text{ \textit{a.e.} on }\Gamma_{2} \right\},
$$
and $\HH^{-1/2}_{00}(\Gamma_{1},\R^d)$ stands for its dual space.
Furthermore the dense and compact embedding
$$
\HH^{1}(\Omega,\R^d)\hookdoubleheadrightarrow~\mathrm{L}^{2}(\Gamma,\R^d),
$$ 
holds true, and since $d\in\{2,3\}$, then we have the continuous embedding  $\HH^{1}(\Omega,\R^d){\hookrightarrow} \LL^{4}(\Gamma,\R^d)$.
\end{myProp}
  
The next proposition is a particular case of a more general statement that can be found in~\cite[Section 2.9 p.56]{TUC}.
\begin{myProp}\label{Ident}
Let $v\in\HH^{-1/2}_{00}(\Gamma_{1},\R^d)$. If there exists $C\geq 0$ such that
$$
\dual{v}{w}_{\HH^{-1/2}_{00}(\Gamma_{1},\R^d)\times \HH^{1/2}_{00}(\Gamma_{1},\R^d)}\leq C\left \| w \right \|_{\LL^{2}(\Gamma_{1},\R^d)},
$$
for all $w\in \HH^{1/2}_{00}(\Gamma_{1},\R^d)$, then $v$ can be identified to an element $h\in \LL^{2}(\Gamma_{1},\R^d)$ with $$\dual{v}{w}_{\HH^{-1/2}_{00}(\Gamma_{1},\R^d)\times \HH^{1/2}_{00}(\Gamma_{1},\R^d)}=\dual{h}{w}_{\LL^{2}(\Gamma_{1},\R^d)},
$$
for all $w\in\HH^{1/2}_{00}(\Gamma_{1},\R^d)$.
\end{myProp}

The next proposition, known as divergence formula, can be found in~\cite[Theorem 4.4.7 p.104]{ALLNUM}.

\begin{myTheorem}[Divergence formula]\label{div}
Let $v\in \HH_{\mathrm{div}}(\Omega, \R^{d\times d})$ where
$$
\HH_{\mathrm{div}}(\Omega, \R^{d\times d}):= \left\{ w\in \mathrm{L}^{2}(\Omega,\R^{d\times d})  \mid \mathrm{div} (w)\in \mathrm{L}^{2}(\Omega,\R^d) \right\},
$$
and $\mathrm{div}(w)$ is the vector whose the $i$-th component is defined by $\mathrm{div}(w)_i:=\mathrm{div}(w_{i})\in\LL^2(\Omega,\R)$, and where $w_i\in\mathrm{L}^{2}(\Omega,\R^{d})$ is the~$i$-th line of~$w$, for all $i\in[[1,d]]$ and for all~$w\in\HH_{\mathrm{div}}(\Omega, \R^{d\times d})$.
Then $v$ admits a normal trace, denoted by~$v\nn \in \HH^{-1/2}(\Gamma,\R^d)$, satisfying
$$
\displaystyle\int_{\Omega}\mathrm{div}(v)\cdot w+\int_{\Omega}v:\nabla w=\dual{v\nn}{w}_{\HH^{-1/2}(\Gamma,\R^d)\times \HH^{1/2}(\Gamma,\R^d)}, \qquad\forall w \in \HH^1(\Omega,\R^{d}).
$$
\end{myTheorem}

\section*{Statements and Declarations}
\paragraph{Funding.}
The authors declare that no funds, grants, or other support were received during the preparation of this manuscript.
\paragraph{Competing Interests.}
The authors have no relevant financial or non-financial interests to disclose.

\bibliographystyle{abbrv}
\bibliography{biblio}

\end{document}